\pdfoutput=1
\documentclass[a4paper,11pt]{amsart}
\usepackage[hmarginratio={1:1},vmarginratio={1:1},lmargin=80.0pt,tmargin=90.0pt]{geometry}

\usepackage{calligra}
\usepackage[T1]{fontenc}
\usepackage[nice]{nicefrac}

\usepackage{mathtools}
\mathtoolsset{mathic}
\usepackage[final]{microtype}
\usepackage{multicol}
\usepackage{tabularx}
\usepackage{latexsym,exscale,enumitem,amsfonts,amssymb}
\usepackage{amsmath,amsthm,amscd,textcomp,bbm}
\usepackage{stmaryrd}
\SetSymbolFont{stmry}{bold}{U}{stmry}{m}{n}
\usepackage{mathrsfs}
\usepackage{float,subcaption}
\usepackage[normalem]{ulem}
\usepackage{tabu}
\usepackage{booktabs}
\usepackage{etex}
\usepackage{graphicx}
\usepackage{fancybox}
\usepackage[table]{xcolor}
\usepackage{thmtools}

\usepackage{caption} 
\usepackage{tikz}
\tikzset{anchorbase/.style={baseline={([yshift=-0.5ex]current bounding box.center)}},tinynodes/.style={font=\tiny,text height=0.75ex,text depth=0.15ex},
  cross line/.style={preaction={draw=white,line width=4.75pt,-}},
  cross line thick/.style={preaction={draw=white,line width=6pt,-}},
  cross line thin/.style={preaction={draw=white,line width=3.5pt,-}},
  colored/.style={line width=2.25},
  uncolored/.style={thin},
  ccolored/.style={line width=1.4},
  dotbullet/.style={fill,circle,inner sep=0.5pt},
}
\usetikzlibrary{decorations.markings}
\usetikzlibrary{decorations.pathreplacing}
\usetikzlibrary{arrows,shapes,positioning,matrix,calc}
\tikzstyle directed=[postaction={decorate,decoration={markings,
    mark=at position #1 with {\arrow{>}}}}]
\tikzstyle rdirected=[postaction={decorate,decoration={markings,
    mark=at position #1 with {\arrow{<}}}}]

\newcommand{\overunder}[3]{\overset{#1}{\underset{#2}{#3}}}

\makeatletter
\renewcommand{\boxed}[1]{\text{\fboxsep=.1em\Ovalbox{\m@th$\displaystyle#1$}}}
\makeatother

\newcommand{\abs}[1]{\left|#1\right|}

\DeclareRobustCommand{\qpar}{q}
\newcommand{\zpar}{{\mathsf z}}

\newcommand{\C}{{\mathbb C}}
\newcommand{\CQ}{{\mathbb C}(\qpar)}
\newcommand{\CQZ}{{\mathbb C}(\qpar)[\zpar^{\pm 1}]}
\newcommand{\N}{{\mathbb{Z}_{\geq 0}}}
\newcommand{\Np}{{\mathbb{Z}_{> 0}}}
\newcommand{\Z}{{\mathbb Z}}

\newcommand{\ring}{{\mathbb K}}

\newcommand{\qRepco}[1]{\boldsymbol{\mathcal{R}\mathrm{ep}}_{\qpar}^{\prime}(#1)}

\newcommand{\fraksl}{\mathfrak{sl}}

\newcommand{\frakosp}{\mathfrak{osp}}

\DeclareMathOperator{\Hom}{Hom}
\DeclareMathOperator{\End}{End}

\newcommand{\id}{\mathrm{id}}

\newcommand{\calC}{{\mathcal{C}}}

\newcommand{\calM}{{\mathcal{M}}}

\newcommand{\one}{\mathbbm{1}}

\DeclareRobustCommand{\dia}{\texorpdfstring{\Gamma}{Gamma}}
\DeclareRobustCommand{\howe}{\texorpdfstring{\Phi}{Phi}}
\DeclareRobustCommand{\lad}{\texorpdfstring{\Upsilon}{Upsilon}}

\newcommand{\ladA}{\Upsilon_{\!\frakgl}}
\newcommand{\ladC}{\Upsilon_{\!\fraksp}}
\newcommand{\ladD}{\Upsilon_{\!\frakso}}

\newcommand{\ext}{{\color{myred}\mathrm{ext}}\!\!\!\phantom{y}}
\newcommand{\sym}{{\color{mygreen}\mathrm{sym}}\!\!\!\phantom{t}}

\newcommand{\diaA}{\Gamma_{\typeA}^{\ext}}
\newcommand{\diaC}{\Gamma_{\typeC}^{\ext}}
\newcommand{\diaD}{\Gamma_{\typeB\typeD}^{\ext}}

\newcommand{\howeA}{\Phi_{\typeA}^{\ext}}
\newcommand{\howeC}{\Phi_{\typeC}^{\ext}}
\newcommand{\howeD}{\Phi_{\typeB\typeD}^{\ext}}

\newcommand{\symdiaA}{\Gamma_{\typeA}^{\sym}}
\newcommand{\symdiaC}{\Gamma_{\typeC}^{\sym}}
\newcommand{\symdiaD}{\Gamma_{\typeB\typeD}^{\sym}}

\newcommand{\symhoweA}{\Phi_{\typeA}^{\sym}}
\newcommand{\symhoweC}{\Phi_{\typeC}^{\sym}}
\newcommand{\symhoweD}{\Phi_{\typeB\typeD}^{\sym}}

\DeclareRobustCommand{\typeA}{\texorpdfstring{\mathbf{A}}{A}}
\DeclareRobustCommand{\typeB}{\texorpdfstring{\mathbf{B}}{B}}
\DeclareRobustCommand{\typeC}{\texorpdfstring{\mathbf{C}}{C}}
\DeclareRobustCommand{\typeD}{\texorpdfstring{\mathbf{D}}{D}}

\newcommand{\frakg}{\mathfrak{g}}
\newcommand{\frakh}{\mathfrak{h}}
\newcommand{\frakso}{\mathfrak{so}}
\newcommand{\fraksp}{\mathfrak{sp}}
\newcommand{\frakgl}{\mathfrak{gl}}

\newcommand{\U}{\mathbf{U}}
\newcommand{\quantumg}{\mathbf{U}_{\qpar}}
\newcommand{\coideal}{\mathbf{U}_{\qpar}^{\prime}}
\newcommand{\Udot}{\dot{\mathbf{U}}_{\qpar}}

\let\epsilon\varepsilon
\newcommand{\quantumq}{\mathbf{q}}
\newcommand{\counit}{\epsilon}

\newcommand{\Aalg}{{\mathscr{A}}}
\newcommand{\coidealA}{\prescript{\Aalg}{}{\coideal}}

\newcommand{\quantumgA}{\prescript{\Aalg}{}{\quantumg}}
\newcommand{\repA}[1]{\prescript{\Aalg}{}{#1}}

\newcommand{\trivmod}{\C_{\qpar}^{\phantom{1}}}

\newcommand{\vecrep}{\mathrm{V}_{\qpar}^{\phantom{1}}}
\newcommand{\vecrepnoq}{\mathrm{V}}

\newcommand{\littleexterior}{\bigwedge}

\newcommand{\exterior}[1]{{\textstyle\littleexterior_{\qpar}^{#1}}}
\newcommand{\exteriors}[1]{\bigwedge_{\qpar}^{#1}}
\newcommand{\exteriornoq}[1]{{\textstyle\littleexterior^{#1}}}

\newcommand{\symmetric}[1]{{\textstyle\mathrm{Sym}_{\qpar}^{#1}}}

\newcommand{\symmetricnoq}[1]{{\textstyle\mathrm{Sym}^{#1}}}

\newcommand{\acts}{\mathop{\,\;\raisebox{1.7ex}{\rotatebox{-90}{$\circlearrowright$}}\;\,}}
\newcommand{\actsreverse}{\mathop{\,\;\raisebox{0ex}{\rotatebox{90}{$\circlearrowleft$}}\;\,}}

\newcommand{\Repq}[1]{\boldsymbol{\mathcal{R}\mathrm{ep}}_{\qpar}(\frakgl_{#1})}
\newcommand{\RepcoD}[1]{\boldsymbol{\mathcal{R}\mathrm{ep}}^{\prime}_{\qpar}(\frakso_{#1})}
\newcommand{\RepcoC}[1]{\boldsymbol{\mathcal{R}\mathrm{ep}}^{\prime}_{\qpar}(\fraksp_{#1})}

\newcommand{\qbracket}[1]{\left[#1\right]}
\newcommand{\qbracketC}[1]{\left[#1\right]_2}
\newcommand{\qbin}[2]{{\textstyle\genfrac{[}{]}{0pt}{}{#1}{#2}}}
\newcommand{\qbinn}[2]{\genfrac{[}{]}{0pt}{}{#1}{#2}}

\newcommand{\simpleroots}{\Pi}
\newcommand{\roots}{\Phi}
\newcommand{\weightlattice}{X}
\newcommand{\I}{{\mathrm I}}

\newcommand{\simple}[2]{{\mathrm L}_{\qpar}(#1,#2)}
\newcommand{\simplenoq}[2]{{\mathrm L}(#1,#2)}
\newcommand{\simpleco}[2]{{\mathrm L}^{\prime}_{\qpar}(#1,#2)}

\newcommand{\partition}{\mathfrak{P}}

\newcommand{\heckecat}{\smash{\boldsymbol{\mathcal{H}}_{\qpar}}}

\newcommand{\qbrauer}[3]{\smash{\mathrm{Br}_{#2,#3}^{#1}}}

\newcommand{\qbrauercat}[2]{\boldsymbol{\mathcal{B}\mathrm{r}}_{#1,#2}}

\newcommand{\brauerA}{\beta_{\typeA}}

\newcommand{\sfE}{{\mathrm{E}}}
\newcommand{\sfF}{{\mathrm{F}}}
\newcommand{\sfK}{{\mathrm{K}}}
\newcommand{\sfB}{{\mathrm{B}}}
\newcommand{\sfX}{{\mathrm{X}}}
\newcommand{\sfY}{{\mathrm{Y}}}

\newcommand{\dotE}{{\mathsf{E}}}
\newcommand{\dotF}{{\mathsf{F}}}
\newcommand{\dotX}{{\mathsf{X}}}
\newcommand{\smallone}{\mathsf{1}}

\renewcommand{\cup}{\mathord{\oldcup}}
\renewcommand{\cap}{\mathord{\oldcap}}
\renewcommand{\cup}{\webcup}
\renewcommand{\cap}{\webcap}

\newcommand{\overcrossing}{\tikz[baseline=0,scale=0.25]{\draw (1,0) -- ++(-1,1); \draw[preaction={draw=white,line width=3pt,-}] (0,0) -- ++(1,1);}}

\newcommand{\undercrossing}{\tikz[baseline=0,scale=0.25]{\draw (0,0) -- ++(1,1); \draw[preaction={draw=white,line width=3pt,-}] (1,0) -- ++(-1,1);}}

\newcommand{\webcap}{\,\tikz[baseline=0.5,scale=0.1]{	\draw[] (0,0)  to [out=90, in=180] (1,1.5) to [out=0, in=90] (2,0);}\,}
\newcommand{\webcup}{\,\tikz[baseline=-7,scale=0.1]{ \draw[] (0,0)  to [out=270, in=180] (1,-1.5) to [out=0, in=270] (2,0);}\,}
\newcommand{\cupcap}{\,\tikz[baseline=.8,scale=0.07]{ \draw[] (0,4)  to [out=270, in=180] (1,2.5) to [out=0, in=270] (2,4); \draw[] (0,0)  to [out=90, in=180] (1,1.5) to [out=0, in=90] (2,0);}\,}

\newcommand{\extmerge}{\,\tikz[baseline=2,scale=0.1]{	\draw[very thick, myred] (0,0) to [out=90,in=225] (1,1.5); \draw[very thick, myred] (2,0) to [out=90,in=315] (1,1.5);	\draw[very thick, myred] (1,1.5) to (1,3);}\,}
\newcommand{\extsplit}{\,\tikz[baseline=2,scale=0.1]{	\draw[very thick, myred] (0,3) to [out=270,in=135] (1,1.5); \draw[very thick, myred] (2,3) to [out=270,in=45] (1,1.5);	\draw[very thick, myred] (1,1.5) to (1,0);}\,}
\newcommand{\symmerge}{\,\tikz[baseline=2,scale=0.1]{	\draw[very thick, mygreen] (0,0) to [out=90,in=225] (1,1.5); \draw[very thick, mygreen] (2,0) to [out=90,in=315] (1,1.5);	\draw[very thick, mygreen] (1,1.5) to (1,3);}\,}
\newcommand{\symsplit}{\,\tikz[baseline=2,scale=0.1]{	\draw[very thick, mygreen] (0,3) to [out=270,in=135] (1,1.5); \draw[very thick, mygreen] (2,3) to [out=270,in=45] (1,1.5);	\draw[very thick, mygreen] (1,1.5) to (1,0);}\,}

\newcommand{\extmergesplit}{\,\tikz[baseline=3.5,scale=0.1]{	\draw[] (0,0)  .. controls ++(0,1) and ++(-0.5,-0.5) .. ++(1,1.5) .. controls ++(0.5,-0.5) and ++(0,1) .. ++(1,-1.5) ; 	\draw[very thick, myred] (1,1.5) --  ++(0,1);	\draw[] (0,4)  .. controls ++(0,-1) and ++(-0.5,0.5) .. ++(1,-1.5) .. controls ++(0.5,0.5) and ++(0,-1) .. ++(1,1.5) ;}\,}
\newcommand{\symmergesplit}{\,\tikz[baseline=3.5,scale=0.1]{	\draw[] (0,0)  .. controls ++(0,1) and ++(-0.5,-0.5) .. ++(1,1.5) .. controls ++(0.5,-0.5) and ++(0,1) .. ++(1,-1.5) ; 	\draw[very thick, mygreen] (1,1.5) --  ++(0,1);	\draw[] (0,4)  .. controls ++(0,-1) and ++(-0.5,0.5) .. ++(1,-1.5) .. controls ++(0.5,0.5) and ++(0,-1) .. ++(1,1.5) ;}\,}

\newcommand{\extmerget}{\,\tikz[baseline=2,scale=0.1]{	\draw[] (0,0) to [out=90,in=225] (1,1.5); \draw[] (2,0) to [out=90,in=315] (1,1.5);	\draw[very thick, myred] (1,1.5) to (1,3);}\,}
\newcommand{\extsplitt}{\,\tikz[baseline=2,scale=0.1]{	\draw[] (0,3) to [out=270,in=135] (1,1.5); \draw[] (2,3) to [out=270,in=45] (1,1.5);	\draw[very thick, myred] (1,1.5) to (1,0);}\,}
\newcommand{\symmerget}{\,\tikz[baseline=2,scale=0.1]{	\draw[] (0,0) to [out=90,in=225] (1,1.5); \draw[] (2,0) to [out=90,in=315] (1,1.5);	\draw[very thick, mygreen] (1,1.5) to (1,3);}\,}
\newcommand{\symsplitt}{\,\tikz[baseline=2,scale=0.1]{	\draw[] (0,3) to [out=270,in=135] (1,1.5); \draw[] (2,3) to [out=270,in=45] (1,1.5);	\draw[very thick, mygreen] (1,1.5) to (1,0);}\,}

\newcommand{\extdotdown}{\,\tikz[baseline=0.5,scale=0.1]{	\draw[very thick,myred] (0,0.7) -- ++(0,2); \node[circle,fill,inner sep=1pt, myred] at (0,0.7) {};}\,}
\newcommand{\extdotup}{\,\tikz[baseline=0.5,scale=0.1]{	\draw[very thick,myred] (0,1.7) -- ++(0,-2); \node[circle,fill,inner sep=1pt, myred] at (0,1.7) {};}\,}
\newcommand{\symdotdown}{\,\tikz[baseline=0.5,scale=0.1]{	\draw[very thick, mygreen] (0,0.7) -- ++(0,2); \node[circle,fill,inner sep=1pt, mygreen] at (0,0.7) {};}\,}
\newcommand{\symdotup}{\,\tikz[baseline=0.5,scale=0.1]{	\draw[very thick,mygreen] (0,1.7) -- ++(0,-2); \node[circle,fill,inner sep=1pt, mygreen] at (0,1.7) {};}\,}

\newcommand{\extdotdowndotup}{\,\tikz[baseline=2,scale=0.09]{ \draw[very thick,myred] (0,4)  to (0,2.75); \draw[very thick,myred] (0,0) to (0,1.25); \node[circle,fill,inner sep=.75pt, myred] at (0,1.25) {}; \node[circle,fill,inner sep=.75pt, myred] at (0,2.75) {};}\,}

\newcommand{\smallmerge}{\,\tikz[baseline=2,scale=0.075]{\draw[very thick] (0,0) to [out=90,in=225] (1,1.5); \draw[very thick] (2,0) to [out=90,in=315] (1,1.5);	\draw[very thick] (1,1.5) to (1,3);}\,}
\newcommand{\smalldotdown}{\,\tikz[baseline=2,scale=0.075]{\draw[white,very thick] (0,0) to [out=90,in=225] (1,1.5); \draw[white,very thick] (2,0) to [out=90,in=315] (1,1.5);	\draw[white,very thick] (1,1.5) to (1,3); \draw[] (0,1) to (0,3); \draw[fill=black] (0,1) circle (.5cm);}\,}
\newcommand{\smallcup}{\,\tikz[baseline=2,scale=0.075]{ \draw[white,very thick] (0,0) to [out=90,in=225] (1,1.5); \draw[white,very thick] (2,0) to [out=90,in=315] (1,1.5);	\draw[white,very thick] (1,1.5) to (1,3); \draw[] (0,2.5)  to [out=270, in=180] (1,1) to [out=0, in=270] (2,2.5);}\,}

\newcommand{\equiso}{\,\tikz[baseline=-1,scale=0.075]{\draw[] (-6,-1) to (-6,3); \draw[] (0,3)  to [out=270, in=180] (1,1.5) to [out=0, in=270] (2,3); \node at (-3,1) {$\otimes$};}\,}
\newcommand{\notequiso}{\,\tikz[baseline=-1,scale=0.075]{\draw[] (8,-1) to (8,3); \draw[] (0,3)  to [out=270, in=180] (1,1.5) to [out=0, in=270] (2,3); \node at (5,1) {$\otimes$};}\,}

\newcommand{\equisp}{\,\tikz[baseline=-1,scale=0.075]{\draw[] (-6,-1) to (-6,3); \draw[very thick, myred] (0,1) to (0,3); \draw[myred,fill=myred] (0,1) circle (.5cm); \node at (-3,1) {$\otimes$}}\,}
\newcommand{\notequisp}{\,\tikz[baseline=-1,scale=0.075]{\draw[] (6,-1) to (6,3); \draw[very thick, myred] (0,1) to (0,3); \draw[myred,fill=myred] (0,1) circle (.5cm); \node at (3,1) {$\otimes$}}\,}

\newcommand{\WebAz}{\smash{\boldsymbol{\mathcal{W}\mathrm{eb}}_{\qpar,\zpar}^{\smallmerge}}}
\newcommand{\WebCz}{\smash{\boldsymbol{\mathcal{W}\mathrm{eb}}_{\qpar,\zpar}^{\smalldotdown}}}
\newcommand{\WebDz}{\smash{\boldsymbol{\mathcal{W}\mathrm{eb}}_{\qpar,\zpar}^{\smallcup}}}

\newcommand{\WebA}{\smash{\boldsymbol{\mathcal{W}\mathrm{eb}}_{\qpar}^{\smallmerge}}}
\newcommand{\WebC}{\smash{\boldsymbol{\mathcal{W}\mathrm{eb}}_{\qpar,\qpar^n}^{\smalldotdown}}}
\newcommand{\WebD}{\smash{\boldsymbol{\mathcal{W}\mathrm{eb}}_{\qpar,\qpar^n}^{\smallcup}}}

\newcommand{\WebCpara}{\smash{\boldsymbol{\mathcal{W}\mathrm{eb}}_{\qpar,-\qpar^{-n}}^{\smalldotdown}}}
\newcommand{\WebDpara}{\smash{\boldsymbol{\mathcal{W}\mathrm{eb}}_{\qpar,-\qpar^{-n}}^{\smallcup}}}

\newcommand{\WebCwithpara}[1]{\smash{\boldsymbol{\mathcal{W}\mathrm{eb}}_{\qpar,#1}^{\smalldotdown}}}
\newcommand{\WebDwithpara}[1]{\smash{\boldsymbol{\mathcal{W}\mathrm{eb}}_{\qpar,#1}^{\smallcup}}}

\newcommand{\macrodotdown}{\tikz[baseline=2,scale=0.1]{\draw[white,very thick] (0,1.5) to (.75,1.5); \draw[black] (0,1) to (0,3); \draw[ultra thin, black, fill=black] (0,1) circle (.5cm);}}
\newcommand{\macrocup}{\tikz[baseline=2,scale=0.1]{\draw[white,very thick] (0,0) to [out=90,in=225] (1,1.5); \draw[white,very thick] (2,0) to [out=90,in=315] (1,1.5);	\draw[white,very thick] (1,1.5) to (1,3); \draw[black] (0,2.5)  to [out=270, in=180] (1,1) to [out=0, in=270] (2,2.5);}}

\DeclareRobustCommand{\typeBDC}{\texorpdfstring{\macrocup}{cup}}
\DeclareRobustCommand{\typeCBD}{\texorpdfstring{\macrodotdown}{dot}}

\newcommand{\brauerD}{\beta_{\typeBDC}}
\newcommand{\brauerC}{\beta_{\typeCBD}}

\newtheoremstyle{myplain} {6pt plus 6pt minus 2pt}
{6pt plus 6pt minus 2pt}
{\itshape}
{}
{\bfseries}
{.}
{.5em}
{}

\theoremstyle{myplain}
\newtheorem{theoremm}{Theorem}[section]
\newtheorem{theoremmain}{Theorem}

\newtheoremstyle{mydefinition} {6pt plus 6pt minus 2pt}
{6pt plus 6pt minus 2pt}
{}
{}
{\bfseries}
{.}
{.5em}
{}

\theoremstyle{mydefinition}

\newtheoremstyle{myexample} {6pt plus 6pt minus 2pt}
{6pt plus 6pt minus 2pt}
{}
{}
{\scshape}
{.}
{.5em}
{}

\newtheoremstyle{myremark} {6pt plus 6pt minus 2pt}
{6pt plus 6pt minus 2pt}
{}
{}
{\scshape}
{.}
{.5em}
{}

\declaretheorem[style=myplain,name=Theorem,numberlike=theoremm]{theorem}

\declaretheorem[style=myplain,name=Lemma,numberlike=theoremm]{lemma}
\declaretheorem[style=myplain,name=Proposition,numberlike=theoremm]{proposition}

\declaretheorem[style=myexample,name=Example,numberlike=theoremm]{example}
\declaretheorem[style=mydefinition,name=Definition,numberlike=theoremm]{definition}
\declaretheorem[style=myremark,name=Remark,numberlike=theoremm]{remark}
\declaretheorem[style=mydefinition,name=Convention,numberlike=theoremm]{convention}

\usepackage{chngcntr}
\counterwithout{figure}{section}

\def\notation#1#2#3{\rlap{\hyperref[#1]{{\color{orchid}#2}}}\hspace*{8.2mm} \hbox to 47mm{#3\hfill}}

\newcommand{\qedmake}{\hfill\ensuremath{\blacksquare}}
\renewcommand{\qedsymbol}{$\blacksquare$}

\definecolor{myred}{RGB}{221,32,37}
\definecolor{mygreen}{RGB}{12,145,12}
\definecolor{myblue}{HTML}{0000cc}
\definecolor{orchid}{RGB}{143,40,194}
\definecolor{lava}{RGB}{207,16,32}
\definecolor{mydarkblue}{RGB}{10,10,170}

\newcommand{\comm}[1]{}
\allowdisplaybreaks

\numberwithin{equation}{section}
\usepackage[final]{hyperref, bookmark}
\hypersetup{
    pdftoolbar=true,        
    pdfmenubar=true,        
    pdffitwindow=false,     
    pdfstartview={FitH},    
    pdftitle={Webs and \texorpdfstring{$\qpar$}{q}-Howe dualities in types \texorpdfstring{$\typeB\typeC\typeD$}{BCD}},    
    pdfauthor={Antonio Sartori and Daniel Tubbenhauer},     
    pdfsubject={},   
    pdfcreator={Antonio Sartori and Daniel Tubbenhauer},   
    pdfproducer={Antonio Sartori and Daniel Tubbenhauer}, 
    pdfkeywords={}, 
    pdfnewwindow=true,      
    colorlinks=true,       
    linkcolor=mydarkblue,          
    citecolor=teal,        
    filecolor=magenta,      
    urlcolor=orchid,          
    linkbordercolor=lava,
    citebordercolor=teal,
    urlbordercolor=orchid,  
    linktocpage=true
}

\let\fullref\autoref
\def\makeautorefname#1#2{\expandafter\def\csname#1autorefname\endcsname{#2}}
\makeautorefname{equation}{Equation}%
\makeautorefname{footnote}{footnote}%
\makeautorefname{item}{item}%
\makeautorefname{figure}{Figure}%
\makeautorefname{table}{Table}%
\makeautorefname{part}{Part}%
\makeautorefname{appendix}{Appendix}%
\makeautorefname{chapter}{Chapter}%
\makeautorefname{section}{Section}%
\makeautorefname{subsection}{Section}%
\makeautorefname{subsubsection}{Section}%
\makeautorefname{paragraph}{Paragraph}%
\makeautorefname{subparagraph}{Paragraph}%
\makeautorefname{theorem}{Theorem}%
\makeautorefname{theo}{Theorem}%
\makeautorefname{thm}{Theorem}%
\makeautorefname{addendum}{Addendum}%
\makeautorefname{addend}{Addendum}%
\makeautorefname{add}{Addendum}%
\makeautorefname{maintheorem}{Main theorem}%
\makeautorefname{mainthm}{Main theorem}%
\makeautorefname{corollary}{Corollary}%
\makeautorefname{corol}{Corollary}%
\makeautorefname{coro}{Corollary}%
\makeautorefname{cor}{Corollary}%
\makeautorefname{lemma}{Lemma}%
\makeautorefname{lemm}{Lemma}%
\makeautorefname{lem}{Lemma}%
\makeautorefname{sublemma}{Sublemma}%
\makeautorefname{sublem}{Sublemma}%
\makeautorefname{subl}{Sublemma}%
\makeautorefname{proposition}{Proposition}%
\makeautorefname{proposit}{Proposition}%
\makeautorefname{propos}{Proposition}%
\makeautorefname{propo}{Proposition}%
\makeautorefname{prop}{Proposition}%
\makeautorefname{property}{Property}
\makeautorefname{proper}{Property}
\makeautorefname{scholium}{Scholium}%
\makeautorefname{step}{Step}%
\makeautorefname{conjecture}{Conjecture}%
\makeautorefname{conject}{Conjecture}%
\makeautorefname{conj}{Conjecture}%
\makeautorefname{question}{Question}
\makeautorefname{questn}{Question}
\makeautorefname{quest}{Question}
\makeautorefname{ques}{Question}
\makeautorefname{qn}{Question}
\makeautorefname{definition}{Definition}%
\makeautorefname{defin}{Definition}%
\makeautorefname{defi}{Definition}%
\makeautorefname{def}{Definition}%
\makeautorefname{dfn}{Definition}%
\makeautorefname{notation}{Notation}
\makeautorefname{nota}{Notation}
\makeautorefname{notn}{Notation}
\makeautorefname{remark}{Remark}%
\makeautorefname{rema}{Remark}%
\makeautorefname{rem}{Remark}%
\makeautorefname{rmk}{Remark}%
\makeautorefname{rk}{Remark}%
\makeautorefname{remarks}{Remarks}%
\makeautorefname{rems}{Remarks}%
\makeautorefname{rmks}{Remarks}%
\makeautorefname{rks}{Remarks}%
\makeautorefname{example}{Example}%
\makeautorefname{examp}{Example}%
\makeautorefname{exmp}{Example}%
\makeautorefname{exam}{Example}%
\makeautorefname{exa}{Example}%
\makeautorefname{algorithm}{Algorith}%
\makeautorefname{algo}{Algorith}%
\makeautorefname{alg}{Algorith}%
\makeautorefname{axiom}{Axiom}%
\makeautorefname{axi}{Axiom}%
\makeautorefname{ax}{Axiom}%
\makeautorefname{case}{Case}%
\makeautorefname{claim}{Claim}%
\makeautorefname{clm}{Claim}%
\makeautorefname{assumption}{Assumption}%
\makeautorefname{assumpt}{Assumption}%
\makeautorefname{conclusion}{Conclusion}%
\makeautorefname{concl}{Conclusion}%
\makeautorefname{conc}{Conclusion}%
\makeautorefname{condition}{Condition}%
\makeautorefname{condit}{Condition}%
\makeautorefname{cond}{Condition}%
\makeautorefname{construction}{Construction}%
\makeautorefname{construct}{Construction}%
\makeautorefname{const}{Construction}%
\makeautorefname{cons}{Construction}%
\makeautorefname{criterion}{Criterion}%
\makeautorefname{criter}{Criterion}%
\makeautorefname{crit}{Criterion}%
\makeautorefname{exercise}{Exercise}%
\makeautorefname{exer}{Exercise}%
\makeautorefname{exe}{Exercise}%
\makeautorefname{problem}{Problem}%
\makeautorefname{problm}{Problem}%
\makeautorefname{probm}{Problem}%
\makeautorefname{prob}{Problem}%
\makeautorefname{solution}{Solution}%
\makeautorefname{soln}{Solution}%
\makeautorefname{sol}{Solution}%
\makeautorefname{summary}{Summary}%
\makeautorefname{summ}{Summary}%
\makeautorefname{sum}{Summary}%
\makeautorefname{operation}{Operation}%
\makeautorefname{oper}{Operation}%
\makeautorefname{observation}{Observation}%
\makeautorefname{observn}{Observation}%
\makeautorefname{obser}{Observation}%
\makeautorefname{obs}{Observation}%
\makeautorefname{ob}{Observation}%
\makeautorefname{convention}{Convention}%
\makeautorefname{convent}{Convention}%
\makeautorefname{conv}{Convention}%
\makeautorefname{cvn}{Convention}%
\makeautorefname{warning}{Warning}%
\makeautorefname{warn}{Warning}%
\makeautorefname{note}{Note}%
\makeautorefname{fact}{Fact}%
\makeautorefname{theoremmain}{Theorem}%

\begin{document}
\vbadness=10001
\hbadness=10001
\title[Webs and $\qpar$-Howe dualities in types $\typeB\typeC\typeD$]{Webs and $\qpar$-Howe dualities in types $\typeB\typeC\typeD$}
\author{Antonio Sartori}
\address{A.S.: \textit{E-mail address:} {\tt anton.sartori@gmail.com}}

\author{Daniel Tubbenhauer}
\address{D.T.: Institut f\"ur Mathematik, Universit\"at Z\"urich, Winterthurerstrasse 190, Campus Irchel, Office Y27J32, CH-8057 Z\"urich, Switzerland,
\textit{E-mail address:} {\tt daniel.tubbenhauer@math.uzh.ch}}

\begin{abstract}
We define 
web categories describing intertwiners for the 
orthogonal and symplectic 
Lie algebras, and, in the quantized setup, 
for certain
orthogonal and symplectic coideal 
subalgebras. 
They generalize the 
Brauer category, and allow us to 
prove quantum versions of some
classical type $\typeB\typeC\typeD$
Howe dualities.
\end{abstract}

\maketitle

\tableofcontents
\renewcommand{\theequation}{\thesection-\arabic{equation}}
\section{Introduction}\label{sec-intro}
Throughout the whole paper we 
fix $k,n\in\Z_{\geq 0}$, and we assume that 
$n$ is even whenever we write $\fraksp_n$.

\subsection{The framework}

Consider the following question: Given some 
Lie algebra $\frakg$, can one 
give a generator-relation 
presentation
for the 
category of its finite-dimensional representations,
or for some well-behaved 
subcategory?

Maybe the best-known instance of this 
is the case
of the monoidal category generated by the vector 
representation $\vecrepnoq$ of $\fraksl_2$, 
or by the corresponding representation $\vecrep$ of 
its quantized enveloping algebra $\quantumg(\fraksl_2)$. Its 
generator-relation 
presentation is known as the 
\emph{Temperley--Lieb category} and goes back to 
work of Rumer--Teller--Weyl \cite{RTW} and Temperley--Lieb \cite{TL} 
(the latter in the quantum setting).

In pioneering work, Kuperberg \cite{Kup1} extended 
this to all rank $2$ simple Lie algebras and their quantum 
enveloping algebras.
However, it was not clear for quite some time how to extend 
Kuperberg's constructions further (although some partial 
results were obtained). Then, in seminal 
work \cite{CKM}, Cautis--Kamnitzer--Mor\-rison 
gave a generator-relation presentation of the monoidal category
generated by (quantum) exterior powers of the vector representation 
$\vecrep$ of $\quantumg(\frakgl_n)$.

Their crucial observation was that a classical tool 
from representation and invariant theory, known as 
\emph{skew Howe duality} \cite{Ho1,Ho}, can be quantized 
and used as a device to describe intertwiners of $\quantumg(\frakgl_n)$.
This \emph{skew $\qpar$-Howe duality}
is based on the $\quantumg(\frakgl_n)$-module decomposition
\begin{equation}\label{eq:ident-q-skew-howe}
\exterior{\bullet}(\vecrep \otimes \C_{\qpar}^k) \cong {\textstyle\bigoplus_{a_i\in\Z_{\geq 0}}}\exterior{a_1}\vecrep\otimes\cdots\otimes
\exterior{a_k}\vecrep.
\end{equation}
Here $\C_{\qpar}= \C(q)$ is the function 
field in one variable $\qpar$ over the complex numbers, 
and $\exterior{\bullet}$ denotes 
the quantum exterior algebra in the sense of \cite{BZ1}.
Having \eqref{eq:ident-q-skew-howe}, one obtains commuting actions
\begin{gather}\label{eq:q-skew-howe}
\quantumg(\frakgl_n)\acts
{\textstyle\bigoplus_{a_i\in\Z_{\geq 0}}}\exterior{a_1}\vecrep\otimes\cdots\otimes\exterior{a_k}\vecrep
\actsreverse\quantumg(\frakgl_k).
\end{gather}
These two actions generate each other's centralizer, and the 
bimodule decomposition can be explicitly given.
Moreover, by studying the kernel of the 
$\quantumg(\frakgl_k)$-action, one can then completely describe the intertwiners 
of $\quantumg(\frakgl_n)$.
In fact, as explained in \cite{CKM}, 
they allow a nice diagrammatic interpretation 
via so-called \emph{$\typeA$-webs}, 
which are basically defined by using the 
$\quantumg(\frakgl_k)$-action.

The results from \cite{CKM} were then 
extended to various other instances.
But, to the best of our knowledge, 
all generalizations so far stay 
in type $\typeA$.

The idea 
which started this paper was to 
extend Cautis--Kamnitzer--Morrison's approach to 
types $\typeB\typeC\typeD$.
However, the main obstacle
immediately arises:
while the 
quantization of skew Howe duality is fairly straightforward in 
type $\typeA$, 
it is not even clear in other types how one can define commuting 
actions as in \eqref{eq:q-skew-howe}.
The underlying problem hereby is that
$\exterior{\bullet}(\vecrep \otimes \C_{\qpar}^k)$ is not flat if 
$\vecrep$ is the vector representation in types 
$\typeB\typeC\typeD$ (while this holds in type $\typeA$, 
cf.\ \cite{BZ1} and \cite[Corollary 4.26]{Zw1}). This means that 
$\exterior{\bullet}(\vecrep \otimes \C_{\qpar}^k)$ 
does not have 
the same dimension as its 
classical counterpart $\exteriornoq{\bullet}(\vecrepnoq \otimes \C^k)$. 
Hence, there is no hope for an 
isomorphism as in \eqref{eq:ident-q-skew-howe} 
outside type $\typeA$, and we cannot 
follow the approach of \cite{CKM}.

To overcome this problem, we consider 
alternative quantizations of $\frakso_n$ and $\fraksp_n$, 
namely as so-called \emph{coideal subalgebras} 
$\coideal(\frakso_n) \subset \quantumg(\frakgl_n)$ 
and $\coideal(\fraksp_n) \subset \quantumg(\frakgl_n)$, see \cite{Le} or \cite{KP}. 
For their vector representations, 
the decomposition \eqref{eq:ident-q-skew-howe} 
does hold,
since they are subalgebras of $\quantumg(\frakgl_n)$.
Hence, we get commuting actions of $\quantumg(\frakgl_k)$ and
of the $\typeA$-webs.
However, since these coideals are proper subalgebras of 
$\quantumg(\frakgl_n)$, such commuting actions 
do not
generate each other's centralizer, cf.\ \eqref{eq:restrict-action}. 
Consequently, the $\typeA$-web category
does not give rise to full functors to the representation 
categories of the coideal subalgebras $\coideal(\frakso_n)$ 
and $\coideal(\fraksp_n)$.

In order to get full functors, we define  
extended web categories, 
which we call \emph{$\typeBDC$- and 
$\typeCBD$-web categories}, and prove that 
they act on  
the representation categories of the coideal subalgebras.
We 
will then show that these extended web categories are closely 
connected to $\quantumg(\frakso_{2k})$ and $\quantumg(\fraksp_{2k})$ 
(these are the usual quantized enveloping algebras!), 
recovering some versions of $\qpar$-Howe duality in types $\typeB\typeC\typeD$.

Note that our approach goes somehow 
\textbf{the opposite way} with respect to \cite{CKM}: 
instead of 
using $\qpar$-Howe duality to obtain a
web calculus, we use our web categories to prove quantized Howe dualities.
The idea of reversing Cautis--Kamnitzer--Morrison's path comes 
from the paper \cite{QS}, where it was first deployed to quantize 
a different kind of Howe duality in type 
$\typeA$ (in which the vector representation appears together with its dual).
This idea was of considerable importance for this work, 
and indeed many diagrammatic proofs in our paper are inspired by \cite{QS}. 

\subsection{Main results and proof strategy}\label{subsec-main-results}

As before, we denote by $\vecrep$ the vector 
representation of $\quantumg(\frakgl_n)$, as well 
as of its coideal subalgebras $\coideal(\frakso_n)$ 
and $\coideal(\fraksp_n)$.
We denote by $\exterior{\bullet}\vecrep$ the 
exterior algebra 
and by $\symmetric{\bullet}\vecrep$ the 
symmetric algebra 
of $\vecrep$.

\subsubsection{Quantizing Howe dualities in types \texorpdfstring{$\typeB\typeC\typeD$}{BCD}}
As recalled above, the quantum version of skew Howe duality \cite[Theorem 6.16]{LZZ}
states
that there are commuting actions
generating each other's centralizer:
\begin{equation}\label{eq:main-no-q}
  \quantumg(\frakgl_n) \acts
  \underbrace{\exterior{\bullet} \vecrep \otimes 
    \dots \otimes \exterior{\bullet} \vecrep}_{k \text{ times}}
  \actsreverse
  \quantumg(\frakgl_k).
\end{equation}
The corresponding bimodule decomposition is
multi\-pli\-city-free and can be explicitly given.
An analog statement holds if we replace $\exterior{\bullet} \vecrep$
with $\symmetric{\bullet}\vecrep$
(although one has to be slightly more careful since the 
representation becomes infinite-dimensional).

As observed by Howe \cite{Ho1,Ho}, in the classical 
setting there are four versions 
of \eqref{eq:main-no-q} in types $\typeB\typeC\typeD$.
Our main result is a quantization of 
Howe's $\typeB\typeC\typeD$-dualities.
In this quantization, notably, on the 
right-hand side the enveloping algebras $\U(\fraksp_{2k})$ 
and $\U(\frakso_{2k})$ become their quantum enveloping algebras,
but on the left-hand side they get replaced by the coideal subalgebras 
$\coideal(\frakso_n)$ and $\coideal(\fraksp_n)$.

\begin{theoremmain}\label{theorem:main-intro}
There are commuting actions:
\begin{gather}
\coideal(\frakso_n)
\acts  
\mathmakebox[\widthof{$\underbrace{\symmetric{\bullet} \vecrep \otimes 
  \dots \otimes \symmetric{\bullet} \vecrep}_{k \text{ times}}$}][c]{\underbrace{\exterior{\bullet} \vecrep \otimes 
\dots \otimes \exterior{\bullet} \vecrep}_{k \text{ times}}}
\actsreverse 
\quantumg(\frakso_{2k}),
\label{eq:main-q-1}\\
\coideal(\frakso_n)
\acts  
\underbrace{\symmetric{\bullet} \vecrep \otimes 
\dots \otimes \symmetric{\bullet} \vecrep}_{k \text{ times}}
\actsreverse 
\Udot(\fraksp_{2k}),
\label{eq:main-q-2}\\
\coideal(\fraksp_n)
\acts  
\mathmakebox[\widthof{$\underbrace{\symmetric{\bullet} \vecrep \otimes 
  \dots \otimes \symmetric{\bullet} \vecrep}_{k \text{ times}}$}][c]{\underbrace{\exterior{\bullet} \vecrep \otimes 
\dots \otimes \exterior{\bullet} \vecrep}_{k \text{ times}}}
\actsreverse 
\quantumg(\fraksp_{2k}),
\label{eq:main-q-3}\\
\coideal(\fraksp_n)
\acts  
\underbrace{\symmetric{\bullet} \vecrep \otimes 
\dots \otimes \symmetric{\bullet} \vecrep}_{k \text{ times}}
\actsreverse 
\Udot(\frakso_{2k}).
\label{eq:main-q-4}
\end{gather}
In \eqref{eq:main-q-1} and \eqref{eq:main-q-2} for $n$ odd,
and in \eqref{eq:main-q-3} and \eqref{eq:main-q-4},
the two actions 
generate each other's centralizer. Hence, the corresponding 
bimodule decompositions are 
multiplicity-free.
Moreover, all the above 
de-quantize to the associated 
classical dualities of Howe.
\end{theoremmain}

In \eqref{eq:main-q-1} and \eqref{eq:main-q-2} for 
$n$ even one has to add an additional intertwiner 
on the right-hand side in order to get a full action 
(see \fullref{remark:sovso}).

Our $\qpar$-Howe dualities 
are related to \eqref{eq:main-no-q} as follows:
\begin{equation}\label{eq:restrict-action}
   \begin{tikzpicture}[anchorbase]
   \matrix (m) [matrix of math nodes, row sep=1em, column
   sep=1em, text height=1.8ex, text depth=0.25ex] {
   \quantumg(\frakgl_n)
   & {\textstyle\bigoplus_{a_i\in\Z_{\geq 0}}}\exterior{a_1}\vecrep\otimes\cdots\otimes\exterior{a_k}\vecrep &  
   \phantom{.}\quantumg(\frakgl_{k}) \\
   \coideal(\frakso_n) & \underbrace{\exterior{\bullet} \vecrep \otimes 
   \dots \otimes \exterior{\bullet} \vecrep}_{k \text{ times}} & \quantumg(\frakso_{2k}), \\};
   \node at ($(m-1-1.east) + (.2,0)$) {$\acts$};
   \node at ($(m-2-1.east) + (.17,0)$) {$\acts$};
   \node at ($(m-1-3.west) + (-.2,0)$) {$\actsreverse$};
   \node at ($(m-2-3.west) + (-.1,0)$) {$\actsreverse$};
   \node at ($(m-1-1.south) + (0,-.2)$) {\rotatebox{90}{$\subset$}};
   \node at ($(m-1-3.south) + (0,-.2)$) {\rotatebox{90}{$\supset$}};
   \node at ($(m-1-2.south) + (0,-.2)$) {\rotatebox{90}{$=$}};
   \node at ($(m-2-2.south) + (0,-.25)$) {$\phantom{=}$};
   \end{tikzpicture}
\end{equation}
and similarly in the other three cases \eqref{eq:main-q-2}, \eqref{eq:main-q-3} 
and \eqref{eq:main-q-4}.

\subsubsection{Explaining the strategy}

Our main tools are certain diagrams made out of
trivalent graphs with 
edge labels from $\Np$,
which we call \emph{$\typeA$-}, \emph{$\typeBDC$-} and \emph{$\typeCBD$-webs}.
 
The $\typeA$-webs where introduced in \cite{CKM} and assemble into 
a monoidal category $\WebA$. 
The $\typeBDC$- and $\typeCBD$-webs are introduced in this paper 
in order to define  
categories $\WebDz$ and $\WebCz$.
These categories are not monoidal, but they come with a 
left action of the monoidal category 
$\WebA$, cf.\ \fullref{remark:coideal}.

\makeautorefname{section}{Sections} 

We will define these web categories in \fullref{sec-Awebs}, \ref{sec-Dwebs} 
and \ref{sec-Cwebs}. 
All the reader needs to know about them 
at the moment is summarized in \fullref{figure:webs}.

\makeautorefname{section}{Section} 

\begin{figure}[ht]
\[
	\begin{tikzpicture}[anchorbase,scale=.25,tinynodes]
	\draw[colored] (0,0) node[below] {$5$} to [out=90, in=225] (1,1.5) to (1,2.25) node[left] {$8$} to (1,3.05);
	\draw[colored] (2,0) node[below] {$3$} to [out=90, in=315] (1,1.5);
	\draw[colored] (1,3) to [out=135, in=270] (0,4.5) to [out=135, in=270] (-1,6);
	\draw[colored] (0,4.5) to [out=45, in=270] (1,6) to (1,9) node[above] {$4$};
	\draw[colored] (-1,6) node[right] {$3$} to [out=90, in=315] (-2,7.5) to (-2,9) node[above] {$9$};
	\draw[colored] (-2,0) node[below] {$3$} to [out=90, in=315] (-3,1.5) to (-3,3.75) node[left] {$6$} to (-3,6) to [out=90, in=225] (-2,7.5);
	\draw[colored] (-4,0) node[below] {$3$} to [out=90, in=225] (-3,1.5);
	\draw[uncolored] (1,3) to [out=45, in=270] (2,4.5) to [out=90, in=180] (3,5.5) to [out=0, in=90](4,4.5) to (4,0) node[below] {\raisebox{-.0275cm}{$1$}};
	\draw[uncolored] (3,9) node[above] {\raisebox{.0375cm}{$1$}} to [out=270, in=180] (4,8) to [out=0, in=270] (5,9) node[above] {\raisebox{.0375cm}{$1$}};
	\draw[uncolored] (7,9) node[above] {\raisebox{.0375cm}{$1$}} to [out=270, in=180] (8,8) to [out=0, in=270] (9,9) node[above] {\raisebox{.0375cm}{$1$}};
	\node at (-.5,3.5) {$7$};
	\draw[uncolored, densely dotted] (3,-2.5) to [out=90, in=270] (2,4.5) to [out=90, in=270] (2,11.5);
	\node at (2.5,12) {$\typeBDC\text{-web}$};
	\node at (-.5,-2) {$\typeA\text{-web}$};
	\node at (6,-2) {$\text{new part}$};
    \end{tikzpicture}
    \qquad \qquad
	\begin{tikzpicture}[anchorbase,scale=.25, tinynodes]
	\draw[uncolored] (-1,0) node[below] {\raisebox{-.0375cm}{$1$}} to [out=90, in=225] (0,1.5) to [out=315, in=90] (1,0) node[below] {\raisebox{-.0275cm}{$1$}};
    \draw[ccolored] (0,1.5) to (0,2.25) node[left] {$2$} to (0,3);
    \draw[colored] (0,3) to (0,3.75) node[left] {$3$} to (0,4.5);
    \draw[ccolored] (0,4.5) to (0,5.25) node[left] {$2$} to [out=90, in=315] (-1,7.5);
    \draw[ccolored] (0,3) to (1.5,2.5) to (1.5,2.25) node[above] {$2$} to (1.5,1.75);
    \draw[ccolored] (0,4.5) to (1.5,5) to (1.5,5.2) node[below] {$2$} to (1.5,5.75);
    \draw[colored] (-3,0) node[below] {$7$} to (-3,6) to [out=90, in=225] (-1,7.5) to (-1,9) node[above] {$9$};
    \node[dotbullet] at (1.5,1.75) {};
	\node[dotbullet] at (1.5,5.75) {};
	\draw[ccolored] (3,0) node[below] {$\raisebox{-.0175cm}{$2$}$} to (3,1.5);
	\node[dotbullet] at (3,1.5) {};
	\draw[ccolored] (5,0) node[below] {$\raisebox{-.0175cm}{$2$}$} to (5,1.5);
	\node[dotbullet] at (5,1.5) {};
	\draw[colored] (-6,0) node[below] {$6$} to (-6,1.5) to [out=135, in=270] (-7,2.99) to (-7,3) node[left] {$5$} to (-7,3.01) to [out=90, in=225] (-6,4.5) to (-6,5.25) node[left] {$6$} to (-6,6) to [out=135, in=270] (-7,7.5) to (-7,9) node[above] {$4$};
	\draw[uncolored] (-6,1.5) to [out=45, in=270] (-5,2.99) to (-5,3) node[right] {$1$} to (-5,3.01) to [out=90, in=315] (-6,4.5);
	\draw[ccolored] (-6,6) to [out=45, in=270] (-5,7.5) to (-5,9) node[above] {$2$};
	\draw[uncolored, densely dotted] (2,-2.5) to [out=90, in=270] (.75,2.25) to (.75,4.75) to [out=90, in=270] (.75,11.5);
	\node at (1,12) {$\typeCBD\text{-web}$};
	\node at (-2.5,-2) {$\typeA\text{-web}$};
	\node at (4.5,-2) {$\text{new part}$};
    \end{tikzpicture}
\]
\caption{Examples of our webs.  
Both, $\typeBDC$- and $\typeCBD$-webs, always consist 
of an $\typeA$-web to the left and a part with new generators 
(cup and cap respectively start and end dots) 
on the right.
}\label{figure:webs}
\end{figure}
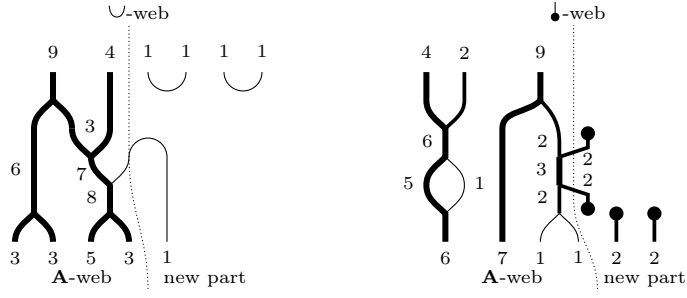

Let $\Repq{n}$, $\RepcoD{n}$ and $\RepcoC{n}$ denote
the categories of finite-di\-men\-sio\-nal representations of
$\quantumg(\frakgl_n)$, $\coideal(\frakso_n)$ and 
$\coideal(\fraksp_n)$, respectively.

Following \cite{CKM}, skew $\qpar$-Howe 
duality gives rise to a 
$\quantumg(\frakgl_n)$-equivariant 
action of $\quantumg(\frakgl_{k})$
on the $k$-fold tensor product of 
$\exterior{\bullet} \vecrep$'s
as in \eqref{eq:main-no-q}.

This induces a functor $\howeA\colon\Udot(\frakgl_{k})\to\Repq{n}$.
By the definition of $\WebA$, this
can also be used to define a functor 
$\diaA \colon \WebA \to \Repq{n}$. 
In fact, 
there is 
a third functor $\ladA\colon\Udot(\frakgl_{k})\to\WebA$
such that $\howeA=\diaA\circ\ladA$.
It follows by skew $\qpar$-Howe duality that all functors $\howeA$, $\diaA$ and $\ladA$ are full.
The same works in the symmetric case (cf.\ \cite{RT} and \cite{TVW}) 
where $\exterior{\bullet} \vecrep$ is replaced by  
$\symmetric{\bullet} \vecrep$: again one constructs full functors
$\symhoweA$ and $\symdiaA$ such that
$\symhoweA=\symdiaA\circ\ladA$. 

Our goal is to have an analogous picture 
in types $\typeB\typeC\typeD$: 
we want to have functors 
$\diaD$, $\diaC$, $\symdiaD$, $\symdiaC$, 
$\ladD$ and $\ladC$
and
commuting 
diagrams as in \fullref{figure:commute}.

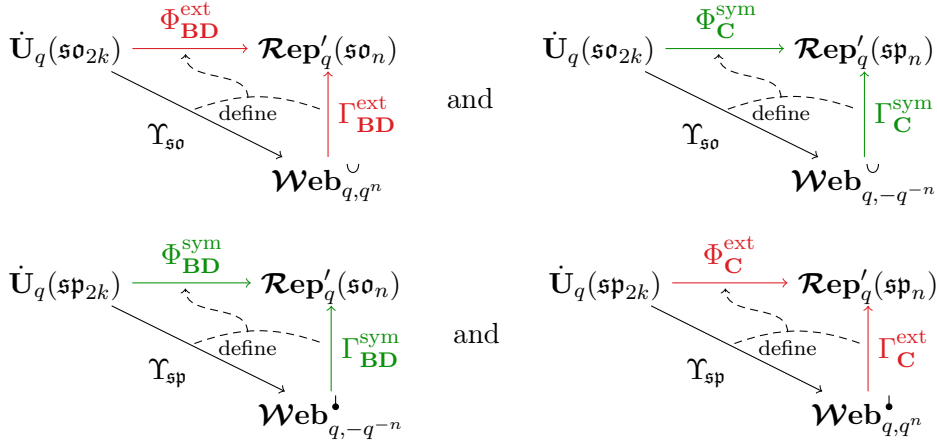
\begin{figure}[ht]
\begin{equation*}
\begin{gathered}
  \begin{tikzpicture}[anchorbase]
  \matrix (m) [matrix of math nodes, row sep=3em, column
  sep=4em, text height=1.8ex, text depth=0.25ex] {
\Udot(\frakso_{2k}) & \RepcoD{n} \\
& \WebD \\};
  \path[->, myred] (m-1-1) edge node[above] {$\howeD$} (m-1-2);
  \path[->] (m-1-1) edge node[below left] {$\ladD$} (m-2-2);
  \path[->, myred] (m-2-2) edge node[right] {$\diaD$} (m-1-2);
  \draw[thin, densely dashed] ($(m-1-1.south) + (1.7,-.5)$) to [out=20,in=160] ($(m-1-2.south) + (-.1,-.5)$);
  \draw[thin, densely dashed, ->] ($(m-1-1.south) + (2.4,-.35)$) to [out=90, in=270] ($(m-1-1.east) + (.7,-.075)$);
  \node at ($(m-1-1.south) + (2.4,-.55)$) {\scriptsize define};
  \end{tikzpicture}
\quad \text{and} \quad
  \begin{tikzpicture}[anchorbase]
  \matrix (m) [matrix of math nodes, row sep=3em, column
  sep=4em, text height=1.8ex, text depth=0.25ex] {
\Udot(\frakso_{2k}) & \RepcoC{n} \\
& \WebDpara \\};
  \path[->, mygreen] (m-1-1) edge node[above] {$\symhoweC$} (m-1-2);
  \path[->] (m-1-1) edge node[below left] {$\ladD$} (m-2-2);
  \path[->, mygreen] (m-2-2) edge node[right] {$\symdiaC$} (m-1-2);
  \draw[thin, densely dashed] ($(m-1-1.south) + (1.7,-.5)$) to [out=20,in=160] ($(m-1-2.south) + (-.1,-.5)$);
  \draw[thin, densely dashed, ->] ($(m-1-1.south) + (2.4,-.35)$) to [out=90, in=270] ($(m-1-1.east) + (.7,-.075)$);
  \node at ($(m-1-1.south) + (2.4,-.55)$) {\scriptsize define};
  \end{tikzpicture}
\end{gathered}
\end{equation*}
\begin{equation*}
\begin{gathered}
\begin{tikzpicture}[anchorbase]
  \matrix (m) [matrix of math nodes, row sep=3em, column
  sep=4em, text height=1.8ex, text depth=0.25ex] {
\Udot(\fraksp_{2k}) & \RepcoD{n} \\
& \WebCpara \\};
  \path[->, mygreen] (m-1-1) edge node[above] {$\symhoweD$} (m-1-2);
  \path[->] (m-1-1) edge node[below left] {$\ladC$} (m-2-2);
  \path[->, mygreen] (m-2-2) edge node[right] {$\symdiaD$} (m-1-2);
  \draw[thin, densely dashed] ($(m-1-1.south) + (1.7,-.5)$) to [out=20,in=160] ($(m-1-2.south) + (-.1,-.5)$);
  \draw[thin, densely dashed, ->] ($(m-1-1.south) + (2.4,-.35)$) to [out=90, in=270] ($(m-1-1.east) + (.7,-.075)$);
  \node at ($(m-1-1.south) + (2.4,-.55)$) {\scriptsize define};
  \end{tikzpicture}
\quad \text{and} \quad
  \begin{tikzpicture}[anchorbase]
  \matrix (m) [matrix of math nodes, row sep=3em, column
  sep=4em, text height=1.8ex, text depth=0.25ex] {
\Udot(\fraksp_{2k}) & \RepcoC{n} \\
& \WebC \\};
  \path[->, myred] (m-1-1) edge node[above] {$\howeC$} (m-1-2);
  \path[->] (m-1-1) edge node[below left] {$\ladC$} (m-2-2);
  \path[->, myred] (m-2-2) edge node[right] {$\diaC$} (m-1-2);
  \draw[thin, densely dashed] ($(m-1-1.south) + (1.7,-.5)$) to [out=20,in=160] ($(m-1-2.south) + (-.1,-.5)$);
  \draw[thin, densely dashed, ->] ($(m-1-1.south) + (2.4,-.35)$) to [out=90, in=270] ($(m-1-1.east) + (.7,-.075)$);
  \node at ($(m-1-1.south) + (2.4,-.55)$) {\scriptsize define};
  \end{tikzpicture}
\end{gathered}
\end{equation*}
\caption{Our main commuting diagrams. We call the various $\howe$'s 
the \emph{Howe functors}, $\dia$'s the \emph{(diagrammatic) presentation functors} 
and $\lad$'s the \emph{ladder functors}.
}\label{figure:commute}
\end{figure}

To summarize (after appropriate parameter substitution in the 
symmetric case):

\begin{theoremmain}\label{theorem:main-intro-2}
There are ladder and presentation functors as 
in \fullref{figure:commute}. These 
define the various Howe
functors therein and hence, 
the actions in \fullref{theorem:main-intro}. 
All of these functors are full in types $\typeB\typeC$.
\end{theoremmain}

As before, fullness in type $\typeD$ can be 
achieved by a slight modification, 
cf.\ \fullref{remark:sovso}.
The connection of the various 
webs and Howe dualities is summarized in \fullref{figure:howe-webs}.

\begin{figure}[ht]
\[
\renewcommand\arraystretch{1.5}
  \arraycolsep=1.5pt
\begin{array}{lcll}
\toprule[1pt]
{\color{myblue}\coideal(\frakso_n)} \acts 
& 
\underbrace{\textstyle {\color{myred}\exterior{\bullet}\vecrep} \otimes \dots \otimes {\color{myred}\exterior{\bullet}\vecrep}}_{k \text{ times}}
& 
\actsreverse \quantumg(\frakso_{2k})
&
\quad\leftrightsquigarrow\quad
\text{``}{\color{myred}\text{exterior }}{\color{myblue}\typeB\typeD\text{-webs}}\text{''}
\\
\midrule
{\color{myblue}\coideal(\frakso_{n})}
\acts
& 
\underbrace{\textstyle {\color{mygreen}\symmetric{\bullet}\vecrep} \otimes \dots \otimes {\color{mygreen}\symmetric{\bullet}\vecrep}}_{k \text{ times}}
&
\actsreverse \quantumg(\fraksp_{2k})
&
\quad\leftrightsquigarrow\quad
\text{``}{\color{mygreen}\text{symmetric }}{\color{myblue}\typeB\typeD\text{-webs}}\text{''}
\\
\midrule
\midrule
{\color{myblue}\coideal(\fraksp_{n})} 
\acts
&
\underbrace{\textstyle {\color{myred}\exterior{\bullet}\vecrep} \otimes \dots \otimes {\color{myred}\exterior{\bullet}\vecrep}}_{k \text{ times}}
&
\actsreverse \quantumg(\fraksp_{2k})
&
\quad\leftrightsquigarrow\quad
\text{``}{\color{myred}\text{exterior }}{\color{myblue}\typeC\text{-webs}}\text{''}
\\
\midrule
{\color{myblue}\coideal(\fraksp_{n})} 
\acts
&
\underbrace{\textstyle {\color{mygreen}\symmetric{\bullet}\vecrep} \otimes \dots \otimes {\color{mygreen}\symmetric{\bullet}\vecrep}}_{k \text{ times}} 
&
\actsreverse \quantumg(\frakso_{2k})
&
\quad\leftrightsquigarrow\quad
\text{``}{\color{mygreen}\text{symmetric }}{\color{myblue}\typeC\text{-webs}}\text{''}
\\\bottomrule[1pt]
\end{array}
\]
\caption{Webs and $\qpar$-Howe dualities.}
\label{figure:howe-webs}
\end{figure}

\makeautorefname{theoremmain}{Theorems} 

Moreover,  
we will explain 
in \fullref{sec:howe-duality}
how \fullref{theorem:main-intro} 
and \ref{theorem:main-intro-2} (in particular, the commuting diagrams 
from \fullref{figure:howe-webs}) 
generalize the (quantum) Brauer category.

\makeautorefname{theoremmain}{Theorem} 

\subsection{Some further remarks}

\begin{remark}\label{remark:coideal}
The coideals $\coideal(\frakso_n)$ and $\coideal(\fraksp_n)$ 
are not Hopf subalgebras of $\quantumg(\frakgl_n)$, 
because they are not closed under comultiplication.
Hence, 
$\qRepco{\frakso_n}$ and $\qRepco{\fraksp_n}$
do not inherit a monoidal structure.
But since $\coideal(\frakso_n)$ and 
$\coideal(\fraksp_n)$ are  left coideal subalgebras 
of $\coideal(\frakgl_n)$, there is a
left action of $\Repq{n}$ on them.
In the web language this translates to the left-right 
partitioning as in \fullref{figure:webs}.

We stress that all these phenomena 
disappear if one de-quantizes.
\end{remark}

\begin{remark}\label{remark:sovso}
Let $\mathrm{O}_n$ be the orthogonal 
group, and $\vecrepnoq$ its vector representation.
Brauer \cite{Br} defined 
the \emph{Brauer algebra}, which 
surjects onto $\End_{\mathrm{O}_n}(\vecrepnoq^{\otimes k})$, 
for all $k$.
But, as Brauer observed (see also \cite[\S 5.1.3]{LZ1}), 
if one wants to replace $\mathrm{O}_n$ by the 
special orthogonal group $\mathrm{SO}_n$, 
then this is not true anymore since:
\smallskip
\begin{enumerate}[label=$\blacktriangleright$]

\setlength\itemsep{.15cm}

\item If $n$ is odd, then $\End_{\mathrm{O}_n}(\vecrepnoq^{\otimes k})
=\End_{\mathrm{SO}_n}(\vecrepnoq^{\otimes k})
$ for all $k$.

\item If $n$ is even, then $\End_{\mathrm{O}_n}(\vecrepnoq^{\otimes k})
=\End_{\mathrm{SO}_n}(\vecrepnoq^{\otimes k})
$ if and only if $n\geq 2k+1$.

\end{enumerate}
(Morally, one ``Brauer diagram generator'' 
is missing for $\mathrm{SO}_n$ if $n$ is even, 
see also \cite{Gr} and \cite{LZ2}.)
As a consequence, surjectivity fails 
in general for $\mathrm{SO}_n$ in type $\typeD$.

We will see in \fullref{sec:howe-duality} that the Brauer algebra 
is closely related to our web categories. Hence, to have surjectivity or 
fullness in general, we would have to add this extra 
Brauer diagram generator to our web categories. 
However, since this is not the main point of our 
construction, we prefer to avoid technicalities. Hence, we
obtain  
slightly weaker statements in type $\typeD$ 
than in types $\typeB\typeC$.
\end{remark}

\begin{remark}\label{remark:web-names}
The algebras on the right-hand side of
our $\qpar$-Howe dualities basically define the web categories, which
on the other hand correspond to the representation categories
of the algebras on the left-hand side.

Indeed, our webs have a representation theoretical 
incarnation via the functors $\dia$ from \fullref{figure:commute}. 
For example, the start and end dots as in \fullref{figure:webs}
correspond (in the de-quantized setting) to the fact that 
$\exteriornoq{2} \vecrepnoq$ (in type $\typeC$) 
respectively $\symmetricnoq{2} \vecrepnoq$ 
(in types $\typeB\typeD$) 
are not indecomposable, but contain a 
copy of the trivial module.  
\end{remark}

\subsection{Conventions} \label{sec-conventions-used}
We work over the ring $\CQZ$ of Laurent polynomial over the complex function field. 
We call $\qpar$ and $\zpar$ \emph{generic parameters}.
We also consider specializations of $\CQZ$ 
obtained by setting $\zpar$ equal to some non-zero value in 
the field $\CQ$. 
(The cases of overriding importance for us are 
the specializations of the form $\zpar=\pm\qpar^{\pm n}$ 
and there is no harm to think of $\zpar=\pm\qpar^{\pm n}$ 
throughout.)

In this setup, let $d_i\in\N$ 
and set $\qpar_i=\qpar^{d_i}$. 
The \emph{($\zpar$-)quantum number}, 
the \emph{quantum factorial},  
and the \emph{quantum binomial} 
are given by (here $s\in\Z$ and $t\in\N$)
\begin{gather}\label{eq:qnumbers-typeAD} 
\begin{aligned}
  &[s]_{i}  = \frac{\qpar_i^s - \qpar_i^{-s}}{\qpar^{\phantom{1}}_i - \qpar_i^{-1}}\in\CQ,\quad
  [\zpar;s]_{i}  = \frac{\zpar^{\phantom{1}}\qpar_i^s - \zpar^{-1}\qpar_i^{-s}}{\qpar^{\phantom{1}}_i - \qpar_i^{-1}}\in\CQZ,\\
  [t]_{i}! &= [t]_{i} [t-1]_{i} \dots [1]_{i}\in\CQ,\quad
  \qbinn{s}{t}_{i} = \frac{[s]_{i} [s-1]_{i} \dots [s-t+1]_{i}}{[t]_{i} [t-1]_{i} \dots [1]_{i}}\in\CQ.
\end{aligned}
\end{gather}
By convention, $[0]_{i}!=1=\qbin{s}{0}_{i}$. 
Note that $[0]_{i}=0=\qbin{0\leq s<t}{t}_{i}$ 
and $[-s]_{i}=-[s]_{i}$.
In case $d_i=1$ we write $[s]=[s]_{1}$ etc. for simplicity of notation.

\makeautorefname{definition}{Definitions}

Let $\ring$ be a ring.
All our categories are assumed to be additive and $\ring$-linear (except in 
\fullref{definition:gen-rel-conventions} and \ref{def:action-of-mon-category}),
and all our functors are assumed to be $\ring$-linear (and hence, additive). 
Which specific choice of $\ring$ we mean will be clear from the context.

\makeautorefname{definition}{Definition}

\subsection{Acknowledgements} 
We
like to thank Pedro Vaz for freely 
sharing his ideas and observations, some of 
which started this project.
We also thank Jonathan Comes, Michael Ehrig,
Hoel Queffelec,
Catharina Stroppel and Arik Wilbert for
some useful discussions. Special thanks to the referee
for helpful comments and suggestions.

The 
Hausdorff Center for Mathematics (HCM) in Bonn
and the GK 1821 in Freiburg 
sponsored some research visits of the authors during 
this project.
Both this support and
the hospitality during our visits are gratefully acknowledged.

D.T.\ likes to thank the 
wastebasket in his office for 
supporting a summer 
of calculations involving crazy quantum scalars -- most of which 
ended in utter chaos. Luckily, the symbol ${}^{\prime}$ 
came around at one point.

\section{A reminder on the \texorpdfstring{$\typeA$}{A}-web category}\label{sec-Awebs}
In this section we recall the construction of 
\emph{$\typeA$-webs} 
in the spirit of \cite{CKM}. (Note that, in contrast to 
\cite{CKM}, we use unoriented diagrams. This is due to the fact that 
the representations which we consider later in \fullref{sec-reps} are self-dual.)

\subsection{The \texorpdfstring{$\typeA$}{A}-web category}\label{subsec-typeA-webs}

We start by fixing conventions:

\begin{convention}\label{eq:dia-conventions}
For us the composition $\circ$ 
in diagram categories will be given
by vertical stacking, while the
monoidal product $\otimes$
will be given by horizontal juxtaposition, and 
identities are given by parallel 
vertical strands. We read our diagrams from bottom to top and left to right, 
e.g.:
\begin{gather}\label{eq:inter-law}
\overunder{(\mathrm{id}\otimes g)}{(f\otimes\mathrm{id})}{\text{\tiny$\circ$}}=
	\begin{tikzpicture}[anchorbase,scale=.25,tinynodes]
	\draw[uncolored, densely dotted] (-6,2.5) node[left] {$\circ$} to (4,2.5) node[right] {$\circ$};
	\draw[uncolored, densely dotted] (-1.25,6) node[above] {$\otimes$} to (-1.25,-1) node[below] {$\otimes$};
	\draw[colored] (-5,2) to (-5,3.75) node[right] {$\cdots$} to (-5,5) node[above] {$a$};
	\draw[colored] (-2.5,2) to (-2.5,5) node[above] {$b$};
	\draw[colored] (0,0) node[below] {$c$} to (0,1.25) node[right] {$\cdots$} to (0,3);
	\draw[colored] (2.5,0) node[below] {$d$} to (2.5,3);
	\draw (-5.5,2) rectangle (-2,.5);
	\draw (-.5,3) rectangle (3,4.5);
	\draw[colored] (-5,0) node[below] {$a$} to (-5,.25) node[right] {$\cdots$} to (-5,.5);
	\draw[colored] (-2.5,0) node[below] {$b$} to (-2.5,.5);
	\draw[colored] (2.5,4.5) to (2.5,5) node[above] {$d$};
	\draw[colored] (0,4.5) to (0,4.75) node[right] {$\cdots$} to (0,5) node[above] {$c$};
	\node at (-3.75,1.25) {$f$};
	\node at (1.25,3.75) {$g$};
	\end{tikzpicture}
=
    \begin{tikzpicture}[anchorbase,scale=.25,tinynodes]
	\draw[colored] (-5,3.25) to (-5,4.125) node[right] {$\cdots$} to (-5,5) node[above] {$a$};
	\draw[colored] (-2.5,3.25) to (-2.5,5) node[above] {$b$};
	\draw[colored] (0,0) node[below] {$c$} to (0,.625) node[right] {$\cdots$} to (0,1.75);
	\draw[colored] (2.5,0) node[below] {$d$} to (2.5,1.75);
	\draw (-5.5,1.75) rectangle (-2,3.25);
	\draw (-.5,1.75) rectangle (3,3.25);
	\draw[colored] (-5,0) node[below] {$a$} to (-5,.625) node[right] {$\cdots$} to (-5,1.75);
	\draw[colored] (-2.5,0) node[below] {$b$} to (-2.5,1.75);
	\draw[colored] (2.5,3.25) to (2.5,5) node[above] {$d$};
	\draw[colored] (0,3.25) to (0,4.125) node[right] {$\cdots$} to (0,5) node[above] {$c$};
	\node at (-3.75,2.5) {$f$};
	\node at (1.25,2.5) {$g$};
	\end{tikzpicture}
=
	\begin{tikzpicture}[anchorbase,scale=.25,tinynodes]
	\draw[uncolored, densely dotted] (6,2.5) node[right] {$\circ$} to (-4,2.5) node[left] {$\circ$};
	\draw[uncolored, densely dotted] (1.25,6) node[above] {$\otimes$} to (1.25,-1) node[below] {$\otimes$};
	\draw[colored] (2.5,2) to (2.5,3.75) node[right] {$\cdots$} to (2.5,5) node[above] {$c$};
	\draw[colored] (5,2) to (5,5) node[above] {$d$};
	\draw[colored] (0,0) node[below] {$b$} to (0,3);
	\draw[colored] (-2.5,0) node[below] {$a$} to (-2.5,1.25) node[right] {$\cdots$} to (-2.5,3);
	\draw (5.5,2) rectangle (2,.5);
	\draw (.5,3) rectangle (-3,4.5);
	\draw[colored] (2.5,0) node[below] {$c$} to (2.5,.25) node[right] {$\cdots$} to (2.5,.5);
	\draw[colored] (5,0) node[below] {$d$} to (5,.5);
	\draw[colored] (-2.5,4.5) to (-2.5,4.75) node[right] {$\cdots$} to (-2.5,5) node[above] {$a$};
	\draw[colored] (0,4.5) to (0,5) node[above] {$b$};
	\node at (3.75,1.25) {$g$};
	\node at (-1.25,3.75) {$f$};
	\end{tikzpicture}
=
\overunder{(f\otimes\mathrm{id})}{(\mathrm{id}\otimes g)}{\text{\tiny$\circ$}}
\end{gather}
Here $f,g$ are some morphisms in the categories in question. 
Moreover, as in the illustration above, we 
tend to omit the symbol $\otimes$ 
between objects.
\end{convention}

\begin{definition}\label{definition:gen-rel-conventions}
We say a (strict) monoidal category $\boldsymbol{\calM}=(\boldsymbol{\calM},\otimes,\one)$ 
is \emph{generated by two finite sets of objects $\mathtt{O}_{\boldsymbol{\calM}}$ 
and morphisms $\mathtt{M}_{\boldsymbol{\calM}}$} 
if any object, respectively morphism, is a $\otimes$, respectively a $\circ$-$\otimes$, composite 
of objects, respectively morphisms, from the two fixed sets 
(we allow the empty composition). 
If we further 
fix a set of relations $\mathtt{R}_{\boldsymbol{\calM}}$ among the morphisms of $\boldsymbol{\calM}$, then 
$\boldsymbol{\calM}$ is meant to be the quotient of the monoidal category freely 
generated by the fixed generators modulo these relations. See e.g. 
\cite[Section XII.1]{Ka1} for details.

Let $\ring$ be some ring. For a monoidal, $\ring$-linear 
category these notions are to be understood verbatim by 
enriching the morphism spaces formally in free $\ring$-modules.

The \emph{additive closure} of $\boldsymbol{\calM}$ means that 
we allow formal direct sums of objects from $\boldsymbol{\calM}$, and 
formal matrices of morphisms from $\boldsymbol{\calM}$. See e.g. \cite[Definition 3.2]{BN1} 
for details. (Beware that Bar-Natan uses a different nomenclature than we do.)
\end{definition}

\subsubsection{The monoidal category of \texorpdfstring{$\typeA$}{A}-webs}

\begin{definition}\label{definition:typeA-webs}
The \emph{$\typeA$-web category $\WebA$} is the 
additive closure of the (strict)
monoidal, $\CQ$-linear category generated 
by objects $a$ for $a \in \Z_{> 0}$ 
(note that the monoidal unity is given by the empty sequence $\varnothing$),
and morphisms
\begin{subequations}
\begingroup
\renewcommand{\theequation}{{\color{black}\textbf{A}}gen}
\begin{gather}\label{eq:Aweb-gens}
	\begin{tikzpicture}[anchorbase,scale=.25,tinynodes]
	\draw[colored] (0,0) node[below] {$a$} to (0,.5) node[left] {$a$} to [out=90, in=225] (1,2.25) to (1,3) node[right] {$a{+}b$} to (1,4) node[above] {$a{+}b$};
	\draw[colored] (2,0) node[below] {$b$} to (2,.5) node[right] {$b$} to [out=90, in=315] (1,2.25);
    \end{tikzpicture}
\colon
a\otimes b \to a+b
\quad\text{and}\quad
    \begin{tikzpicture}[anchorbase,scale=.25,tinynodes]
	\draw[colored] (0,4) node[above] {$a$} to (0,3.5) node[left] {$a$} to [out=270, in=135] (1,1.75) to (1,1) node[right] {$a{+}b$} to (1,0) node[below] {$a{+}b$};
	\draw[colored] (2,4) node[above] {$b$} to (2,3.5) node[right] {$b$} to [out=270, in=45] (1,1.75);
     \end{tikzpicture}
\colon
a+b \to a\otimes b,
\end{gather}
\endgroup
\end{subequations}
\addtocounter{equation}{-1}
(which we call 
\emph{merge} and \emph{split}),
modulo the relations:

\begin{subequations}
\begingroup
\renewcommand{\theequation}{{\color{black}\textbf{A}}\arabic{equation}}
\begin{enumerate}[label=$\vartriangleright$]

\setlength\itemsep{.15cm}

\item \emph{Associativity} and \emph{coassociativity}
\begin{gather}\label{eq:asso}
    \begin{tikzpicture}[anchorbase,scale=.25,tinynodes]
	\draw[colored] (0,0) node[below] {$a$} to [out=90, in=225] (1,2);
	\draw[colored] (2,0) node[below] {$b$} to [out=90, in=-45] (1,2);
	\draw[colored] (4,0) node[below] {$c$} to (4,2);
	\draw[colored] (1,2) to [out=90, in=225] (2.5,4);
	\draw[colored] (4,2) to [out=90, in=-45] (2.5,4);
	\draw[colored] (2.5,4) to (2.5,6) node[above] {$a{+}b{+}c$};
	\end{tikzpicture}
=
	\begin{tikzpicture}[anchorbase,scale=.25,tinynodes]
	\draw[colored] (0,0) node[below] {$c$} to [out=90, in=-45] (-1,2);
	\draw[colored] (-2,0) node[below] {$b$} to [out=90, in=225] (-1,2);
	\draw[colored] (-4,0) node[below] {$a$} to (-4,2);
	\draw[colored] (-1,2) to [out=90, in=-45] (-2.5,4);
	\draw[colored] (-4,2) to [out=90, in=225] (-2.5,4);
	\draw[colored] (-2.5,4) to (-2.5,6) node[above] {$a{+}b{+}c$};
	\end{tikzpicture}
\quad\text{and}\quad
	\begin{tikzpicture}[anchorbase,scale=.25,tinynodes]
	\draw[colored] (0,0) node[above] {$a$} to [out=270, in=135] (1,-2);
	\draw[colored] (2,0) node[above] {$b$} to [out=270, in=45] (1,-2);
	\draw[colored] (4,0) node[above] {$c$} to (4,-2);
	\draw[colored] (1,-2) to [out=270, in=135] (2.5,-4);
	\draw[colored] (4,-2) to [out=270, in=45] (2.5,-4);
	\draw[colored] (2.5,-4) to (2.5,-6) node[below] {$a{+}b{+}c$};
	\end{tikzpicture}
=
	\begin{tikzpicture}[anchorbase,scale=.25,tinynodes]
	\draw[colored] (0,0) node[above] {$c$} to [out=270, in=45] (-1,-2);
	\draw[colored] (-2,0) node[above] {$b$} to [out=270, in=135] (-1,-2);
	\draw[colored] (-4,0) node[above] {$a$} to (-4,-2);
	\draw[colored] (-1,-2) to [out=270, in=45] (-2.5,-4);
	\draw[colored] (-4,-2) to [out=270, in=135] (-2.5,-4);
	\draw[colored] (-2.5,-4) to (-2.5,-6) node[below] {$a{+}b{+}c$};
	\end{tikzpicture}
\end{gather}
	
\item The \emph{(thin) square switch}
	
\begin{gather}\label{eq:square}
	\begin{tikzpicture}[anchorbase,scale=.25,tinynodes]
	\draw[colored] (0,0) node[below] {$a$} to (0,4) node[above] {$a$};
	\draw[colored] (2.5,0) node[below] {$b$} to (2.5,4) node[above] {$b$};
	\draw[uncolored] (0,1) to (1.25,1.25) to (2.5,1.5);
	\draw[uncolored] (0,3) to (1.25,2.75) to (2.5,2.5);
	\end{tikzpicture}
=
	\begin{tikzpicture}[anchorbase,scale=.25,tinynodes]
	\draw[colored] (0,0) node[below] {$a$} to (0,4) node[above] {$a$};
	\draw[colored] (2.5,0) node[below] {$b$} to (2.5,4) node[above] {$b$};
	\draw[uncolored] (2.5,1) to (1.25,1.25) to (0,1.5);
	\draw[uncolored] (2.5,3) to (1.25,2.75) to (0,2.5);
	\end{tikzpicture}
+\qbracket{a-b}
	\begin{tikzpicture}[anchorbase,scale=.25,tinynodes]
	\draw[colored] (0,0) node[below] {$a$} to (0,4) node[above] {$a$};
	\draw[colored] (2.5,0) node[below] {$b$} to (2.5,4) node[above] {$b$};
	\end{tikzpicture}
\end{gather}
\end{enumerate}
\endgroup
\end{subequations}
\end{definition}

\addtocounter{equation}{-1}

Every diagram representing a morphism in $\WebA$
will be called an \emph{$\typeA$-web}. Note that 
the interchange law \eqref{eq:inter-law} allows us to 
use topological height moves among $\typeA$-webs, as well as 
other topological manipulations which keep an upward-directedness 
of $\typeA$-webs (i.e. no critical points, when 
one sees $\typeA$-webs as embedded graphs), and we do so 
in the following. In fact, we simplified our illustrations 
by sometimes drawing them in a topological fashion, a 
shorthand which we will use throughout. However, we stress 
that all our web calculi are rigidly built from generating sets.

\begin{convention}\label{eq:dia-conventions-thickness}
We call the label of an edge the \emph{thickness} of the edge in question. 
Although we do not allow edges labeled $0$ or negative labeled edges, it is convenient 
in illustrations to allow edges which are potentially zero -- these 
are to be erased to obtain the corresponding $\typeA$-web -- or negative -- which 
set the $\typeA$-web to be zero.
Edges labeled $1$, called \emph{thin}, will play an important role 
and we illustrate them thinner than 
arbitrary labeled edges, cf.\ \eqref{eq:square}. 
Moreover, edges of thickness $2$ also play 
a special role and are displayed slightly thicker 
than thin edges.
We sometimes omit the 
edge labels: if they are omitted, then they 
can be recovered from the illustrated ones, 
or are $1$ or $2$ whenever they correspond to thinner edges. 
\end{convention}

Later it will be convenient to consider $\WebA$ 
as a $\CQZ$-linear category, denoted by $\WebAz$, which can be easily 
achieved via scalar extension.

\begin{remark}\label{remark:no-thick-things}
Note that the \emph{thick square switches}, i.e.
\begin{gather}\label{eq:thick-square}
	\begin{tikzpicture}[anchorbase,scale=.25,tinynodes]
	\draw[colored] (0,0) node[below] {$a$} to (0,4) node[above] {$\begin{matrix}a{-}c\\{+}d\end{matrix}$};
	\draw[colored] (2.5,0) node[below] {$b$} to (2.5,4) node[above] {$\begin{matrix}b{+}c\\{-}d\end{matrix}$};
	\draw[colored] (0,1) to (1.25,1.25) node[below] {$c$} to (2.5,1.5);
	\draw[colored] (0,3) to (1.25,2.75) node[above] {$d$} to (2.5,2.5);
	\end{tikzpicture}
=
{\textstyle\sum_{e}}
\qbin{a-b+c-d}{e}
	\begin{tikzpicture}[anchorbase,scale=.25,tinynodes]
	\draw[colored] (2.5,0) node[below] {$b$} to (2.5,4) node[above] {$\begin{matrix}b{+}c\\{-}d\end{matrix}$};
	\draw[colored] (0,0) node[below] {$a$} to (0,4) node[above] {$\begin{matrix}a{-}c\\{+}d\end{matrix}$};
	\draw[colored] (2.5,1) to (1.25,1.25) node[below] {$d{-}e$} to (0,1.5);
	\draw[colored] (2.5,3) to (1.25,2.75) node[above] {$c{-}e$} to (0,2.5);
	\end{tikzpicture}
\quad\text{and}\quad
	\begin{tikzpicture}[anchorbase,scale=.25,tinynodes]
	\draw[colored] (2.5,0) node[below] {$b$} to (2.5,4) node[above] {$\begin{matrix}b{-}c\\{+}d\end{matrix}$};
	\draw[colored] (0,0) node[below] {$a$} to (0,4) node[above] {$\begin{matrix}a{+}c\\{-}d\end{matrix}$};
	\draw[colored] (2.5,1) to (1.25,1.25) node[below] {$c$} to (0,1.5);
	\draw[colored] (2.5,3) to (1.25,2.75) node[above] {$d$} to (0,2.5);
	\end{tikzpicture}
=
{\textstyle\sum_{e}}
\qbin{-a+b-c+d}{e}
	\begin{tikzpicture}[anchorbase,scale=.25,tinynodes]
	\draw[colored] (0,0) node[below] {$a$} to (0,4) node[above] {$\begin{matrix}a{+}c\\{-}d\end{matrix}$};
	\draw[colored] (2.5,0) node[below] {$b$} to (2.5,4) node[above] {$\begin{matrix}b{-}c\\{+}d\end{matrix}$};
	\draw[colored] (0,1) to (1.25,1.25) node[below] {$d{-}e$} to (2.5,1.5);
	\draw[colored] (0,3) to (1.25,2.75) node[above] {$c{-}e$} to (2.5,2.5);
	\end{tikzpicture}
\end{gather}
where $e\in\N$,
as well as the \emph{divided power collapsing}, i.e.
\begin{gather}\label{eq:EE-square}
	\begin{tikzpicture}[anchorbase,scale=.25,tinynodes]
	\draw[colored] (0,0) node[below] {$a$} to (0,4) node[above] {$\begin{matrix}a{+}c\\{+}d\end{matrix}$};
	\draw[colored] (2.5,0) node[below] {$b$} to (2.5,4) node[above] {$\begin{matrix}b{-}c\\{-}d\end{matrix}$};
	\draw[colored] (0,1.5) to (1.25,1.25) node[below] {$c$} to (2.5,1);
	\draw[colored] (0,3) to (1.25,2.75) node[above] {$d$} to (2.5,2.5);
	\end{tikzpicture}
=
\qbin{c+d}{d}
	\begin{tikzpicture}[anchorbase,scale=.25,tinynodes]
	\draw[colored] (0,0) node[below] {$a$} to (0,4) node[above] {$\begin{matrix}a{+}c\\{+}d\end{matrix}$};
	\draw[colored] (2.5,0) node[below] {$b$} to (2.5,4) node[above] {$\begin{matrix}b{-}c\\{-}d\end{matrix}$};
	\draw[colored] (0,2.25) to (1.25,2) node[above] {$c{+}d$} to (2.5,1.75);
	\end{tikzpicture}
\quad\text{and}\quad
	\begin{tikzpicture}[anchorbase,scale=.25,tinynodes]
	\draw[colored] (0,0) node[below] {$a$} to (0,4) node[above] {$\begin{matrix}a{-}c\\{-}d\end{matrix}$};
	\draw[colored] (2.5,0) node[below] {$b$} to (2.5,4) node[above] {$\begin{matrix}b{+}c\\{+}d\end{matrix}$};
	\draw[colored] (0,1) to (1.25,1.25) node[below] {$c$} to (2.5,1.5);
	\draw[colored] (0,2.5) to (1.25,2.75) node[above] {$d$} to (2.5,3);
	\end{tikzpicture}
=
\qbin{c+d}{d}
	\begin{tikzpicture}[anchorbase,scale=.25,tinynodes]
	\draw[colored] (0,0) node[below] {$a$} to (0,4) node[above] {$\begin{matrix}a{-}c\\{-}d\end{matrix}$};
	\draw[colored] (2.5,0) node[below] {$b$} to (2.5,4) node[above] {$\begin{matrix}b{+}c\\{+}d\end{matrix}$};
	\draw[colored] (0,1.75) to (1.25,2) node[above] {$c{+}d$} to (2.5,2.25);
	\end{tikzpicture}
\end{gather}
can be deduced from the above relations
since we work over $\CQ$. An example is:
\begin{gather*}
\begin{aligned}
	\begin{tikzpicture}[anchorbase,scale=.25,tinynodes]
	\draw[colored] (0,0) node[below] {$a$} to (0,4) node[above] {$a{-}1$};
	\draw[colored] (2.5,0) node[below] {$b$} to (2.5,4) node[above] {$b{+}1$};
	\draw[ccolored] (0,1) to (2.5,1.5);
	\draw[uncolored] (0,3) to (1.25,2.75) to (2.5,2.5);
	\end{tikzpicture}
&\overset{\eqref{eq:square}}{=}
\tfrac{1}{[2]}
	\begin{tikzpicture}[anchorbase,scale=.25,tinynodes]
	\draw[colored] (0,0) node[below] {$a$} to (0,4) node[above] {$a{-}1$};
	\draw[colored] (2.5,0) node[below] {$b$} to (2.5,4) node[above] {$b{+}1$};
	\draw[ccolored] (0,.5) to (1.25,.75);
	\draw[ccolored] (1.25,1.75) to (2.5,2);
	\draw[uncolored] (1.25,.75) to [out=135, in=225] (1.25,1.75);
	\draw[uncolored] (1.25,.75) to [out=45, in=315] (1.25,1.75);
	\draw[uncolored] (0,3) to (1.25,2.75) to (2.5,2.5);
	\end{tikzpicture}
\overset{\eqref{eq:asso}}{=}
\tfrac{1}{[2]}
	\begin{tikzpicture}[anchorbase,scale=.25,tinynodes]
	\draw[colored] (0,0) node[below] {$a$} to (0,4) node[above] {$a{-}1$};
	\draw[colored] (2.5,0) node[below] {$b$} to (2.5,4) node[above] {$b{+}1$};
	\draw[uncolored] (0,.5) to (2.5,1);
	\draw[uncolored] (0,1.25) to (2.5,1.75);
	\draw[uncolored] (0,3) to (1.25,2.75) to (2.5,2.5);
	\end{tikzpicture}
\overset{\eqref{eq:square}}{=}
\tfrac{1}{[2]}
	\begin{tikzpicture}[anchorbase,scale=.25,tinynodes]
	\draw[colored] (0,0) node[below] {$a$} to (0,4) node[above] {$a{-}1$};
	\draw[colored] (2.5,0) node[below] {$b$} to (2.5,4) node[above] {$b{+}1$};
	\draw[uncolored] (0,.75) to (2.5,1.25);
	\draw[uncolored] (0,2.25) to (2.5,1.75);
	\draw[uncolored] (0,2.75) to (2.5,3.25);
	\end{tikzpicture}
+\tfrac{[a-b-2]}{[2]}
	\begin{tikzpicture}[anchorbase,scale=.25,tinynodes]
	\draw[colored] (0,0) node[below] {$a$} to (0,4) node[above] {$a{-}1$};
	\draw[colored] (2.5,0) node[below] {$b$} to (2.5,4) node[above] {$b{+}1$};
	\draw[uncolored] (0,1.75) to (2.5,2.25);
	\end{tikzpicture}
\\
\overset{\eqref{eq:square}}{=}&
\tfrac{1}{[2]}
	\begin{tikzpicture}[anchorbase,scale=.25,tinynodes]
	\draw[colored] (0,0) node[below] {$a$} to (0,4) node[above] {$a{-}1$};
	\draw[colored] (2.5,0) node[below] {$b$} to (2.5,4) node[above] {$b{+}1$};
	\draw[uncolored] (0,1) to (2.5,.5);
	\draw[uncolored] (0,1.75) to (2.5,2.25);
	\draw[uncolored] (0,2.5) to (1.25,2.75) to (2.5,3);
	\end{tikzpicture}
+\qbracket{a-b+1}
	\begin{tikzpicture}[anchorbase,scale=.25,tinynodes]
	\draw[colored] (0,0) node[below] {$a$} to (0,4) node[above] {$a{-}1$};
	\draw[colored] (2.5,0) node[below] {$b$} to (2.5,4) node[above] {$b{+}1$};
	\draw[uncolored] (0,1.75) to (2.5,2.25);
	\end{tikzpicture}
\overset{\eqref{eq:EE-square}}{=}
	\begin{tikzpicture}[anchorbase,scale=.25,tinynodes]
	\draw[colored] (0,0) node[below] {$a$} to (0,4) node[above] {$a{-}1$};
	\draw[colored] (2.5,0) node[below] {$b$} to (2.5,4) node[above] {$b{+}1$};
	\draw[uncolored] (0,1) to (1.25,1.25) to (2.5,1.5);
	\draw[ccolored] (0,3) to (2.5,2.5);
	\end{tikzpicture}
+\qbracket{a-b+1}
	\begin{tikzpicture}[anchorbase,scale=.25,tinynodes]
	\draw[colored] (0,0) node[below] {$a$} to (0,4) node[above] {$a{-}1$};
	\draw[colored] (2.5,0) node[below] {$b$} to (2.5,4) node[above] {$b{+}1$};
	\draw[uncolored] (0,1.75) to (2.5,2.25);
	\end{tikzpicture}
\end{aligned}
\end{gather*}
The first step here is called \emph{explosion}.
This is a general feature for (many) web categories: 
the web calculus is basically determined by what happens 
in the case of thin labels,
as the thick ones can be reduced to the thin ones via 
explosion. We will see this phenomenon 
turning up later on as well.

Note also that the so-called 
\emph{digon removals}, i.e.\ 
\begin{gather}\label{eq:digon}
\begin{tikzpicture}[anchorbase,scale=.25,tinynodes]
	\draw[colored] (0,0) node[below] {$a{+}b$} to (0,1);
	\draw[colored] (0,3) to (0,4) node[above] {$a{+}b$};
	\draw[colored] (0,1) to [out=135, in=270] (-.75,2) node[left] {$a$} to [out=90, in=225] (0,3);
	\draw[colored] (0,1) to [out=45, in=270] (.75,2) node[right] {$b$} to [out=90, in=315] (0,3);
	\end{tikzpicture}
=
\qbin{a+b}{b}
	\begin{tikzpicture}[anchorbase,scale=.25,tinynodes]
	\draw[colored] (0,0) node[below] {$a{+}b$} to (0,4) node[above] {$a{+}b$};
	\end{tikzpicture}
\end{gather}
are special cases of the 
square switches.
\end{remark}

\begin{remark}\label{remark:web-Serre}
By one of the main results of \cite{CKM}, 
we have a list of additional relations 
which we call the \emph{$\typeA$-web Serre relations}. 
We just give a blueprint example (cf.\ \cite[Lemma 2.2.1]{CKM}):
\[
\qbracket{2}
      \begin{tikzpicture}[anchorbase,xscale=.3,yscale=.25,tinynodes]
      \draw[colored] (0,0) node[below] {$a$} to (0,6) node [above] {$a{+}2$};
      \draw[colored] (2,0) node[below] {$b$} to (2,6) node [above] {$b{-}1$};
      \draw[colored] (4,0) node[below] {$c$} to (4,6) node [above] {$c{-}1$};
      \draw[uncolored] (2,1.5) to (0,2.5);
      \draw[uncolored] (2,3) to (4,2);
      \draw[uncolored] (2,4.5) to (0,5.5);
      \end{tikzpicture}
=
	  \begin{tikzpicture}[anchorbase,xscale=.3,yscale=.25,tinynodes]
      \draw[colored] (0,0) node[below] {$a$} to (0,6) node [above] {$a{+}2$};
      \draw[colored] (2,0) node[below] {$b$} to (2,6) node [above] {$b{-}1$};
      \draw[colored] (4,0) node[below] {$c$} to (4,6) node [above] {$c{-}1$};
      \draw[uncolored] (2,3) to (0,4);
      \draw[uncolored] (2,1.5) to (4,.5);
      \draw[uncolored] (2,4.5) to (0,5.5);
      \end{tikzpicture}
+
      \begin{tikzpicture}[anchorbase,xscale=.3,yscale=.25,tinynodes]
      \draw[colored] (0,0) node[below] {$a$} to (0,6) node [above] {$a{+}2$};
      \draw[colored] (2,0) node[below] {$b$} to (2,6) node [above] {$b{-}1$};
      \draw[colored] (4,0) node[below] {$c$} to (4,6) node [above] {$c{-}1$};
      \draw[uncolored] (2,1.5) to (0,2.5);
      \draw[uncolored] (2,4.5) to (4,3.5);
      \draw[uncolored] (2,3) to (0,4);
      \end{tikzpicture}
\]
Since we work over $\CQ$, thick versions of these hold as well. 
We leave it to the reader to write them down, 
keeping in mind that they are ``web versions'' of the 
higher order Serre relations \eqref{eq:higherSerre} of type 
$\typeA$. 
(We refer to these specifying $s,t$ as therein.)
\end{remark}

\subsubsection{The braiding}

Recall that $\WebA$ is a braided  
category. There is some freedom in the choice of 
scaling of the braiding. For us the most convenient choice 
for \emph{thin overcrossings} (left crossing in \eqref{eq:braiding}) 
and \emph{thin undercrossing} (right crossing  in \eqref{eq:braiding}) is:
\begin{equation}\label{eq:braiding}
    \begin{tikzpicture}[anchorbase,scale=.25,tinynodes]
	\draw[uncolored] (2,0) node[below] {$1$} .. controls ++(0,2) and  ++(0,-2) .. ++(-2,4) node[above] {$1$}; 
	\draw[uncolored,cross line] (0,0) node[below] {$1$} .. controls ++(0,2) and  ++(0,-2) .. ++(2,4) node[above] {$1$}; 
    \end{tikzpicture} 
=-\qpar^{-1}
    \begin{tikzpicture}[anchorbase,scale=.25,tinynodes]
	\draw[uncolored] (2,0) node[below] {$1$} -- ++(0,4) node[above] {$1$}; 
	\draw[uncolored] (0,0) node[below] {$1$} --  ++(0,4) node[above] {$1$}; 
    \end{tikzpicture}
+
    \begin{tikzpicture}[anchorbase,scale=.25,tinynodes]
	\draw[uncolored] (0,0) node[below] {$1$} .. controls ++(0,1) and ++(-0.5,-0.5) .. ++(1,1.5) .. controls ++(0.5,-0.5) and ++(0,1) .. ++(1,-1.5) node[below] {$1$}; 
	\draw[ccolored] (1,1.5) to (1,2.5);
	\draw[uncolored] (0,4) node[above] {$1$} .. controls ++(0,-1) and ++(-0.5,0.5) .. ++(1,-1.5) .. controls ++(0.5,0.5) and ++(0,-1) .. ++(1,1.5) node[above] {$1$}; 
    \end{tikzpicture}
\quad\text{and}\quad
    \begin{tikzpicture}[anchorbase,scale=.25,tinynodes]
	\draw[uncolored] (0,0) node[below] {$1$} .. controls ++(0,2) and  ++(0,-2) .. ++(2,4) node[above] {$1$}; 
	\draw[uncolored,cross line] (2,0) node[below] {$1$} .. controls ++(0,2) and  ++(0,-2) .. ++(-2,4) node[above] {$1$}; 
    \end{tikzpicture} 
=-\qpar
    \begin{tikzpicture}[anchorbase,scale=.25,tinynodes]
	\draw[uncolored] (2,0) node[below] {$1$} -- ++(0,4) node[above] {$1$}; 
	\draw[uncolored] (0,0) node[below] {$1$} --  ++(0,4) node[above] {$1$}; 
    \end{tikzpicture}
+
    \begin{tikzpicture}[anchorbase,scale=.25,tinynodes]
	\draw[uncolored] (0,0) node[below] {$1$} .. controls ++(0,1) and ++(-0.5,-0.5) .. ++(1,1.5) .. controls ++(0.5,-0.5) and ++(0,1) .. ++(1,-1.5) node[below] {$1$}; 
	\draw[ccolored] (1,1.5) to (1,2.5);
	\draw[uncolored] (0,4) node[above] {$1$} .. controls ++(0,-1) and ++(-0.5,0.5) .. ++(1,-1.5) .. controls ++(0.5,0.5) and ++(0,-1) .. ++(1,1.5) node[above] {$1$}; 
    \end{tikzpicture}
\end{equation}
Recall that a braiding on $\WebA$ is, via explosion, uniquely 
determined by specifying \eqref{eq:braiding} 
(see e.g.\ \cite[Lemma 5.12]{QS}). That is, we also get thick over- and undercrossings 
and one can inductively compute how these 
are expressed in terms of the $\typeA$-web generators
from \eqref{eq:Aweb-gens}.

\begin{remark}\label{remark:web-symmetries}
The category $\WebA$ 
has a $\qpar$-anti-linear (that is, flipping $\qpar\leftrightarrow\qpar^{-1}$) 
involution $\Psi$ given by switching the crossings and an 
anti-involution $\omega$ given by 
taking the vertical mirror of a diagram.
In particular, it suffices to give 
relations involving one type of crossing, 
and we will do so below.
\end{remark}

We remark that the naturality of the braiding is 
equivalent to the following \emph{pitchfork} relations, 
which hold for all values of $a$, $b$ and $c$:
\begin{gather}\label{eq:pitchfork}
    \begin{tikzpicture}[anchorbase,scale=.25,tinynodes]
	\draw[colored] (0,-4) node[below] {$b$} to (0,-3.5) to [out=90, in=225] (1,-1.75) to (1,-1) to (1,0) node[above] {$b{+}c$};
	\draw[colored] (2,-4) node[below] {$c$} to (2,-3.5) to [out=90, in=315] (1,-1.75);
	\draw[colored, cross line thick] (3,0) node[above] {$a$} to (3,-.75) to (-1,-1.25) to (-1,-4) node[below] {$a$};
    \end{tikzpicture}
=
	\begin{tikzpicture}[anchorbase,scale=.25,tinynodes]
	\draw[colored] (0,-4) node[below] {$b$} to (0,-3.5) to [out=90, in=225] (1,-1.75) to (1,-1) to (1,0) node[above] {$b{+}c$};
	\draw[colored] (2,-4) node[below] {$c$} to (2,-3.5) to [out=90, in=315] (1,-1.75);
	\draw[colored, cross line thick] (3,0) node[above] {$a$} to (3,-2.75) to (-1,-3.25) to (-1,-4) node[below] {$a$};
    \end{tikzpicture}
\quad\text{and}\quad
	\begin{tikzpicture}[anchorbase,scale=.25,tinynodes]
	\draw[colored] (0,4) node[above] {$b$} to (0,3.5) to [out=270, in=135] (1,1.75) to (1,1) to (1,0) node[below] {$b{+}c$};
	\draw[colored] (2,4) node[above] {$c$} to (2,3.5) to [out=270, in=45] (1,1.75);
	\draw[colored, cross line thick] (-1,0) node[below] {$a$} to (-1,.75) to (3,1.25) to (3,4) node[above] {$a$};
    \end{tikzpicture}
=
	\begin{tikzpicture}[anchorbase,scale=.25,tinynodes]
	\draw[colored] (0,4) node[above] {$b$} to (0,3.5) to [out=270, in=135] (1,1.75) to (1,1) to (1,0) node[below] {$b{+}c$};
	\draw[colored] (2,4) node[above] {$c$} to (2,3.5) to [out=270, in=45] (1,1.75);
	\draw[colored, cross line thick] (-1,0) node[below] {$a$} to (-1,2.75) to (3,3.25) to (3,4) node[above] {$a$};
    \end{tikzpicture}
\end{gather}

We additionally need the following relations:

\begin{lemma}\label{lemma:ttwist}
For all $a,b,c$
the 
\emph{trivalent twists} hold 
in $\WebA$:
\begin{gather}\label{eq:ttwist}
	\begin{tikzpicture}[anchorbase,scale=.25,tinynodes]
	\draw[colored] (0,-2.5) to [out=90, in=225] (1,-1.25) to (1,-1) to (1,0) node[above] {$a{+}b$};
	\draw[colored] (2,-2.5) to [out=90, in=315] (1,-1.25);
	\draw[colored] (0,-2.5) to [out=270, in=90] (2,-4) node[below] {$b$};
	\draw[colored, cross line thick] (2,-2.5) to [out=270, in=90] (0,-4) node[below] {$a$};
	\draw[colored] (0,-2.48) to (0,-2.52);
	\draw[colored] (2,-2.48) to (2,-2.52);
    \end{tikzpicture}
=\qpar^{ab}
    \begin{tikzpicture}[anchorbase,scale=.25,tinynodes]
	\draw[colored] (0,-4) node[below] {$a$} to (0,-3.5) to [out=90, in=225] (1,-1.75) to (1,-1) to (1,0) node[above] {$a{+}b$};
	\draw[colored] (2,-4) node[below] {$b$} to (2,-3.5) to [out=90, in=315] (1,-1.75);
    \end{tikzpicture}
\quad\text{and}\quad
	\begin{tikzpicture}[anchorbase,scale=.25,tinynodes]
	\draw[colored] (0,2.5) to [out=270, in=135] (1,1.25) to (1,1) to (1,0) node[below] {$a{+}b$};
	\draw[colored] (2,2.5) to [out=270, in=45] (1,1.25);
	\draw[colored] (2,2.5) to [out=90, in=270] (0,4) node[above] {$a$};
	\draw[colored, cross line thick] (0,2.5) to [out=90, in=270] (2,4) node[above] {$b$};
	\draw[colored] (0,2.48) to (0,2.52);
	\draw[colored] (2,2.48) to (2,2.52);
    \end{tikzpicture}
=\qpar^{ab}
    \begin{tikzpicture}[anchorbase,scale=.25,tinynodes]
	\draw[colored] (0,4) node[above] {$a$} to (0,3.5) to [out=270, in=135] (1,1.75) to (1,1) to (1,0) node[below] {$a{+}b$};
	\draw[colored] (2,4) node[above] {$b$} to (2,3.5) to [out=270, in=45] (1,1.75);
     \end{tikzpicture}
\end{gather}
\end{lemma}

\begin{proof}
These relations are easily verified inductively by using explosion.
\end{proof}

\section{The \texorpdfstring{$\typeBDC$}{cup}-web category}\label{sec-Dwebs}
Next, we define 
a web category which, as we will see, will describe exterior $\typeB\typeD$-webs 
as well as symmetric $\typeC$-webs. 
We call its morphisms \emph{$\typeBDC$-webs}.

\subsection{Categories with a monoidal action}\label{subsec-monoidal-action}

We will 
define webs of types $\typeB\typeC\typeD$ as morphisms of 
categories with a left monoidal action of the monoidal 
category $\WebA$, as formalized by the following definition, 
following \cite[Section 2]{HO1} or \cite[Sections 7.1 and 7.3]{EGNO}.

\begin{definition}\label{def:action-of-mon-category}
Let $\boldsymbol{\calM}=(\boldsymbol{\calM},\otimes,\one)$ be a 
(strict) monoidal category, and $\boldsymbol{\calC}$ be a category. 
A \emph{(left) action of} $\boldsymbol{\calM}$ \emph{on} $\boldsymbol{\calC}$ is a 
bifunctor 
$\otimes \colon \boldsymbol{\calM} \times \boldsymbol{\calC} \rightarrow \boldsymbol{\calC}$ 
with 
natural isomorphisms $(X \otimes Y) \otimes C \cong X \otimes (Y \otimes C)$ 
and $\one \otimes C \cong \one\cong C\otimes\one$ for $X,Y \in \boldsymbol{\calM}$, $C \in \boldsymbol{\calC}$ satisfying 
the usual coherence conditions 
(see e.g.\ \cite[Section 2]{HO1}, \cite[Definition 7.11]{EGNO} 
or \cite[Definition IV.4.7]{Wei}). We will 
then say that $\boldsymbol{\calC}$ is an $\boldsymbol{\calM}$\emph{-category}.
 
In case $\boldsymbol{\calM}$ and $\boldsymbol{\calC}$ are 
both $\ring$-linear over a ring $\ring$, we 
additionally assume that $\otimes$ is $\ring$-bilinear on morphisms.

The \emph{additive closure} of an $\boldsymbol{\calM}$-category is to be understood 
verbatim as in \fullref{definition:gen-rel-conventions}, where we additionally extend
the action of $\boldsymbol{\calM}$ to direct sums component-wise.

Without assuming that $\calM$ has generators/relations: We say $\boldsymbol{\calC}$ \emph{is generated by 
two finite sets $\mathtt{O}_{\boldsymbol{\calC}}$ of objects and $\mathtt{M}_{\boldsymbol{\calC}}$ of morphisms}
if every object is of the form $X \otimes C$, where $X \in \boldsymbol{\calM}$ and $C$ 
is a $\otimes$ composite of objects from $\mathtt{O}_{\boldsymbol{\calC}}$, and similarly for morphisms.
If we further 
fix a set of relations $\mathtt{R}_{\boldsymbol{\calC}}$ among the morphisms of $\boldsymbol{\calC}$, then 
$\boldsymbol{\calC}$ is meant to be the quotient of the $\boldsymbol{\calM}$-category freely 
generated by the fixed generators modulo the left $\boldsymbol{\calM}$-ideal spanned by these 
relations. (This definition can be spelled out in details analogously to 
e.g. \cite[Section XII.1]{Ka1}.)
\end{definition}

\newpage

\subsection{The diagrammatic \texorpdfstring{$\typeBDC$}{cup}-web category}\label{subsec-diacat-D}

\subsubsection{\texorpdfstring{$\typeBDC$}{cup}-webs}

In this section we work over $\CQZ$, if not stated otherwise. 
For the definition of the quantum numbers see \eqref{eq:qnumbers-typeAD}.

\begin{definition}\label{definition:typeBDwebs}
The \emph{$\typeBDC$-web category} $\WebDz$ is the additive closure 
of the $\CQZ$-li\-near $\WebAz$-category 
generated by the object $\varnothing$
and by the 
\emph{cup} and \emph{cap morphisms}
\begin{subequations}
\begingroup
\renewcommand{\theequation}{$\typeBDC$gen}
\begin{equation}\label{eq:cup-cap}
	\begin{tikzpicture}[anchorbase,scale=.25, tinynodes]
	\draw[uncolored] (-1,3) node[above] {$1$} to [out=270, in=180] (0,2) to [out=0, in=270] (1,3) node[above] {$1$};
	\node at (0,0) {$\phantom{.}$};
    \end{tikzpicture}
\colon \varnothing\to 1\otimes 1
\quad\text{and}\quad
    \begin{tikzpicture}[anchorbase,scale=.25, tinynodes]
	\draw[uncolored] (-1,-3) node[below] {$1$} to [out=90, in=180] (0,-2) to [out=0, in=90] (1,-3) node[below] {$1$};
	\node at (0,0) {$\phantom{.}$};
    \end{tikzpicture}
\colon 1\otimes 1\to \varnothing,
\end{equation}
\endgroup
\end{subequations}
\addtocounter{equation}{-1}
modulo the following relations:
\begin{subequations}
\begingroup
\renewcommand{\theequation}{$\typeBDC$\arabic{equation}}
\begin{enumerate}[label=$\vartriangleright$]

\setlength\itemsep{.15cm}

\item The \emph{circle removal}
\begin{equation}\label{eq:circle}
      \begin{tikzpicture}[anchorbase,scale=.25, tinynodes]
      \draw[uncolored] (-1,2.5) arc (-180:180:1cm);
      \end{tikzpicture}
\,=\qbracket{\zpar;0}.
\end{equation}

\item The \emph{bubble removal}
\begin{equation}\label{eq:bubble}
      \begin{tikzpicture}[anchorbase,scale=.25, tinynodes]
      \draw[uncolored] (0,0) node[below] {$1$} -- ++(0,2.4)  arc (180:0:1cm) -- node[right] {$\phantom{1}$} ++(0,-0.4) arc (0:-180:1cm) -- ++(0,2.4) node[above] {$1$};
      \draw[ccolored] (0,1.6) -- ++(0,1.1);
      \end{tikzpicture} 
=\qbracket{\zpar;-1}
      \begin{tikzpicture}[anchorbase,scale=.25, tinynodes]
      \draw[uncolored] (0,0) node[below] {$1$} -- ++(0,4.4) node[above] {$1$};
      \end{tikzpicture} 
\end{equation}
   
\item The \emph{lasso move}
\begin{equation}\label{eq:lasso}
      \begin{tikzpicture}[anchorbase,scale=.25,tinynodes]
      \draw[uncolored] (0,0) node[below] {$1$} -- ++(0,7) node [above] {$1$};
      \draw[uncolored] (2,0) node[below] {$1$} -- ++(0,7) node [above] {$1$};
      \draw[ccolored] (0,3) -- ++(0,1);
      \draw[ccolored] (2,3) -- ++(0,1);
      \draw[uncolored] (0,3) -- ++(1,-1);
      \draw[uncolored, cross line] (1,2) .. controls ++(1,-1) and ++ (0,-1.5) ..  ++(3,-1) node[below right=-0.7ex] {$\phantom{1}$};
      \draw[uncolored] (4,1) .. controls ++(0,0.8) and ++(1,-1) .. ++(-2,2);
      \draw[uncolored] (0,4) -- ++(1,1);
      \draw[uncolored, cross line] (1,5) .. controls ++(1,1) and ++ (0,1.5) ..  ++(3,1) node[above right=-0.7ex] {$\phantom{1}$};
      \draw[uncolored] (4,6) .. controls ++(0,-0.8) and ++(1,1) .. ++(-2,-2);
      \end{tikzpicture}
=
      \begin{tikzpicture}[anchorbase,scale=.25,tinynodes]
      \draw[uncolored] (0,0) node[below] {$1$} -- ++(0,.4)  arc (180:0:1cm)  -- ++(0,-.4) node[below] {$1$};
      \draw[uncolored] (0,4) node[above] {$1$} -- ++(0,-.4)   arc (180:360:1cm)  -- ++(0,.4) node[above] {$1$};
      \end{tikzpicture}
+\qbracket{\zpar;-2}
      \begin{tikzpicture}[anchorbase,scale=.25,tinynodes]
      \draw[uncolored] (0,0) node[below] {$1$} -- ++(0,4) node[above] {$1$};
      \draw[uncolored] (2,0) node[below] {$1$} -- ++(0,4) node[above] {$1$};
      \end{tikzpicture} 
\end{equation}
  
\item The \emph{lollipop relations}
\begin{equation}\label{eq:lollipop}
      \begin{tikzpicture}[anchorbase,scale=.25,tinynodes]
      \draw[uncolored] (0,0) node[above] {$2$} -- ++(0,-1) .. controls ++(0.5,-0.5) and ++(0,0.5) .. ++(1,-1) arc (0:-180:1cm) node[left] {$\phantom{1}$} .. controls ++(0,0.5) and ++(-0.5,-0.5) .. ++(1,1);
      \draw[ccolored] (0,0) -- ++(0,-1);
      \end{tikzpicture} 
=0
\quad\text{and}\quad
      \begin{tikzpicture}[anchorbase,scale=.25,tinynodes]
      \draw[uncolored] (0,0) node[below] {$2$} -- ++(0,1) .. controls ++(0.5,0.5) and ++(0,-0.5) .. ++(1,1) arc (0:180:1cm) node[left] {$\phantom{1}$} .. controls ++(0,-0.5) and ++(-0.5,0.5) .. ++(1,-1);
      \draw[ccolored] (0,0) -- ++(0,1);
      \end{tikzpicture}
=0.
\end{equation}

\item The \emph{merge-split sliding relations}
\begin{equation}\label{eq:sliding}
\begin{tikzpicture}[anchorbase,xscale=-.25,tinynodes,yscale=-0.125]
      \draw[uncolored] (4,0) node[above] {$1$} .. controls ++(0,1) and ++(-0.5,-0.5) .. ++(1,1.5) .. controls ++(0.5,-0.5) and ++(0,1) .. ++(1,-1.5) node[above] {$1$}; 
      \draw[ccolored] (5,1.5) -- ++(0,1);
      \draw[uncolored] (4,4)  .. controls ++(0,-1) and ++(-0.5,0.5) .. ++(1,-1.5) .. controls ++(0.5,0.5) and ++(0,-1) .. ++(1,1.5);
      \draw[uncolored] (2,0) node[above] {$1$} -- ++(0,4);
      \draw[uncolored] (4,8) arc (180:0:1cm and 2cm) ;
      \draw (0,0) node[above] {$1$} -- ++(0,8) arc (180:0:1cm and 2cm);
      \draw (6,8) -- ++(0,-4);
      \node at (1,11.25) {$\phantom{1}$};
      \draw[uncolored] (2,4)  .. controls ++(0,1) and ++(-0.5,-0.5) .. ++(1,1.5) .. controls ++(0.5,-0.5) and ++(0,1) .. ++(1,-1.5) ;
      \draw[ccolored] (3,5.5) -- ++(0,1);
      \draw[uncolored] (2,8)  .. controls ++(0,-1) and ++(-0.5,0.5) .. ++(1,-1.5) .. controls ++(0.5,0.5) and ++(0,-1) .. ++(1,1.5);
      \end{tikzpicture}
=
      \begin{tikzpicture}[anchorbase,xscale=-.25,tinynodes,yscale=-0.125]
      \draw[uncolored] (0,0) node[above] {$1$} .. controls ++(0,1) and ++(-0.5,-0.5) .. ++(1,1.5) .. controls ++(0.5,-0.5) and ++(0,1) .. ++(1,-1.5) node[above] {$1$}; 
      \draw[ccolored] (1,1.5) -- ++(0,1);
      \draw[uncolored] (0,4)  .. controls ++(0,-1) and ++(-0.5,0.5) .. ++(1,-1.5) .. controls ++(0.5,0.5) and ++(0,-1) .. ++(1,1.5);
	  \draw[uncolored] (4,0) node[above] {$1$}  -- ++(0,4);
	  \draw[uncolored] (4,8)  arc (180:0:1cm and 2cm) -- ++(0,-8) node[above] {$1$} ;
      \draw (0,4) -- ++(0,4) arc (180:0:1cm and 2cm);
      \draw[uncolored] (2,4)  .. controls ++(0,1) and ++(-0.5,-0.5) .. ++(1,1.5) .. controls ++(0.5,-0.5) and ++(0,1) .. ++(1,-1.5) ;
      \draw[ccolored] (3,5.5) -- ++(0,1);
      \draw[uncolored] (2,8)  .. controls ++(0,-1) and ++(-0.5,0.5) .. ++(1,-1.5) .. controls ++(0.5,0.5) and ++(0,-1) .. ++(1,1.5);
      \node at (1,11.25) {$\phantom{1}$}; 
      \end{tikzpicture}
\quad\text{and}\quad
      \begin{tikzpicture}[anchorbase,xscale=.25,tinynodes,yscale=0.125]
      \draw[uncolored] (4,0) node[below] {$1$} .. controls ++(0,1) and ++(-0.5,-0.5) .. ++(1,1.5) .. controls ++(0.5,-0.5) and ++(0,1) .. ++(1,-1.5) node[below] {$1$}; 
      \draw[ccolored] (5,1.5) -- ++(0,1);
      \draw[uncolored] (4,4)  .. controls ++(0,-1) and ++(-0.5,0.5) .. ++(1,-1.5) .. controls ++(0.5,0.5) and ++(0,-1) .. ++(1,1.5);
      \draw[uncolored] (2,0) node[below] {$1$} -- ++(0,4);
      \draw[uncolored] (4,8) arc (180:0:1cm and 2cm) ;
      \draw (0,0) node[below] {$1$} -- ++(0,8) arc (180:0:1cm and 2cm);
      \draw (6,8) -- ++(0,-4);
      \node at (1,11.25) {$\phantom{1}$};
      \draw[uncolored] (2,4)  .. controls ++(0,1) and ++(-0.5,-0.5) .. ++(1,1.5) .. controls ++(0.5,-0.5) and ++(0,1) .. ++(1,-1.5) ;
      \draw[ccolored] (3,5.5) -- ++(0,1);
      \draw[uncolored] (2,8)  .. controls ++(0,-1) and ++(-0.5,0.5) .. ++(1,-1.5) .. controls ++(0.5,0.5) and ++(0,-1) .. ++(1,1.5);
        \node at (1,11) {$\phantom{1}$}; 
      \end{tikzpicture}
=
      \begin{tikzpicture}[anchorbase,xscale=.25,tinynodes,yscale=0.125]
      \draw[uncolored] (0,0) node[below] {$1$} .. controls ++(0,1) and ++(-0.5,-0.5) .. ++(1,1.5) .. controls ++(0.5,-0.5) and ++(0,1) .. ++(1,-1.5) node[below] {$1$}; 
      \draw[ccolored] (1,1.5) -- ++(0,1);
      \draw[uncolored] (0,4)  .. controls ++(0,-1) and ++(-0.5,0.5) .. ++(1,-1.5) .. controls ++(0.5,0.5) and ++(0,-1) .. ++(1,1.5);
	  \draw[uncolored] (4,0) node[below] {$1$}  -- ++(0,4);
	  \draw[uncolored] (4,8)  arc (180:0:1cm and 2cm) -- ++(0,-8) node[below] {$1$} ;
      \draw (0,4) -- ++(0,4) arc (180:0:1cm and 2cm);
      \draw[uncolored] (2,4)  .. controls ++(0,1) and ++(-0.5,-0.5) .. ++(1,1.5) .. controls ++(0.5,-0.5) and ++(0,1) .. ++(1,-1.5) ;
      \draw[ccolored] (3,5.5) -- ++(0,1);
      \draw[uncolored] (2,8)  .. controls ++(0,-1) and ++(-0.5,0.5) .. ++(1,-1.5) .. controls ++(0.5,0.5) and ++(0,-1) .. ++(1,1.5);
      \node at (1,11) {$\phantom{1}$};
      \end{tikzpicture}
\end{equation}
\end{enumerate}
\endgroup
\end{subequations}
\addtocounter{equation}{-1}
\end{definition}

\begin{remark}\label{remark:over-under-crossing}
Thanks to relation \eqref{eq:lollipop}, it is 
irrelevant whether we use overcrossings or 
undercrossings in \eqref{eq:lasso}. 
Moreover, one directly sees that the 
symmetries $\Psi$ and $\omega$ from \fullref{remark:web-symmetries} 
extend to $\WebDz$ (where we assume that $\Psi$ 
also flips $\zpar\leftrightarrow\zpar^{-1}$). Abusing notation, 
we denote these symmetries by the same symbols.
\end{remark}

\begin{remark}\label{remark:beware-monoidal}
Beware that a cup or a cap in a 
diagram representing a morphism in $\WebDz$ 
is only allowed if there are no strands on its right, 
cf.\ \fullref{figure:webs}. Here are some additional examples:
\begin{gather*}
\text{Allowed: }
\begin{tikzpicture}[anchorbase,scale=.25,tinynodes]
      \draw[uncolored] (-2,0) node[above] {$1$} to (-2,-2) node[below] {$1$};
      \draw[uncolored] (0,0) node[above] {$1$} to [out=270, in=180] (1,-1) to [out=0, in=270] (2,0) node[above] {$1$};
      \draw[densely dashed] (1,-1) to [out=315, in=180] (3,-1.5);
      \node at (0,1) {\phantom{a}};
      \node at (0,-4) {\phantom{a}};
\end{tikzpicture}
\qquad
\text{Not allowed: }
\begin{tikzpicture}[anchorbase,scale=.25,tinynodes]
      \draw[uncolored] (2,0) node[above] {$1$} to (2,-2) node[below] {$1$};
      \draw[uncolored] (0,0) node[above] {$1$} to [out=270, in=0] (-1,-1) to [out=180, in=270] (-2,0) node[above] {$1$};
      \draw[densely dashed] (-1,-1) to [out=315, in=180] (3,-1.5);
      \node at (0,1) {\phantom{a}};
      \node at (0,-4) {\phantom{a}};
      \draw[thick, myred] (-2.25,-2) to (3.25,0);
\end{tikzpicture}
\end{gather*}
In particular, there are no zig-zag-type relations:
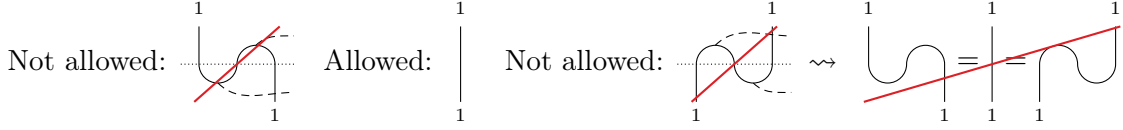
\begin{gather*}
\text{Not allowed: }
\begin{tikzpicture}[anchorbase,scale=.25,tinynodes]
      \draw[uncolored] (0,0) node[below] {$1$} to (0,2) to [out=90, in=0] (-1,3) to [out=180, in=90] (-2,2) to [out=270, in=0] (-3,1) to [out=180, in=270] (-4,2) to (-4,4) node[above] {$1$};
      \draw[uncolored, densely dotted] (1,2) to (-5,2);
      \draw[densely dashed] (-1,3) to [out=45, in=180] (1,3.5);
      \draw[densely dashed] (-3,1) to [out=315, in=180] (1,.5);
      \draw[thick, myred] (.25,4) to (-4.25,0);
\end{tikzpicture}
\quad
\text{Allowed: }
\begin{tikzpicture}[anchorbase,scale=.25,tinynodes]
      \draw[uncolored] (0,0) node[below] {$1$} to (0,4) node[above] {$1$};
\end{tikzpicture}
\quad
\text{Not allowed: }
\begin{tikzpicture}[anchorbase,scale=.25,tinynodes]
      \draw[uncolored] (0,0) node[below] {$1$} to (0,2) to [out=90, in=180] (1,3) to [out=0, in=90] (2,2) to [out=270, in=180] (3,1) to [out=0, in=270] (4,2) to (4,4) node[above] {$1$};
      \draw[uncolored, densely dotted] (-1,2) to (5,2);
      \draw[densely dashed] (1,3) to [out=45, in=180] (5,3.5);
      \draw[densely dashed] (3,1) to [out=315, in=180] (5,.5);
      \draw[thick, myred] (-.25,0) to (4.25,4);
\end{tikzpicture}
\;\rightsquigarrow\;
\begin{tikzpicture}[anchorbase,scale=.25,tinynodes]
	  \draw[uncolored] (-5,0) node[below] {$1$} to (-5,2) to [out=90, in=0] (-6,3) to [out=180, in=90] (-7,2) to [out=270, in=0] (-8,1) to [out=180, in=270] (-9,2) to (-9,4) node[above] {$1$};
	  \node at (-3.75,2) {\text{\normalsize$=$}};
	  \draw[uncolored] (-2.5,0) node[below] {$1$} to (-2.5,4) node[above] {$1$};
	  \node at (-1.25,2) {\text{\normalsize$=$}};
      \draw[uncolored] (0,0) node[below] {$1$} to (0,2) to [out=90, in=180] (1,3) to [out=0, in=90] (2,2) to [out=270, in=180] (3,1) to [out=0, in=270] (4,2) to (4,4) node[above] {$1$};
      \draw[thick, myred] (-9.25,0) to (4.25,4);
\end{tikzpicture}
\end{gather*}
and also other types of isotopy-like relations do not hold. 
We will meet the representation theoretical interpretation of this left-right partitioning 
in \fullref{sec-reps}, see also \fullref{remark:not-an-intertwiner-so}.
\end{remark}

\subsubsection{Topological versions of the \texorpdfstring{$\typeBDC$}{BD,C}-web relations}

Next, we give some alternative, topologically 
more meaningful, relations to our defining relations 
from above. 

\begin{subequations}
\begingroup
\renewcommand{\theequation}{$\typeBDC$\alph{equation}}
\begin{lemma}
  \renewcommand{\qedsymbol}{}
  \label{lemma:bubble-equi}
The bubble removal \eqref{eq:bubble} is equivalent to
\begin{equation}\label{eq:bubble-equi}
      \begin{tikzpicture}[anchorbase,scale=.25, tinynodes]
      \draw[uncolored] (0,4) node[above] {$1$} .. controls ++(0,-3) and ++(0,-1.5) .. ++(2,-2);
      \draw[uncolored,cross line] (0,0) node[below] {$1$} .. controls ++(0,3) and ++(0,1.5) .. ++(2,2);
      \end{tikzpicture} 
=-\zpar^{-1}
      \begin{tikzpicture}[anchorbase,scale=.25, tinynodes]
      \draw[uncolored] (0,0) node[below] {$1$} -- ++(0,4) node[above] {$1$};
      \end{tikzpicture}
\end{equation}
\end{lemma}

\begin{lemma}\renewcommand{\qedsymbol}{}\label{lemma:lasso-equi}
The lasso move \eqref{eq:lasso} is equivalent to
\begin{gather}\label{eq:lasso-equi}
     \begin{tikzpicture}[anchorbase,xscale=.25,yscale=.20,tinynodes]
     \draw[uncolored] (6,5.5) .. controls ++(0,-1.5) and ++(0,1.5) .. ++(-4,-4) -- ++(0,-1.5) node[below] {$1$};
     \draw[uncolored,cross line] (0,7) node[above] {$1$} -- ++(0,-1.5)  .. controls ++(0,-1.5) and ++(0,+1.5) .. ++(4,-4) arc (180:360:1cm) .. controls ++(0,1.5) and ++(0,-1.5) .. ++(-4,+4) -- ++(0,1.5) node[above] {$1$};
     \draw[uncolored,cross line] (0,0) node[below] {$1$} -- ++(0,1.5) .. controls ++(0,1.5) and ++(0,-1.5) .. ++(4,4) arc (180:0:1cm); 
     \end{tikzpicture} 
=
      \begin{tikzpicture}[anchorbase,scale=.25,tinynodes]
      \draw[uncolored] (0,0) node[below] {$1$} -- ++(0,.4)  arc (180:0:1cm)  -- ++(0,-.4) node[below] {$1$};
      \draw[uncolored] (0,4) node[above] {$1$} -- ++(0,-.4)   arc (180:360:1cm)  -- ++(0,.4) node[above] {$1$};
      \end{tikzpicture}
\end{gather}
\end{lemma}

\begin{lemma}\renewcommand{\qedsymbol}{}\label{lemma:lolli-equi}
The lollipop relations \eqref{eq:lollipop} are equivalent to
\begin{gather}\label{eq:lolli-equi}
        \begin{tikzpicture}[anchorbase,scale=.25,tinynodes]
	    \draw[uncolored] (-1,0) node[above] {$1$} .. controls ++(0,-1) and ++(0,1) .. ++(2,-2)  arc (0:-180:1cm);
        \draw[uncolored, cross line] (-1,-2) .. controls ++(0,1) and ++(0,-1) .. ++(2,2) node[above] {$1$};
        \end{tikzpicture} 
=-\qpar^{-1}
        \begin{tikzpicture}[anchorbase,scale=.25,tinynodes]
	    \draw[uncolored] (-1,0) node[above] {$1$} to [out=270, in=180] (0,-1) to [out=0, in=270] (1,0) node[above] {$1$};
	    \node at (0,-2.3) {$\phantom{1}$}; 
        \end{tikzpicture}
\quad\text{and}\quad
        \begin{tikzpicture}[anchorbase,scale=.25,tinynodes]
        \draw[uncolored] (-1,2) .. controls ++(0,-1) and ++(0,1) .. ++(2,-2) node[below] {$1$};
	    \draw[uncolored, cross line] (-1,0) node[below] {$1$} .. controls ++(0,1) and ++(0,-1) .. ++(2,2)  arc (0:180:1cm);
        \end{tikzpicture} 
=-\qpar^{-1}
        \begin{tikzpicture}[anchorbase,scale=.25,tinynodes]
	    \draw[uncolored] (-1,0) node[below] {$1$} to [out=90, in=180] (0,1) to [out=0, in=90] (1,0) node[below] {$1$};
	    \node at (0,2.5) {$\phantom{1}$}; 
        \end{tikzpicture}
\end{gather}
\end{lemma}

\begin{lemma}\renewcommand{\qedsymbol}{}\label{lemma:sliding-equi}
The merge-split sliding relations \eqref{eq:sliding} are equivalent to
\begin{gather}\label{eq:sliding-equi}
      \begin{tikzpicture}[anchorbase,xscale=.25,tinynodes,yscale=0.125]
      \draw[uncolored] (2,0) node[below] {$1$} .. controls ++(0,2) and ++(0,-2) .. ++(-2,4);
      \draw[uncolored,cross line] (0,0) node[below] {$1$} .. controls ++(0,2) and ++(0,-2) .. ++(2,4);
	  \draw[uncolored] (4,0) node[below] {$1$}  -- ++(0,4) .. controls ++(0,2) and  ++(0,-2) .. ++(-2,4);
	  \draw[uncolored,cross line] (2,4)  .. controls ++(0,2) and  ++(0,-2) .. ++(2,4) arc (180:0:1cm and 2cm) -- ++(0,-8) node[below] {$1$} ;
      \draw (0,4) -- ++(0,4) arc (180:0:1cm and 2cm); 
      \node at (1,11) {$\phantom{1}$};
      \end{tikzpicture}
=
      \begin{tikzpicture}[anchorbase,xscale=.25,tinynodes,yscale=0.125]
      \draw[uncolored] (6,0) node[below] {$1$} .. controls ++(0,2) and ++(0,-2) .. ++(-2,4);
      \draw[uncolored,cross line] (4,0) node[below] {$1$} .. controls ++(0,2) and ++(0,-2) .. ++(2,4);
	  \draw[uncolored] (4,4) .. controls ++(0,2) and  ++(0,-2) .. ++(-2,4);
	  \draw[uncolored,cross line] (2,0) node[below] {$1$} -- ++(0,4)  .. controls ++(0,2) and  ++(0,-2) .. ++(2,4) arc (180:0:1cm and 2cm) ;
      \draw (0,0) node[below] {$1$} -- ++(0,8) arc (180:0:1cm and 2cm);
      \draw (6,8) -- ++(0,-4);
      \node at (1,11) {$\phantom{1}$};
      \end{tikzpicture}
\quad\text{and}\quad
      \begin{tikzpicture}[anchorbase,xscale=-.25,tinynodes,yscale=-0.125]
      \draw[uncolored] (6,0) node[above] {$1$} .. controls ++(0,2) and ++(0,-2) .. ++(-2,4);
      \draw[uncolored,cross line] (4,0) node[above] {$1$} .. controls ++(0,2) and ++(0,-2) .. ++(2,4);
	  \draw[uncolored] (4,4) .. controls ++(0,2) and  ++(0,-2) .. ++(-2,4);
	  \draw[uncolored,cross line] (2,0) node[above] {$1$} -- ++(0,4)  .. controls ++(0,2) and  ++(0,-2) .. ++(2,4) arc (180:0:1cm and 2cm) ;
      \draw (0,0) node[above] {$1$} -- ++(0,8) arc (180:0:1cm and 2cm);
      \draw (6,8) -- ++(0,-4);
      \node at (1,11.25) {$\phantom{1}$};
      \end{tikzpicture}
=
      \begin{tikzpicture}[anchorbase,xscale=-.25,tinynodes,yscale=-0.125]
      \draw[uncolored] (2,0) node[above] {$1$} .. controls ++(0,2) and ++(0,-2) .. ++(-2,4);
      \draw[uncolored,cross line] (0,0) node[above] {$1$} .. controls ++(0,2) and ++(0,-2) .. ++(2,4);
	  \draw[uncolored] (4,0) node[above] {$1$}  -- ++(0,4) .. controls ++(0,2) and  ++(0,-2) .. ++(-2,4);
	  \draw[uncolored,cross line] (2,4)  .. controls ++(0,2) and  ++(0,-2) .. ++(2,4) arc (180:0:1cm and 2cm) -- ++(0,-8) node[above] {$1$} ;
      \draw (0,4) -- ++(0,4) arc (180:0:1cm and 2cm); 
      \node at (1,11.25) {$\phantom{1}$};
      \end{tikzpicture}
\end{gather}
\end{lemma}

\endgroup
\end{subequations}
\addtocounter{equation}{-1}

\makeautorefname{lemma}{Lemmas} 

We give the proofs of 
\fullref{lemma:bubble-equi}, \ref{lemma:lasso-equi},
\ref{lemma:lolli-equi} and \ref{lemma:sliding-equi} 
after we have commented on the topological 
nature of the $\typeBDC$-web calculus.

\makeautorefname{lemma}{Lemma} 

Note that, by using the involution $\Psi$ and 
the anti-involution $\omega$, we obtain many 
more equivalent relations.

\subsubsection{Why \texorpdfstring{$\typeBDC$}{cup}-webs do not form a monoidal category}

The first thing to note is that the $\typeBDC$-web calculus is only 
partially topological: Some topological manipulations are allowed, e.g. \eqref{eq:lasso-equi}, but 
its similar looking counterparts do not necessarily hold. For example, we have
\begin{gather*}
	  \begin{tikzpicture}[anchorbase,xscale=.25,yscale=.20,tinynodes]
     \draw[uncolored] (0,7) node[above] {$1$} -- ++(0,-1.5)  .. controls ++(0,-1.5) and ++(0,+1.5) .. ++(4,-4) arc (180:360:1cm) .. controls ++(0,1.5) and ++(0,-1.5) .. ++(-4,+4) -- ++(0,1.5) node[above] {$1$};
     \draw[uncolored,cross line] (0,0) node[below] {$1$} -- ++(0,1.5) .. controls ++(0,1.5) and ++(0,-1.5) .. ++(4,4) arc (180:0:1cm); 
     \draw[uncolored,cross line] (6,5.5) .. controls ++(0,-1.5) and ++(0,1.5) .. ++(-4,-4) -- ++(0,-1.5) node[below] {$1$};
     \end{tikzpicture} 
\neq
      \begin{tikzpicture}[anchorbase,scale=.25,tinynodes]
      \draw[uncolored] (0,0) node[below] {$1$} -- ++(0,.4)  arc (180:0:1cm)  -- ++(0,-.4) node[below] {$1$};
      \draw[uncolored] (0,4) node[above] {$1$} -- ++(0,-.4)   arc (180:360:1cm)  -- ++(0,.4) node[above] {$1$};
      \end{tikzpicture}
\end{gather*}
Moreover, one is not allowed to use certain 
isotopies cf. \fullref{remark:beware-monoidal}.
In particular, there is no interchange law \eqref{eq:inter-law}; and 
\eqref{eq:bubble-equi} and \eqref{eq:lolli-equi} are different relations (``turning your head is forbidden'').

Furthermore, one may be tempted to define arbitrary cups and caps as in the following picture:
\begin{gather}\label{eq:cup-cap-extension}
    \begin{tikzpicture}[anchorbase,scale=.25, tinynodes]
	\draw[uncolored] (-1,2) node[above] {\raisebox{0.025cm}{$1$}} to [out=270, in=180] (0,1) to [out=0, in=270] (1,2) node[above] {\raisebox{0.025cm}{$1$}};
	\draw[colored] (-3,-2) node[below] {$b$} to (-3,2) node[above] {$b$};
	\draw[colored] (-5,-2) node[below] {$a$} to (-5,2) node[above] {$a$};
	\draw[colored] (3,-2) node[below] {$c$} to (3,2) node[above] {$c$};
	\draw[colored] (5,-2) node[below] {$d$} to (5,2) node[above] {$d$};
	\node at (-4,-1.5) {$\cdots$};
	\node at (4,-1.5) {$\cdots$};
	\node at (-4,1.5) {$\cdots$};
	\node at (4,1.5) {$\cdots$};
    \end{tikzpicture}
\,=\,
    \begin{tikzpicture}[anchorbase,scale=.25, tinynodes]
	\draw[uncolored] (7,-1) to [out=270, in=180] (8,-2) to [out=0, in=270] (9,-1);
	\draw[colored] (-3,-2) node[below] {$b$} to (-3,2) node[above] {$b$};
	\draw[colored] (-5,-2) node[below] {$a$} to (-5,2) node[above] {$a$};
	\draw[colored] (3,-2) node[below] {$c$} to (3,2) node[above] {$c$};
	\draw[colored] (5,-2) node[below] {$d$} to (5,2) node[above] {$d$};
	\draw[uncolored, cross line] (-1,2) node[above] {\raisebox{0.025cm}{$1$}} .. controls ++(0,-3) and ++(0,2) ..  (7,-1);
	\draw[uncolored, cross line] (1,2) node[above] {\raisebox{0.025cm}{$1$}} .. controls ++(0,-2) and ++(0,3) ..  (9,-1);
	\node at (-4,-1.5) {$\cdots$};
	\node at (4,-1.5) {$\cdots$};
	\node at (-4,1.5) {$\cdots$};
	\node at (4,1.5) {$\cdots$};
    \end{tikzpicture}
\end{gather}
However, this is dangerous since the diagram
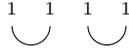
\[    
	\begin{tikzpicture}[anchorbase,scale=.25, tinynodes]
	\draw[uncolored] (-1,2) node[above] {$1$} to [out=270, in=180] (0,1) to [out=0, in=270] (1,2) node[above] {$1$};
	\draw[uncolored] (3,2) node[above] {$1$} to [out=270, in=180] (4,1) to [out=0, in=270] (5,2) node[above] {$1$};
    \end{tikzpicture}
\]
would be ambiguous, as it could be any of the following two pictures:
\begin{gather}\label{eq:shorthand}    
	\begin{tikzpicture}[anchorbase,scale=.25, tinynodes]
	\draw[uncolored] (-1,2) node[above] {$1$} .. controls ++(0,-2) and ++(0,1.5) .. ++(4,-3) arc(180:360:1cm) .. controls ++(0,1.5)  and ++(0,-1.5) .. (1,2) node[above] {$1$};
	\draw[uncolored] (3,2) node[above] {$1$} to [out=270, in=180] (4,1) to [out=0, in=270] (5,2) node[above] {$1$};
    \end{tikzpicture}
\quad\text{or}\quad
     \begin{tikzpicture}[anchorbase,scale=.25, tinynodes]
	 \draw[uncolored] (3,2) node[above] {$1$} -- ++(0,-3) to [out=270, in=180] ++(1,-1) to [out=0, in=270] ++(1,1) -- ++(0,3) node[above] {$1$};
	 \draw[uncolored,cross line] (-1,2) node[above] {$1$} .. controls ++(0,-3) and ++(0,1.5) .. ++(8,-3) arc(180:360:1cm) .. controls ++(0,2.5)  and ++(0,-1.5) .. (1,2) node[above] {$1$};
     \end{tikzpicture}
\end{gather}
Unfortunately, these are not equal. 
(We note that, in the setting of categories 
with a monoidal action, the first diagram is the 
correct meaning, and we already used this before, 
namely 
in \eqref{eq:sliding-equi}, cf. \fullref{remark:beware-monoidal}.)

To summarize, 
one has the whole power of topological manipulations for 
the $\typeA$-web part, but for cups and caps one has to be extremely careful. 
For example, \eqref{eq:lasso-equi} and \eqref{eq:sliding-equi} 
are the only Reidemeister type $2$ moves involving cups 
and caps which hold.

All of these problems disappear if one de-quantizes, 
and the resulting $\typeBDC$-web category at $q=1$
is a genuine monoidal category. Hence, $\WebDz$ 
gives an example of a deformation of a monoidal category which is 
not monoidal anymore. 
This is related, as we shall see in 
\fullref{sec-reps}, to the well-understood fact that 
the quantization of the inclusion $\frakso_{n} \subseteq \frakgl_n$ 
cannot be realized as an inclusion of Hopf algebras, but only as the inclusion of a 
coideal subalgebra.

Actually, in the de-quantized case
the resulting web category
is 
not just monoidal, 
but also gets a topological flavor by defining 
\emph{thick cup and cap morphisms} via explosion,
cf.\ \fullref{remark:no-thick-things}, and cups and caps between 
$\typeA$ web strands as in \eqref{eq:cup-cap-extension}. The 
corresponding web categories will satisfy all reasonable kinds of isotopies.
This
is very 
much in the spirit 
of the original ``web categories'' introduced 
by Kuperberg \cite{Kup1}.

\subsection{Some useful lemmas}\label{subsec-lemmas-BD}

Until the end of this section we 
will work in $\WebDz$.

\makeautorefname{lemma}{Lemmas} 

\begin{proof}[Proof of \fullref{lemma:bubble-equi}, \ref{lemma:lasso-equi}, \ref{lemma:lolli-equi} 
and \ref{lemma:sliding-equi}]
This basically follows by expanding the crossings 
using \eqref{eq:braiding}. 
However, we give the necessary calculations for \fullref{lemma:bubble-equi}, \ref{lemma:lolli-equi} 
and \ref{lemma:sliding-equi} in full detail since they serve as a 
blueprint for all our calculations in the present and the next section.
(Verifying the equivalence between \eqref{eq:lasso} 
and \eqref{eq:lasso-equi}, which inspired the name lasso move, 
is lengthy but easy, and follows along the same lines.) 
The main idea is to use the fact that the $\typeA$-web calculus is 
topological, and then carefully arrange the diagrams to apply the 
defining relations from \fullref{definition:typeBDwebs}.

\makeautorefname{lemma}{Lemma}

\begin{description}[leftmargin=0pt,itemsep=1ex]
\item[\textbf{\eqref{eq:bubble} implies \eqref{eq:bubble-equi}}]
Here is the calculation:
\[
	  \begin{tikzpicture}[anchorbase,scale=.25, tinynodes]
	  \draw[uncolored] (0,6) node [above] {1} to (0,4) to [out=270, in=90] (2,0);
      \draw[uncolored, cross line] (0,-2) node [below] {1} to (0,0) to [out=90, in=270] (2,4);
      \draw[uncolored] (2,0) to [out=270, in=180] (3,-1) to [out=0, in=270] (4,0);
      \draw[uncolored] (2,4) to [out=90, in=180] (3,5) to [out=0, in=90] (4,4);
      \draw[uncolored] (4,0) to (4,4);
      \end{tikzpicture}  
\overset{\eqref{eq:braiding}}{=}
-\qpar^{-1}
	  \begin{tikzpicture}[anchorbase,scale=.25, tinynodes]
      \draw[uncolored] (0,6) node [above] {1} to (0,4) to [out=270, in=90] (0,0);
      \draw[uncolored] (0,-2) node [below] {1} to (0,0);
      \draw[uncolored] (2,0) to (2,4);
      \draw[uncolored] (2,0) to [out=270, in=180] (3,-1) to [out=0, in=270] (4,0);
      \draw[uncolored] (2,4) to [out=90, in=180] (3,5) to [out=0, in=90] (4,4);
      \draw[uncolored] (4,0) to (4,4);
      \draw[densely dashed] (1.75,0) rectangle (4.25,4);
      \end{tikzpicture} 
+
      \begin{tikzpicture}[anchorbase,scale=.25, tinynodes]
      \draw[uncolored] (0,-2) node [below] {1} to (0,0) to [out=90, in=225] (1,1.5);
      \draw[uncolored] (2,0) to [out=90, in=315] (1,1.5);
      \draw[ccolored] (1,1.5) to (1,2.5);
      \draw[uncolored] (0,6) node [above] {1} to (0,4) to [out=270, in=135] (1,2.5);
      \draw[uncolored] (2,4) to [out=270, in=45] (1,2.5);
      \draw[uncolored] (2,0) to [out=270, in=180] (3,-1) to [out=0, in=270] (4,0);
      \draw[uncolored] (2,4) to [out=90, in=180] (3,5) to [out=0, in=90] (4,4);
      \draw[uncolored] (4,0) to (4,4);
      \end{tikzpicture} 
=
-\qpar^{-1}
	  \begin{tikzpicture}[anchorbase,scale=.25, tinynodes]
      \draw[uncolored] (0,6) node [above] {1} to (0,4) to [out=270, in=90] (0,0);
      \draw[uncolored] (0,-2) node [below] {1} to (0,0);
      \draw[uncolored] (2,2) to [out=270, in=180] (3,1) to [out=0, in=270] (4,2);
      \draw[uncolored] (2,2) to [out=90, in=180] (3,3) to [out=0, in=90] (4,2);
      \end{tikzpicture} 
+
      \begin{tikzpicture}[anchorbase,scale=.25, tinynodes]
      \draw[uncolored] (0,-2) node [below] {1} to (0,0) to [out=90, in=225] (1,1.5);
      \draw[uncolored] (2,0) to [out=90, in=315] (1,1.5);
      \draw[ccolored] (1,1.5) to (1,2.5);
      \draw[uncolored] (0,6) node [above] {1} to (0,4) to [out=270, in=135] (1,2.5);
      \draw[uncolored] (2,4) to [out=270, in=45] (1,2.5);
      \draw[uncolored] (2,0) to [out=270, in=180] (3,-1) to [out=0, in=270] (4,0);
      \draw[uncolored] (2,4) to [out=90, in=180] (3,5) to [out=0, in=90] (4,4);
      \draw[uncolored] (4,0) to (4,4);
      \end{tikzpicture}
\overunder{\eqref{eq:circle}}{\eqref{eq:bubble}}{=}
\begin{aligned}
&-\qpar^{-1}\qbracket{\zpar;0}\\
& \qquad +\qbracket{\zpar;1}     
\end{aligned}
	  \begin{tikzpicture}[anchorbase,scale=.25, tinynodes]
	  \draw[uncolored] (0,6) node [above] {1} to (0,0);
      \draw[uncolored, cross line] (0,-2) node [below] {1} to (0,0);
      \end{tikzpicture}
=
-\zpar^{-1}
	  \begin{tikzpicture}[anchorbase,scale=.25, tinynodes]
	  \draw[uncolored] (0,6) node [above] {1} to (0,0);
      \draw[uncolored, cross line] (0,-2) node [below] {1} to (0,0);
      \end{tikzpicture} 
\]
Note hereby that we used a topological manipulation on an $\typeA$-web part.

\item[\textbf{\eqref{eq:lollipop} implies \eqref{eq:lolli-equi}}] Let us verify the 
right equation:
\[
      \begin{tikzpicture}[anchorbase,scale=.25, tinynodes]
	  \draw[uncolored] (0,4) to [out=270, in=90] (2,0) node [below] {1};
      \draw[uncolored, cross line] (0,0) node [below] {1} to [out=90, in=270] (2,4);
      \draw[uncolored] (0,4) to [out=90, in=180] (1,5) to [out=0, in=90] (2,4);
      \end{tikzpicture}
\overset{\eqref{eq:braiding}}{=}
-\qpar^{-1}
	  \begin{tikzpicture}[anchorbase,scale=.25, tinynodes]
	  \draw[uncolored] (0,4) to [out=270, in=90] (0,0) node [below] {1};
      \draw[uncolored] (2,0) node [below] {1} to [out=90, in=270] (2,4);
      \draw[uncolored] (0,4) to [out=90, in=180] (1,5) to [out=0, in=90] (2,4);
      \end{tikzpicture}
+
      \begin{tikzpicture}[anchorbase,scale=.25, tinynodes]
	  \draw[uncolored] (0,0) node [below] {1} to [out=90, in=225] (1,1.5);
      \draw[uncolored] (2,0) node [below] {1} to [out=90, in=315] (1,1.5);
      \draw[ccolored] (1,1.5) to (1,2.5);
      \draw[uncolored] (0,4) to [out=270, in=135] (1,2.5);
      \draw[uncolored] (2,4) to [out=270, in=45] (1,2.5);
      \draw[uncolored] (0,4) to [out=90, in=180] (1,5) to [out=0, in=90] (2,4);
      \end{tikzpicture}
\overset{\eqref{eq:lollipop}}{=}
-\qpar^{-1}
	  \begin{tikzpicture}[anchorbase,scale=.25, tinynodes]
	  \draw[uncolored] (0,4) to [out=270, in=90] (0,0) node [below] {1};
      \draw[uncolored] (2,0) node [below] {1} to [out=90, in=270] (2,4);
      \draw[uncolored] (0,4) to [out=90, in=180] (1,5) to [out=0, in=90] (2,4);
      \draw[densely dashed] (-.25,.25) rectangle (2.25,4);
      \end{tikzpicture}
=
-\qpar^{-1}
      \begin{tikzpicture}[anchorbase,scale=.25, tinynodes]
	  \draw[white] (0,4) to [out=270, in=90] (0,0) node [below] {1};
      \draw[white] (2,0) node [below] {1} to [out=90, in=270] (2,4);
      \draw[white] (0,4) to [out=90, in=180] (1,5) to [out=0, in=90] (2,4);
      \draw[uncolored] (0,0) node [below] {1} to [out=90, in=180] (1,1) to [out=0, in=90] (2,0) node [below] {1};
      \end{tikzpicture}
\]
using the same trick as before.

\item[\textbf{\eqref{eq:sliding} implies \eqref{eq:sliding-equi}}] This can be shown 
as above: expanding the expressions in \eqref{eq:sliding-equi} gives 
four terms, two of which are equal, two of which are zero. The main 
topological manipulation one needs is of the form
\[
      \begin{tikzpicture}[anchorbase,scale=.25, tinynodes]
	  \draw[uncolored] (0,0) to [out=90, in=225] (1,1.5);
      \draw[uncolored] (2,0) to [out=90, in=315] (1,1.5);
      \draw[ccolored] (1,1.5) to (1,2.5);
      \draw[uncolored] (0,4) node [above] {1} to [out=270, in=135] (1,2.5);
      \draw[uncolored] (2,4) node [above] {1} to [out=270, in=45] (1,2.5);
      \draw[uncolored] (4,4) node [above] {1} to [out=270, in=180] (5,3) to [out=0, in=270] (6,4) node [above] {1};
      \draw[uncolored] (4,-4) to [out=270, in=180] (5,-5) to [out=0, in=270] (6,-4);
      \draw[uncolored] (0,0) to [out=270, in=90] (4,-4);
      \draw[uncolored] (2,0) to [out=270, in=90] (6,-4);
      \draw[densely dashed] (-.5,3.75) to (2.5,3.75) to [out=270, in=90] (6.5,-4) to (3.5,-4) to [out=90, in=270] (-.5,0) to (-.5,3.75);
      \end{tikzpicture}
=
      \begin{tikzpicture}[anchorbase,scale=.25, tinynodes]
	  \draw[uncolored] (4,-4) to [out=90, in=225] (5,-2.5);
      \draw[uncolored] (6,-4) to [out=90, in=315] (5,-2.5);
      \draw[ccolored] (5,-1.5) to (5,-2.5);
      \draw[uncolored] (4,0) to [out=270, in=135] (5,-1.5);
      \draw[uncolored] (6,0) to [out=270, in=45] (5,-1.5);
      \draw[uncolored] (4,4) node [above] {1} to [out=270, in=180] (5,3) to [out=0, in=270] (6,4) node [above] {1};
      \draw[uncolored] (4,-4) to [out=270, in=180] (5,-5) to [out=0, in=270] (6,-4);
      \draw[uncolored] (0,4) node [above] {1} to [out=270, in=90] (4,0);
      \draw[uncolored] (2,4) node [above] {1} to [out=270, in=90] (6,0);
      \end{tikzpicture}
\]
where we recall that two cups next to each 
are actually a shorthand for the left diagram in \eqref{eq:shorthand}. 
To this we can then apply \eqref{eq:lollipop}.
\end{description}

The other implications follow similarly.
\end{proof}

Our next aim it to derive some diagrammatic relations 
which, as we will see later, correspond to 
relations in the quantum group $\quantumg(\frakso_{2k})$. 
In the proofs of the following lemmas, we will repeatedly use 
the defining relations of $\WebDz$,
as well as the 
topological and 
braided structure of $\WebAz$ 
(in particular, \eqref{eq:pitchfork} and \eqref{eq:ttwist}). At each step, we will indicate the 
most important relations that we use.

\begin{lemma}\label{lemma:typeD-EF-FE}
For all $a,b$ we have
\begin{equation}\label{eq:typeD-EF-FE}

\end{equation}
\end{lemma}

\begin{proof} 
Using the naturality of the braiding and the defining relations 
as well as the relations for $\typeA$-webs, we compute:
\begin{align*}
 	  &%
\end{align*}
The last step 
is just a tedious calculation with quantum numbers.
\end{proof}

\begin{lemma}\label{lemma-typeD-EF}
For all $a,b$ we have
\begin{equation}\label{eq:typeD-EF}
      %
\end{equation}
\end{lemma}

\begin{proof}
We get by the definition of the braiding:
\begin{align*}
      %
\end{equation*}
and hence we may assume $a=0$. Now, we have
\begin{align*} 
      %
\end{equation}
\end{lemma}

\begin{proof}
Observing that \eqref{eq:typeD-serre-relation-2} is equivalent to
\[
\qbracket{2}
      %
\]
the proof follows from the $\typeA$-web Serre relations 
(by applying the corresponding relation for $s=1$, $t=2$ to the marked part), 
cf.\ \fullref{remark:web-Serre}.
\end{proof}

\begin{lemma}\label{lemma:typeD-serre-relation}
For all $a,b,c$ we have
\begin{equation}\label{eq:typeD-serre-relation}
\qbracket{2}
      %
\end{equation}
\end{lemma}

\begin{proof}
First note that
\begin{align*}
      \tikzset{every picture/.style={xscale=0.9}}
      %
 
\end{align*}
Thus, the statement follows from the 
thick square switch relations \eqref{eq:thick-square}.
\end{proof}

\section{The \texorpdfstring{$\typeCBD$}{dot}-web category}\label{sec-Cwebs}
In this section, which is structured exactly as the previous one,
we define another 
web category which will play a complimentary 
role to the $\typeBDC$-web category, as it 
describes exterior $\typeC$-webs 
and symmetric $\typeB\typeD$-webs. 
We call its morphisms \emph{$\typeCBD$-webs}.

\newpage

\subsection{The diagrammatic \texorpdfstring{$\typeCBD$}{dot}-web category}\label{subsec-diacat-C}

\subsubsection{\texorpdfstring{$\typeCBD$}{dot}-webs}

Again, we work over $\CQZ$, and
we define:

\begin{definition}\label{definition:typeCwebs}
The \emph{$\typeCBD$-web category} 
$\WebCz$ is the additive closure 
of the $\CQZ$-li\-near $\WebAz$-category 
generated by the object $\varnothing$ 
and by the 
\emph{start/end dot morphisms}
\begin{subequations}
\begingroup
\renewcommand{\theequation}{$\typeCBD$gen}
\begin{equation}
\label{eq:dot-morphisms}
   \begin{tikzpicture}[anchorbase,scale=.25, tinynodes]
	\draw[ccolored] (0,0) to (0,1.5) node[above] {$2$};
	\node[dotbullet] at (0,0) {};
	\node at (0,1.5) {$\phantom{1}$};
   \end{tikzpicture}
\colon \varnothing\to 2
\quad\text{and}\quad
   \begin{tikzpicture}[anchorbase,scale=.25, tinynodes]
	\draw[ccolored] (0,0) to (0,-1.5) node[below] {$2$};
	\node[dotbullet] at (0,0) {};
	\node at (0,1.5) {$\phantom{1}$};
   \end{tikzpicture}
\colon 2\to \varnothing,
\end{equation}
\endgroup
\end{subequations}
\addtocounter{equation}{-1}
modulo the following relations:
\smallskip
\begin{enumerate}[label=$\vartriangleright$]

\setlength\itemsep{.15cm}

\begin{subequations}
\begingroup
\renewcommand{\theequation}{$\typeCBD$\arabic{equation}}
\item The \emph{barbell removal}
\begin{equation}\label{eq:barbell-C}
     \begin{tikzpicture}[anchorbase,scale=.25, tinynodes]
     \draw[ccolored] (0,-1) to (0,0) to (0,1);
	 \node[dotbullet] at (0,-1) {};
	 \node[dotbullet] at (0,1) {}; 
     \end{tikzpicture}
\,=\qbracketC{\zpar;0}.
\end{equation}
  
\item The \emph{thin K removal}
\begin{equation}\label{eq:thin-k}
      \begin{tikzpicture}[anchorbase,scale=.25, tinynodes]
      \draw[uncolored] (0,-2) node[below] {$1$} to (0,-.5);
      \draw[colored] (0,-.5) to (0,0) to (0,.5);
      \draw[uncolored] (0,.5) to (0,2) node[above] {$1$};
      \draw[ccolored] (0,-.5) to (1.5,-1) to (1.5,-1.75);
      \draw[ccolored] (0,.5) to (1.5,1) to (1.5,1.75);
      \node[dotbullet] at (1.5,-1.75) {};
	  \node[dotbullet] at (1.5,1.75) {};
      \end{tikzpicture}
=\qbracketC{\zpar;-1}
      \begin{tikzpicture}[anchorbase,scale=.25, tinynodes]
      \draw[uncolored] (0,-2) node[below] {$1$} to (0,2) node[above] {$1$};
      \end{tikzpicture}
\end{equation}
  
\item The \emph{thick K opening}
\begin{equation}\label{eq:thick-k}
      \begin{tikzpicture}[anchorbase,scale=.25, tinynodes]
      \draw[ccolored] (0,-2) node[below] {$2$} to (0,-.5);
      \draw[colored] (0,-.5) to (0,0) to (0,.5);
      \draw[ccolored] (0,.5) to (0,2) node[above] {$2$};
      \draw[ccolored] (0,-.5) to (1.5,-1) to (1.5,-1.75);
      \draw[ccolored] (0,.5) to (1.5,1) to (1.5,1.75);
      \node[dotbullet] at (1.5,-1.75) {};
	  \node[dotbullet] at (1.5,1.75) {};
      \end{tikzpicture}
=
      \begin{tikzpicture}[anchorbase,scale=.25, tinynodes]
      \draw[ccolored] (0,-2) node[below] {$2$} to (0,-.75);
      \draw[ccolored] (0,.75) to (0,2) node[above] {$2$};
      \node[dotbullet] at (0,-.75) {};
	  \node[dotbullet] at (0,.75) {};
      \end{tikzpicture}
+\qbracketC{\zpar;-2}
      \begin{tikzpicture}[anchorbase,scale=.25, tinynodes]
      \draw[ccolored] (0,-2) node[below] {$2$} to (0,2) node[above] {$2$};
      \end{tikzpicture}
\end{equation}
\item The \emph{merge-split sliding relations}
\begin{gather}\label{eq:typeC-sliding}
\begin{tikzpicture}[anchorbase,xscale=-.25,tinynodes,yscale=-0.125]
      \draw[uncolored] (4,0) node[above] {$1$} .. controls ++(0,1) and ++(-0.5,-0.5) .. ++(1,1.5) .. controls ++(0.5,-0.5) and ++(0,1) .. ++(1,-1.5) node[above] {$1$}; 
      \draw[ccolored] (5,1.5) -- ++(0,1);
      \draw[uncolored] (4,4)  .. controls ++(0,-1) and ++(-0.5,0.5) .. ++(1,-1.5) .. controls ++(0.5,0.5) and ++(0,-1) .. ++(1,1.5);
      \draw[uncolored] (2,0) node[above] {$1$} -- ++(0,4);
      \draw[uncolored] (4,8) arc (180:0:1cm and 2cm) ;
      \draw (0,0) node[above] {$1$} -- ++(0,8) arc (180:0:1cm and 2cm);
      \draw (6,8) -- ++(0,-4);
      \node at (1,11.25) {$\phantom{1}$};
      \draw[uncolored] (2,4)  .. controls ++(0,1) and ++(-0.5,-0.5) .. ++(1,1.5) .. controls ++(0.5,-0.5) and ++(0,1) .. ++(1,-1.5) ;
      \draw[ccolored] (3,5.5) -- ++(0,1);
      \draw[uncolored] (2,8)  .. controls ++(0,-1) and ++(-0.5,0.5) .. ++(1,-1.5) .. controls ++(0.5,0.5) and ++(0,-1) .. ++(1,1.5);
      \end{tikzpicture}
=
      \begin{tikzpicture}[anchorbase,xscale=-.25,tinynodes,yscale=-0.125]
      \draw[uncolored] (0,0) node[above] {$1$} .. controls ++(0,1) and ++(-0.5,-0.5) .. ++(1,1.5) .. controls ++(0.5,-0.5) and ++(0,1) .. ++(1,-1.5) node[above] {$1$}; 
      \draw[ccolored] (1,1.5) -- ++(0,1);
      \draw[uncolored] (0,4)  .. controls ++(0,-1) and ++(-0.5,0.5) .. ++(1,-1.5) .. controls ++(0.5,0.5) and ++(0,-1) .. ++(1,1.5);
	  \draw[uncolored] (4,0) node[above] {$1$}  -- ++(0,4);
	  \draw[uncolored] (4,8)  arc (180:0:1cm and 2cm) -- ++(0,-8) node[above] {$1$} ;
      \draw (0,4) -- ++(0,4) arc (180:0:1cm and 2cm);
      \draw[uncolored] (2,4)  .. controls ++(0,1) and ++(-0.5,-0.5) .. ++(1,1.5) .. controls ++(0.5,-0.5) and ++(0,1) .. ++(1,-1.5) ;
      \draw[ccolored] (3,5.5) -- ++(0,1);
      \draw[uncolored] (2,8)  .. controls ++(0,-1) and ++(-0.5,0.5) .. ++(1,-1.5) .. controls ++(0.5,0.5) and ++(0,-1) .. ++(1,1.5);
      \node at (1,11.25) {$\phantom{1}$}; 
      \end{tikzpicture}
\quad\text{and}\quad
      \begin{tikzpicture}[anchorbase,xscale=.25,tinynodes,yscale=0.125]
      \draw[uncolored] (4,0) node[below] {$1$} .. controls ++(0,1) and ++(-0.5,-0.5) .. ++(1,1.5) .. controls ++(0.5,-0.5) and ++(0,1) .. ++(1,-1.5) node[below] {$1$}; 
      \draw[ccolored] (5,1.5) -- ++(0,1);
      \draw[uncolored] (4,4)  .. controls ++(0,-1) and ++(-0.5,0.5) .. ++(1,-1.5) .. controls ++(0.5,0.5) and ++(0,-1) .. ++(1,1.5);
      \draw[uncolored] (2,0) node[below] {$1$} -- ++(0,4);
      \draw[uncolored] (4,8) arc (180:0:1cm and 2cm) ;
      \draw (0,0) node[below] {$1$} -- ++(0,8) arc (180:0:1cm and 2cm);
      \draw (6,8) -- ++(0,-4);
      \node at (1,11.25) {$\phantom{1}$};
      \draw[uncolored] (2,4)  .. controls ++(0,1) and ++(-0.5,-0.5) .. ++(1,1.5) .. controls ++(0.5,-0.5) and ++(0,1) .. ++(1,-1.5) ;
      \draw[ccolored] (3,5.5) -- ++(0,1);
      \draw[uncolored] (2,8)  .. controls ++(0,-1) and ++(-0.5,0.5) .. ++(1,-1.5) .. controls ++(0.5,0.5) and ++(0,-1) .. ++(1,1.5);
        \node at (1,11) {$\phantom{1}$}; 
      \end{tikzpicture}
=
      \begin{tikzpicture}[anchorbase,xscale=.25,tinynodes,yscale=0.125]
      \draw[uncolored] (0,0) node[below] {$1$} .. controls ++(0,1) and ++(-0.5,-0.5) .. ++(1,1.5) .. controls ++(0.5,-0.5) and ++(0,1) .. ++(1,-1.5) node[below] {$1$}; 
      \draw[ccolored] (1,1.5) -- ++(0,1);
      \draw[uncolored] (0,4)  .. controls ++(0,-1) and ++(-0.5,0.5) .. ++(1,-1.5) .. controls ++(0.5,0.5) and ++(0,-1) .. ++(1,1.5);
	  \draw[uncolored] (4,0) node[below] {$1$}  -- ++(0,4);
	  \draw[uncolored] (4,8)  arc (180:0:1cm and 2cm) -- ++(0,-8) node[below] {$1$} ;
      \draw (0,4) -- ++(0,4) arc (180:0:1cm and 2cm);
      \draw[uncolored] (2,4)  .. controls ++(0,1) and ++(-0.5,-0.5) .. ++(1,1.5) .. controls ++(0.5,-0.5) and ++(0,1) .. ++(1,-1.5) ;
      \draw[ccolored] (3,5.5) -- ++(0,1);
      \draw[uncolored] (2,8)  .. controls ++(0,-1) and ++(-0.5,0.5) .. ++(1,-1.5) .. controls ++(0.5,0.5) and ++(0,-1) .. ++(1,1.5);
      \node at (1,11) {$\phantom{1}$};
      \end{tikzpicture}
\end{gather}
\endgroup
\end{subequations}
\addtocounter{equation}{-1}
where the 
\emph{cup} and \emph{cap morphisms} are defined as
\begin{equation}\label{eq:cup-cap-C}
    \begin{tikzpicture}[anchorbase,scale=.25, tinynodes]
	\draw[uncolored] (-1,3) node[above] {$1$} to [out=270, in=180] (0,2) to [out=0, in=270] (1,3) node[above] {$1$};
	\node at (0,.4) {$\phantom{.}$};
    \end{tikzpicture}
=
    \begin{tikzpicture}[anchorbase,scale=.25, tinynodes]
	\draw[uncolored] (-1,3) node[above] {$1$} to [out=270, in=135] (0,1.5) to [out=45, in=270] (1,3) node[above] {$1$};
	\draw[ccolored] (0,1.5) to (0,0);
	\node[dotbullet] at (0,0) {};
    \end{tikzpicture}
\colon \varnothing\to 1\otimes 1
\quad\text{and}\quad
    \begin{tikzpicture}[anchorbase,scale=.25, tinynodes]
	\draw[uncolored] (-1,-3) node[below] {$1$} to [out=90, in=180] (0,-2) to [out=0, in=90] (1,-3) node[below] {$1$};
	\node at (0,-.4) {$\phantom{.}$};
    \end{tikzpicture}
=
    \begin{tikzpicture}[anchorbase,scale=.25, tinynodes]
	\draw[uncolored] (-1,-3) node[below] {$1$} to [out=90, in=225] (0,-1.5) to [out=315, in=90] (1,-3) node[below] {$1$};
	\draw[ccolored] (0,-1.5) to (0,0);
	\node[dotbullet] at (0,0) {};
    \end{tikzpicture}
\colon 1\otimes 1\to \varnothing.
\end{equation}
\end{enumerate}
\end{definition}

\begin{remark}\label{remark:beware-monoidal-2}
As before for $\typeBDC$-webs, 
the dot morphisms are only allowed 
if there are no strands to their right, cf.\ \fullref{remark:beware-monoidal} 
(see also below). For example,
\[
\text{Allowed: }
\begin{tikzpicture}[anchorbase,scale=.25,tinynodes]
      \draw[uncolored] (-2,0) node[above] {$1$} to (-2,-3) node[below] {$1$};
      \draw[ccolored] (0,0) node[above] {$2$} to (0,-1.5);
      \node[dotbullet] at (0,-1.5) {};
      \draw[densely dashed] (0,-1.5) to [out=315, in=180] (1,-2);
      \node at (0,1) {\phantom{a}};
      \node at (0,-4) {\phantom{a}};
\end{tikzpicture}
\;\text{ or }\;
\begin{tikzpicture}[anchorbase,scale=.25,tinynodes]
      \draw[ccolored] (-2,0) node[below] {$2$} to (-2,3);
      \node[dotbullet] at (-2,3) {};
      \draw[ccolored] (0,0) node[below] {$2$} to (0,1);
      \node[dotbullet] at (0,1) {};
      \draw[uncolored, densely dotted] (-3,2) to (1,2);
      \draw[densely dashed] (-2,3) to [out=45, in=180] (1,3.5);
      \draw[densely dashed] (0,1) to [out=45, in=180] (1,1.5);
      \node at (0,-1) {\phantom{a}};
      \node at (0,4) {\phantom{a}};
\end{tikzpicture}
\qquad
\text{Not allowed: }
\begin{tikzpicture}[anchorbase,scale=.25,tinynodes]
      \draw[uncolored] (2,0) node[above] {$1$} to (2,-3) node[below] {$1$};
      \draw[ccolored] (0,0) node[above] {$2$} to (0,-1.5);
      \node[dotbullet] at (0,-1.5) {};
      \draw[densely dashed] (0,-1.5) to [out=315, in=180] (3,-2);
      \node at (0,1) {\phantom{a}};
      \node at (0,-4) {\phantom{a}};
      \draw[thick, myred] (-.25,-3) to (2.25,0);
\end{tikzpicture}
\;\text{ or }\;
\begin{tikzpicture}[anchorbase,scale=.25,tinynodes]
      \draw[ccolored] (2,0) node[below] {$2$} to (2,3);
      \node[dotbullet] at (2,3) {};
      \draw[ccolored] (0,0) node[below] {$2$} to (0,1);
      \node[dotbullet] at (0,1) {};
      \draw[uncolored, densely dotted] (3,2) to (-1,2);
      \draw[densely dashed] (2,3) to [out=45, in=180] (3,3.5);
      \draw[densely dashed] (0,1) to [out=45, in=180] (3,1.5);
      \node at (0,-1) {\phantom{a}};
      \node at (0,4) {\phantom{a}};
      \draw[thick, myred] (-.25,0) to (2.25,3);
\end{tikzpicture}
\]
In particular, we get the same restrictions on topological manipulations 
as for $\typeBDC$-webs, and again there will be a representation theoretical 
explanation of this in \fullref{sec-reps}, see also \fullref{remark:not-an-intertwiner-sp}.
Moreover, 
the category $\WebCz$ has the 
same (anti)-involutions as $\WebDz$ (cf.\ \fullref{remark:web-symmetries}), 
which we, abusing notation, denote also by 
$\Psi$ and $\omega$.
\end{remark}

\subsubsection{Topological versions of the \texorpdfstring{$\typeCBD$}{dot}-web relations}

For completeness, we give some topologically meaningful versions 
of the relations above.

\begin{subequations}
\begingroup
\renewcommand{\theequation}{$\typeCBD$\alph{equation}}

\begin{lemma}\label{lemma:barbell-equi-C}
The barbell removal \eqref{eq:barbell-C} is equivalent to
\begin{equation}\label{eq:circle-C}
     \begin{tikzpicture}[anchorbase,scale=.25, tinynodes]
     \draw[uncolored] (-1,2.5) arc (-180:180:1cm);
     \end{tikzpicture}
=\qbracket{\zpar;0}.
\end{equation}
\end{lemma}

\begin{lemma}\label{lemma:thin-k-equi}
The thin K removal \eqref{eq:thin-k} is equivalent to
\begin{equation}\label{eq:thin-k-equi}
      \begin{tikzpicture}[anchorbase,scale=.25, tinynodes]
      \draw[uncolored] (0,4) node[above] {$1$} .. controls ++(0,-3) and ++(0,-1.5) .. ++(2,-2);
      \draw[uncolored,cross line] (0,0) node[below] {$1$} .. controls ++(0,3) and ++(0,1.5) .. ++(2,2);
      \end{tikzpicture} 
=-\zpar^{-1}
      \begin{tikzpicture}[anchorbase,scale=.25, tinynodes]
      \draw[uncolored] (0,0) node[below] {$1$} -- ++(0,4) node[above] {$1$};
      \end{tikzpicture}
\end{equation}
\end{lemma}

\begin{lemma}\label{lemma:thick-k-equi}
The thick K opening \eqref{eq:thick-k} is equivalent to
\begin{equation}\label{eq:thick-k-equi}
     \begin{tikzpicture}[anchorbase,xscale=.25,yscale=.20,tinynodes]
     \draw[uncolored] (6,5.5) .. controls ++(0,-1.5) and ++(0,1.5) .. ++(-4,-4) -- ++(0,-1.5) node[below] {$1$};
     \draw[uncolored,cross line] (0,7) node[above] {$1$} -- ++(0,-1.5)  .. controls ++(0,-1.5) and ++(0,+1.5) .. ++(4,-4) arc (180:360:1cm) .. controls ++(0,1.5) and ++(0,-1.5) .. ++(-4,+4) -- ++(0,1.5) node[above] {$1$};
     \draw[uncolored,cross line] (0,0) node[below] {$1$} -- ++(0,1.5) .. controls ++(0,1.5) and ++(0,-1.5) .. ++(4,4) arc (180:0:1cm); 
     \end{tikzpicture}  
=
     \begin{tikzpicture}[anchorbase,scale=.25,tinynodes]
     \draw[uncolored] (0,0) node[below] {$1$} -- ++(0,.4)  arc (180:0:1cm)  -- ++(0,-.4) node[below] {$1$};
     \draw[uncolored] (0,4) node[above] {$1$} -- ++(0,-.4)   arc (180:360:1cm)  -- ++(0,.4) node[above] {$1$};
     \end{tikzpicture}
\end{equation}
\end{lemma}

\begin{lemma}\label{lemma:kink-removal-C}
The following relations hold:
\begin{gather}\label{eq:kink-C}
        \begin{tikzpicture}[anchorbase,scale=.25,tinynodes]
	    \draw[uncolored] (-1,0) node[above] {$1$} .. controls ++(0,-1) and ++(0,1) .. ++(2,-2)  arc (0:-180:1cm);
        \draw[uncolored, cross line] (-1,-2) .. controls ++(0,1) and ++(0,-1) .. ++(2,2) node[above] {$1$};
        \end{tikzpicture} 
=\qpar
        \begin{tikzpicture}[anchorbase,scale=.25,tinynodes]
	    \draw[uncolored] (-1,0) node[above] {$1$} to [out=270, in=180] (0,-1) to [out=0, in=270] (1,0) node[above] {$1$};
	    \node at (0,-2.2) {$\phantom{.}$}; 
        \end{tikzpicture}
\quad\text{and}\quad
        \begin{tikzpicture}[anchorbase,scale=.25,tinynodes]
        \draw[uncolored] (-1,2) .. controls ++(0,-1) and ++(0,1) .. ++(2,-2) node[below] {$1$};
	    \draw[uncolored, cross line] (-1,0) node[below] {$1$} .. controls ++(0,1) and ++(0,-1) .. ++(2,2)  arc (0:180:1cm);
        \end{tikzpicture} 
=\qpar
        \begin{tikzpicture}[anchorbase,scale=.25,tinynodes]
	    \draw[uncolored] (-1,0) node[below] {$1$} to [out=90, in=180] (0,1) to [out=0, in=90] (1,0) node[below] {$1$};
	    \node at (0,2.6) {$\phantom{.}$}; 
        \end{tikzpicture}
\end{gather}
\end{lemma}

\begin{lemma}\label{lemma:sliding-equi-C}
The merge-split sliding relations \eqref{eq:sliding} are equivalent to
\begin{gather}\label{eq:sliding-equi-C}
      \begin{tikzpicture}[anchorbase,xscale=.25,tinynodes,yscale=0.125]
      \draw[uncolored] (2,0) node[below] {$1$} .. controls ++(0,2) and ++(0,-2) .. ++(-2,4);
      \draw[uncolored,cross line] (0,0) node[below] {$1$} .. controls ++(0,2) and ++(0,-2) .. ++(2,4);
	  \draw[uncolored] (4,0) node[below] {$1$}  -- ++(0,4) .. controls ++(0,2) and  ++(0,-2) .. ++(-2,4);
	  \draw[uncolored,cross line] (2,4)  .. controls ++(0,2) and  ++(0,-2) .. ++(2,4) arc (180:0:1cm and 2cm) -- ++(0,-8) node[below] {$1$} ;
      \draw (0,4) -- ++(0,4) arc (180:0:1cm and 2cm); 
      \node at (1,11) {$\phantom{1}$};
      \end{tikzpicture}
=
      \begin{tikzpicture}[anchorbase,xscale=.25,tinynodes,yscale=0.125]
      \draw[uncolored] (6,0) node[below] {$1$} .. controls ++(0,2) and ++(0,-2) .. ++(-2,4);
      \draw[uncolored,cross line] (4,0) node[below] {$1$} .. controls ++(0,2) and ++(0,-2) .. ++(2,4);
	  \draw[uncolored] (4,4) .. controls ++(0,2) and  ++(0,-2) .. ++(-2,4);
	  \draw[uncolored,cross line] (2,0) node[below] {$1$} -- ++(0,4)  .. controls ++(0,2) and  ++(0,-2) .. ++(2,4) arc (180:0:1cm and 2cm) ;
      \draw (0,0) node[below] {$1$} -- ++(0,8) arc (180:0:1cm and 2cm);
      \draw (6,8) -- ++(0,-4);
      \node at (1,11) {$\phantom{1}$};
      \end{tikzpicture}
\quad\text{and}\quad
      \begin{tikzpicture}[anchorbase,xscale=-.25,tinynodes,yscale=-0.125]
      \draw[uncolored] (6,0) node[above] {$1$} .. controls ++(0,2) and ++(0,-2) .. ++(-2,4);
      \draw[uncolored,cross line] (4,0) node[above] {$1$} .. controls ++(0,2) and ++(0,-2) .. ++(2,4);
	  \draw[uncolored] (4,4) .. controls ++(0,2) and  ++(0,-2) .. ++(-2,4);
	  \draw[uncolored,cross line] (2,0) node[above] {$1$} -- ++(0,4)  .. controls ++(0,2) and  ++(0,-2) .. ++(2,4) arc (180:0:1cm and 2cm) ;
      \draw (0,0) node[above] {$1$} -- ++(0,8) arc (180:0:1cm and 2cm);
      \draw (6,8) -- ++(0,-4);
      \node at (1,11.25) {$\phantom{1}$};
      \end{tikzpicture}
=
      \begin{tikzpicture}[anchorbase,xscale=-.25,tinynodes,yscale=-0.125]
      \draw[uncolored] (2,0) node[above] {$1$} .. controls ++(0,2) and ++(0,-2) .. ++(-2,4);
      \draw[uncolored,cross line] (0,0) node[above] {$1$} .. controls ++(0,2) and ++(0,-2) .. ++(2,4);
	  \draw[uncolored] (4,0) node[above] {$1$}  -- ++(0,4) .. controls ++(0,2) and  ++(0,-2) .. ++(-2,4);
	  \draw[uncolored,cross line] (2,4)  .. controls ++(0,2) and  ++(0,-2) .. ++(2,4) arc (180:0:1cm and 2cm) -- ++(0,-8) node[above] {$1$} ;
      \draw (0,4) -- ++(0,4) arc (180:0:1cm and 2cm); 
      \node at (1,11.25) {$\phantom{1}$};
      \end{tikzpicture}
\end{gather}
\end{lemma}

\endgroup
\end{subequations}
\addtocounter{equation}{-1}

\makeautorefname{lemma}{Lemmas} 

\begin{proof}[Proof of \fullref{lemma:barbell-equi-C}, \ref{lemma:thin-k-equi}, \ref{lemma:thick-k-equi}, \ref{lemma:kink-removal-C} 
and \ref{lemma:sliding-equi-C}]\renewcommand{\qedsymbol}{}
Again, the equations can be 
checked by expanding the crossings 
using \eqref{eq:braiding} (although it requires some time 
and patience to verify 
that \eqref{eq:thick-k-equi} is equivalent to \eqref{eq:thick-k}).
Let us check one as an example, 
showing that \eqref{eq:circle-C} and \eqref{eq:thin-k} 
imply \eqref{eq:thin-k-equi}:
\begin{gather}\label{eq:lemma-calc}
	  \begin{tikzpicture}[anchorbase,scale=.25, tinynodes]
      \draw[uncolored] (0,4) node[above] {$1$} .. controls ++(0,-3) and ++(0,-1.5) .. ++(2,-2);
      \draw[uncolored,cross line] (0,0) node[below] {$1$} .. controls ++(0,3) and ++(0,1.5) .. ++(2,2);
      \end{tikzpicture} 
=-\qpar^{-1}
      \begin{tikzpicture}[anchorbase,scale=.25, tinynodes]
      \draw[uncolored] (0,4) node[above] {$1$} to (0,0) node[below] {$1$};
      \draw[uncolored] (.75,2) arc (-180:180:1cm);
      \end{tikzpicture}
+
      \begin{tikzpicture}[anchorbase,scale=.25, tinynodes]
      \draw[uncolored] (0,4) node[above] {$1$} to (0,2.5);
      \draw[uncolored] (0,0) node[below] {$1$} to (0,1.5);
      \draw[ccolored] (0,2.5) to (0,1.5);
      \draw[uncolored] (2,3) to (2,1);
      \draw[uncolored] (0,2.5) to (2,3);
      \draw[uncolored] (0,1.5) to (2,1);
      \draw[ccolored] (2,3) to (2,3.75);
      \draw[ccolored] (2,1) to (2,.25);
      \node[dotbullet] at (2,3.75) {};
	  \node[dotbullet] at (2,.25) {};
      \end{tikzpicture}
\overset{\eqref{eq:square}}{=}
-\qpar^{-1}
      \begin{tikzpicture}[anchorbase,scale=.25, tinynodes]
      \draw[uncolored] (0,4) node[above] {$1$} to (0,0) node[below] {$1$};
      \draw[uncolored] (.75,2) arc (-180:180:1cm);
      \end{tikzpicture}
+
      \begin{tikzpicture}[anchorbase,scale=.25, tinynodes]
      \draw[uncolored] (0,4) node[above] {$1$} to (0,3);
      \draw[uncolored] (0,0) node[below] {$1$} to (0,1);
      \draw[colored] (2,2.5) to (2,1.5);
      \draw[uncolored] (0,3) to (2,2.5);
      \draw[uncolored] (0,1) to (2,1.5);
      \draw[ccolored] (2,2.5) to (2,3.75);
      \draw[ccolored] (2,1.5) to (2,.25);
      \node[dotbullet] at (2,3.75) {};
	  \node[dotbullet] at (2,.25) {};
      \end{tikzpicture}
+
      \begin{tikzpicture}[anchorbase,scale=.25, tinynodes]
      \draw[uncolored] (0,4) node[above] {$1$} to (0,0) node[below] {$1$};
      \draw[ccolored] (2,3.75) to (2,2) to (2,.25);
      \node[dotbullet] at (2,3.75) {};
	  \node[dotbullet] at (2,.25) {};
      \end{tikzpicture}
\overunder{\eqref{eq:circle-C}}{\eqref{eq:thin-k}}{=}\!\!\!
-\zpar^{-1}
      \begin{tikzpicture}[anchorbase,scale=.25, tinynodes]
      \draw[uncolored] (0,4) node[above] {$1$} to (0,0) node[below] {$1$};
    \end{tikzpicture}
\hspace*{2.2cm}\qedmake\hspace*{-2.2cm}
\end{gather}
\end{proof}

\makeautorefname{lemma}{Lemma} 

Again, by using $\Psi$ and 
$\omega$, we obtain many more equivalent relations.

\subsubsection{Why \texorpdfstring{$\typeCBD$}{dot}-webs do not form a monoidal category}

Again, as for
$\typeBDC$-webs,
the $\typeCBD$-web category is not monoidal.
As will become clear later, this is related to the fact that
the inclusion $\fraksp_{n} \hookrightarrow \frakgl_n$ can
only be quantized naturally as an inclusion of
a coideal subalgebra. However, de-quantization gives again a 
genuine monoidal, topologically flavored category of $\typeCBD$-webs.

\subsection{Some more useful lemmas}\label{subsec-lemmas-C}

Until the end of the section  we 
work in $\WebCz$, and
we derive some diagrammatic relations which,
as we will see later, correspond to  
relations in the quantum group $\quantumg(\fraksp_{2k})$.

The philosophy is again to 
``manipulate the $\typeA$-web part and to keep dots where they are''.

\begin{lemma}\label{lemma:typeC-EF-FE}
For all $a$ we have
\begin{equation}\label{eq:typeC-EF-FE}

\end{align*}
A tedious but straightforward computation gives the claimed coefficients.
\end{proof}

\begin{lemma}\label{lemma-typeC-EF}
For all $a,b$ we have
\begin{equation}\label{eq:typeC-EF}
      %
\end{equation}
\end{lemma}

\begin{proof}
Clear by associativity \eqref{eq:asso}.
\end{proof}

\begin{lemma}\label{lemma-typeC-Serre}
For all $a,b$ we have
\begin{equation}\label{eq:typeC-Serre}
      %
\end{equation}
\end{lemma}

\begin{proof}
First, we note that \eqref{eq:typeC-Serre} is equivalent to
\[
      %
\]
Next, we can apply the $\typeA$-web Serre relations 
(cf.\ \fullref{remark:web-Serre}) to the marked part (for $s=2$, $t=3$) and we are done.
\end{proof}

\begin{lemma}\label{lemma:useful-lemma-for-Serre}
We have
\begin{gather}\label{eq:small-label-lemma}
        %
\end{gather}
\end{lemma}

\begin{proof}
The first equation is equivalent to the 
merge-split sliding relation \eqref{eq:typeC-sliding} through the chain of equalities
\begin{gather*}
      %
\end{gather*}
Using \eqref{eq:digon}, the second equality is an immediate 
consequence of the first one.
\end{proof}

\begin{lemma}\label{lemma:some-typeC-lemma}
We have
\begin{gather}\label{eq:4-label-lemma}
        %
\end{gather}
\end{lemma}

\begin{proof}
We compute, using \eqref{eq:small-label-lemma}, that
\begin{gather*}
      %
\end{gather*}
Noting that $\qbracket{3}-1=\qbracketC{2}$, we are done.
\end{proof}

\begin{lemma}\label{lemma-typeC-Serre-2}
For all $a,b$ we have
\begin{equation}\label{eq:typeC-Serre-2}
\qbracketC{2}
      %
\end{equation}
\end{lemma}

\begin{proof}
The proof is a repeated application of 
the $\typeA$-web Serre relations 
(cf.\ \fullref{remark:web-Serre}). We always indicate where we apply these 
and for what values of $s,t$.

We start by applying these for $s=1$, $t=4$ as follows.
\begin{align*}
    %
\end{align*}
Similarly, but for $s=1$, $t=3$, we can rewrite the second term as
\begin{gather*}
    %
\end{gather*}
Combining these gives
\begin{align*}
    %
\end{align*}
Next and as before, this time with $s=1$, $t=1$, we get for the second term
\begin{gather*}
    %
\end{gather*}
Again, by combining this with the above we get
\begin{align*}
    %
\end{align*}
We can rewrite this as
\begin{align*}
    %
\end{align*}
On the other side, using now the $s=1$, $t=1$ case, we have
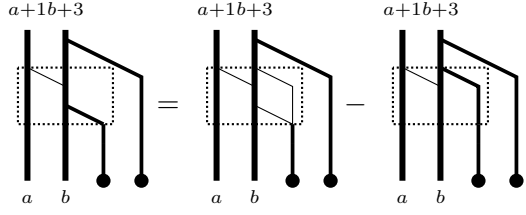
\begin{gather*}
    %
\end{gather*}
Putting everything together, we get the claimed equality.
\end{proof}

\section{Representation theoretical background}\label{sec-reps}
In this section we fix our conventions for the quantum enveloping algebras
and recall the definition of the 
coideal subalgebras $\coideal(\frakso_n)$ and $\coideal(\fraksp_n)$.
We will also consider their vector representations, 
the associated exterior and symmetric powers, and construct some intertwiners.

\subsection{Quantum enveloping algebras}\label{subsec:conventionsqgroup}

Let $\frakg$ be a reductive Lie algebra with 
simple roots $(\alpha_i)_{i\in I}$, simple coroots $(h_i)_{i \in I}$
and weight lattice $\weightlattice$. 
Denote by $a_{ij} = \langle h_i, \alpha_j \rangle$ 
the entries of the Cartan matrix, and by $d_i\in\N$ 
the minimal values such that the matrix $(d_ia_{ij})_{i,j\in\I}$ 
is symmetric and positive definite, see 
also below.

Throughout, all indices are always 
from the evident sets, e.g.\ if we write 
$\sfE_i$, then we always assume that $i \in \I$.

\begin{definition}\label{definition-quantumgroup}
The \emph{quantum enveloping algebra} 
$\quantumg(\frakg)$ \emph{of} $\frakg$ is 
the associative, unital $\CQ$-algebra 
generated by $\quantumq^h$ for 
$h\in \weightlattice^*$, and by
$\sfE_i$, $\sfF_i$ for 
$i \in\I$, 
subject to:
\begin{gather}
    \quantumq^0=1, \quad
    \quantumq^h \quantumq^{h^{\prime}}  = \quantumq^{h+h^{\prime}}, \quad
    \quantumq^h \sfE_i =\qpar^{\langle h, \alpha_i \rangle} \sfE_i \quantumq^h, \quad 
    \quantumq^h \sfF_i=\qpar^{-\langle h,\alpha_i \rangle}\sfF_i\quantumq^h,
\label{eq:qh-rel}
\end{gather}
\begin{equation}
 \sfE_i\sfF_j - \sfF_j\sfE_i = \delta_{ij} \frac{\sfK^{\phantom{1}}_i-\sfK_i^{-1}}{\qpar^{\phantom{1}}_i-\qpar_i^{-1}},\label{eq:EF-rel}
\end{equation}
\begin{align}
{\textstyle\sum_{v=0}^{1-a_{ij}}} (-1)^v \qbin{1-a_{ij}}{v}_{i} \sfE_i^{1-a_{ij}-s} \sfE_{j}^{\phantom{1}} \sfE_i^v & = 0, \qquad \text{for } i \neq j, \label{eq:E-Serre}
\\ 
{\textstyle\sum_{v=0}^{1-a_{ij}}} (-1)^v \qbin{1-a_{ij}}{v}_{i} \sfF_i^{1-a_{ij}-s} \sfF_{j}^{\phantom{1}} \sfF_i^v & = 0, \qquad \text{for } i \neq j.\label{eq:F-Serre}
\end{align}
The latter two relations are the so-called \emph{Serre relations}.
Here, $\sfK_i = \quantumq^{d_i h_i}$
and the quantum binomials are as in \eqref{eq:qnumbers-typeAD}.
\end{definition}

\subsubsection{Root and weight conventions}
Before we can give our key examples of 
\fullref{definition-quantumgroup}, we fix some conventions 
which will be important for explicit computations.

Fix $m \in \Z_{\geq 1}$, the \emph{rank} 
(which usually will be denoted $k$ or $n$, depending 
on which side of Howe duality we are, cf. \fullref{sec-main}).
Let $\frakg=\fraksp_{2m}$ 
or $\frakg=\frakso_{2m}$, and we denote by 
$\roots$ and $\simpleroots$ the sets of \emph{roots}
and \emph{simple roots}, which we choose accordingly to \fullref{table:types}.
Here $\{\epsilon_i \mid i \in\I\}$ for $\I=\{1,\dots,m\}$ denotes a 
chosen basis of the dual of the \emph{Cartan} $\frakh$, 
which is orthonormal with respect to the \emph{Killing form} $(\cdot,\cdot)$.
Correspondingly, we have the \emph{weight lattice} $\weightlattice$ and \emph{dominant integral weights} $\weightlattice^+$.
We let also as usual $\{h_i \in \frakh \mid i \in\I\}$ 
be the basis of $\frakh$
determined by $\langle h_i , \lambda \rangle = 2\frac{(\alpha_i,\lambda)}{(\alpha_i,\alpha_i)}$ 
for $\lambda\in\weightlattice$.
Moreover, recall that the \emph{Cartan matrix} $A=(a_{ij})_{i,j\in\I}$ 
is defined via $a_{ij} = \langle h_i , \alpha_j \rangle$.
The sequence $(d_1,\dots,d_m)$ is chosen with $d_i\in\N$ 
 for $i=1,\dots,m$ minimal such that the matrix $(d_ia_{ij})_{i,j\in\I}$ 
is symmetric and positive definite.
(The Cartan datum can also be read off from the 
corresponding \emph{Dynkin diagram} $D$.)

We do not need to fix a Cartan datum for type $\typeB$, 
since in this paper we only encounter the type $\typeB$ Lie algebra 
$\frakso_n$ (for $n$ odd) in the coideal $\coideal(\frakso_n)$, 
and never in the quantum enveloping algebra $\quantumg(\frakso_n)$.

\begin{table} \tabulinesep=.35ex
\noindent\hspace{-0.1\textwidth}\begin{tabu} to 0.9\textwidth {c|X[c]|X[c]}
& $\fraksp_{2m}$ & $\frakso_{2m}$ 
\\ 
\hline $D$ & 
    \begin{tikzpicture}[baseline={([yshift=-0.5ex]0,0.01)},scale=.65,tinynodes]
	\draw [thick, red] (1.75,0.01) to (2.75,0.01);
 	\draw [thick, red] (3.25,0.01) to (3.45,0.01);
 	\draw [thick, red] (4.05,0.01) to (4.25,0.01);
 	\draw [thick] (4.75,0.11) to (5.75,0.11);
 	\draw [thick] (4.75,-0.09) to (5.75,-0.09);
 	\draw [thick] (5.35,0.26) to (5.15,0.01);
 	\draw [thick] (5.35,-0.24) to (5.15,0.01);
 	\draw [thick, red] (1.5,0) circle (0.25cm);
 	\draw [thick, red] (3,0) circle (0.25cm);
 	\draw [thick, red] (4.5,0) circle (0.25cm);
 	\draw [thick] (6,0) circle (0.25cm);
 	\node at (3.8,0) {$\cdots$};
 	\node at (1.5,0.5) {$\alpha_1$};
 	\node at (3,0.5) {$\alpha_2$};
 	\node at (4.5,0.5) {$\alpha_{m{-}1}$};
 	\node at (6,0.5) {$\alpha_{m}$};
	\end{tikzpicture}
&
    \begin{tikzpicture}[anchorbase,scale=.65,tinynodes]
	\draw [thick, red] (1.75,0.01) to (2.75,0.01);
 	\draw [thick, red] (3.25,0.01) to (3.45,0.01);
 	\draw [thick, red] (4.05,0.01) to (4.25,0.01);
 	\draw [thick, red] (4.75,0.01) to (5.75,0.01);
 	\draw [thick, red] (6.138,0.161) to (6.65,1.1);
 	\draw [thick] (6.144,-0.169) to (6.65,-1.1);
 	\draw [thick, red] (1.5,0) circle (0.25cm);
 	\draw [thick, red] (3,0) circle (0.25cm);
 	\draw [thick, red] (4.5,0) circle (0.25cm);
 	\draw [thick, red, fill=white] (6,0) circle (0.25cm);
 	\draw [thick, red] (6.75,1.3) circle (0.25cm);
 	\draw [thick] (6.75,-1.3) circle (0.25cm);
 	\node at (3.8,0) {$\cdots$};
 	\node at (1.5,0.5) {$\alpha_1$};
 	\node at (3,0.5) {$\alpha_2$};
 	\node at (4.5,0.5) {$\alpha_{m-3}$};
 	\node at (5.65,0.5) {$\alpha_{m-2}$};
 	\node at (5.95,1.3) {$\alpha_{m-1}$};
 	\node at (6.2,-1.3) {$\alpha_{m}$};
	\end{tikzpicture}
\\ \hline  
$A$ & 
$\left( 
\begin{tabular}{ c c c c c}
  $2$ & $-1$ & $0$ & \multicolumn{1}{c!{\vline}}{$\cdots$} & $0$ \\
  $-1$ & $2$ & $\ddots$ & \multicolumn{1}{c!{\vline}}{$\ddots$} & $\vdots$ \\
  $0$ & $\ddots$ & $\ddots$ & \multicolumn{1}{c!{\vline}}{$-1$} & $0$ \\
   $\vdots$ & $\ddots$ & $-1$ & \multicolumn{1}{c!{\vline}}{$2$} & $-2$ \\
   \cline{1-4}
   $0$ & $\cdots$ & $0$ & $-1$ & $2$
\end{tabular} 
\right)$
& 
$\left(
\begin{tabular}{ c c c c c}
  $2$ & $-1$ & $0$ & \multicolumn{1}{c!{\vline}}{$\cdots$} & $0$ \\
  $-1$ & $2$ & $\ddots$ & \multicolumn{1}{c!{\vline}}{$\ddots$} & $\vdots$ \\
  $0$ & $\ddots$ & $\ddots$ & \multicolumn{1}{c!{\vline}}{$-1$} & $-1$ \\
   $\vdots$ & $\ddots$ & $-1$ & \multicolumn{1}{c!{\vline}}{$2$} & $0$ \\
   \cline{1-4}
   $0$ & $\cdots$ & $-1$ & $0$ & $2$
\end{tabular}
\right)$
\\ \hline 
$\vec{d}$  & $(1,\dots,1,2)$ & $(1,\dots,1)$ 
\\ \hline $\simpleroots$  & $ \begin{aligned}
  \alpha_1 &= \epsilon_1 - \epsilon_{2}\\
   & \shortvdotswithin{=} \alpha_{m-1} & = \epsilon_{m-1} - \epsilon_m\\
  \alpha_m &= 2\epsilon_m %
\end{aligned}$ & $ \begin{aligned}
  \alpha_1 &= \epsilon_1 - \epsilon_{2}\\
   & \shortvdotswithin{=} \alpha_{k-1} & = \epsilon_{m-1} - \epsilon_m\\
  \alpha_m &= \epsilon_{m-1} + \epsilon_m %
\end{aligned}$ 
\\ \hline
$\roots$ &
$\{ \pm 2\epsilon_i, \pm \epsilon_i \pm \epsilon_j \mid i\neq j\in\I\}$ &
$\{ \pm \epsilon_i \pm \epsilon_j \mid i\neq j\in\I\}$
\\ \hline 
$\weightlattice$  &
 $\Z^m$ & $\Z^m \oplus \big(\tfrac{1}{2}+\Z\big)^m$ 
\\ \hline
$\weightlattice^+$  &
$\begin{multlined}\{(\lambda_1,\dots,\lambda_m) \in 
\weightlattice \mid \\ \lambda_1 \geq \dots \geq \lambda_{m} \geq 0\}\end{multlined}$ &
$\begin{multlined}\{(\lambda_1,\dots,\lambda_m) \in 
\weightlattice \mid \\ \lambda_1 \geq \dots \geq \lambda_{m-1} 
\geq \abs{\lambda_m}\}\end{multlined}$
\\
\end{tabu} \vspace{1ex}
\caption{Our conventions for types $\typeC_m$ and $\typeD_m$. 
Here we also specify the type $\typeA_{m-1}$ Cartan datum 
by considering the subgraphs of $D$ and the submatrix of $A$
as indicated (for both types).}
\label{table:types}
\end{table}

\begin{example}\label{example:Serre-relations}
Besides $\frakgl_n$, we will consider the cases $\frakg=\fraksp_{2k}$ 
and $\frakg=\frakso_{2k}$ 
with conventions fixed above.
The corresponding Serre relations for the $\sfE_i$'s are
\begin{gather}
\sfE_{k-1}^3\sfE_k^{\phantom{2}}
+
[3]\sfE_{k-1}^{\phantom{2}}\sfE_k^{\phantom{2}}\sfE_{k-1}^2
=
\sfE_k^{\phantom{2}}\sfE_{k-1}^3
+
[3]\sfE_{k-1}^2\sfE_k^{\phantom{2}}\sfE_{k-1}^{\phantom{2}},\label{eq:serre-sp-1}
\\
[2]_2\sfE_k^{\phantom{2}}\sfE_{k-1}^{\phantom{2}}\sfE_k^{\phantom{2}}
=
\sfE_{k-1}^{\phantom{2}}\sfE_k^2
+
\sfE_k^2\sfE_{k-1}^{\phantom{2}}\label{eq:serre-sp-2}
\end{gather}
in case $\frakg=\fraksp_{2k}$, and for $\frakg=\frakso_{2k}$ they are
\begin{gather}
\sfE_{k-1}^{\phantom{2}}\sfE_k^{\phantom{2}}
=
\sfE_k^{\phantom{2}}\sfE_{k-1}^{\phantom{2}},\label{eq:serre-so-1}
\\
[2]\sfE_{k-2}^{\phantom{2}}\sfE_{k}^{\phantom{2}}\sfE_{k-2}^{\phantom{2}}
=
\sfE_{k-2}^2\sfE_{k}^{\phantom{2}}
+
\sfE_{k}^{\phantom{2}}\sfE_{k-2}^2,\label{eq:serre-so-2}
\\
[2]\sfE_{k}^{\phantom{2}}\sfE_{k-2}^{\phantom{2}}\sfE_{k}^{\phantom{2}}
=
\sfE_{k-2}^{\phantom{2}}\sfE_k^2
+
\sfE_k^2\sfE_{k-2}^{\phantom{2}}.\label{eq:serre-so-3}
\end{gather}
Additionally, there are versions 
involving $\sfF_k$'s, and the type $\typeA$ Serre relations:
\begin{gather*}
[2]\sfE_{i}^{\phantom{2}}\sfE_{i-1}^{\phantom{2}}\sfE_{i}^{\phantom{2}}
=
\sfE_{i-1}^{\phantom{2}}\sfE_i^2
+
\sfE_i^2\sfE_{i-1}^{\phantom{2}},
\end{gather*}
where $i$ is not $k$.
\end{example}

As usual, we define 
the \emph{divided powers} 
\[
\sfE_i^{(s)} = \tfrac{1}{[s]_{i} !}\sfE_i^s
\quad\text{ and }\quad 
\sfF_i^{(s)} = \tfrac{1}{[s]_{i} !}\sfF_i^s, 
\qquad s\in\N.
\] 
One can then show that 
the
\emph{higher order Serre relations}
\begin{gather}\label{eq:higherSerre}
\begin{aligned} 
{\textstyle\sum_{u+v=t}} (-1)^v \qpar_i^{\varepsilon u(-a_{ij} s - t + 1)} \sfE_i^{(u)} \sfE_{j}^{(s)} \sfE_i^{(v)} & = 0, 
\qquad \text{for } i \neq j,\\
{\textstyle\sum_{u+v=t}} (-1)^v \qpar_i^{\varepsilon u(-a_{ij} s - t +1)} \sfF_i^{(u)} \sfF_{j}^{(s)} \sfF_i^{(v)} & = 0,
\qquad \text{for } i \neq j,
\end{aligned}
\end{gather}
hold for $\varepsilon = \pm 1$, for all $s,t \in \Z$ with $s \geq 1$ 
and $t > -a_{ij}$ 
(see e.g.\ \cite[Chapter 7]{Lus} 
and in particular Proposition 7.1.5 therein).

Moreover, recall that $\quantumg(\frakg)$ has the 
structure of a Hopf algebra.
We use the following 
conventions for the
comultiplication 
$\Delta \colon \quantumg(\frakg) \rightarrow \quantumg(\frakg) \otimes \quantumg(\frakg)$, 
the counit $\counit \colon  \quantumg(\frakg) \rightarrow \CQ$ and the 
antipode $S \colon \quantumg(\frakg) \rightarrow \quantumg(\frakg)$:
\begin{gather}\label{eq:conventionsHopf}
\begin{aligned}
\Delta(\quantumq^h) = \quantumq^h \otimes \quantumq^h,\;\,
  \Delta(\sfE_i) = \sfE_i \otimes \sfK_i &+ 1 \otimes \sfE_i,\;\,
  \Delta(\sfF_i) = \sfF_i \otimes 1 + \sfK_i^{-1} \otimes \sfF_i,\\
 \counit(\quantumq^h)  = 1,\quad &\counit(\sfE_i)  = 0,\quad  \counit(\sfF_i)  = 0,\\ 
S(\quantumq^h)  = \quantumq^{-h},\quad
S(\sfE_i)  &= - \sfE_i \sfK_i^{-1},\quad  S(\sfF_i)  = - \sfK_i \sfF_i.
\end{aligned}
\end{gather}

\subsubsection{The idempotented versions}

Next, following \cite[Chapter 23]{Lus}, we define:

\begin{definition}\label{definition:udot}
The \emph{idempotented quantum enveloping algebra} $\Udot(\frakg)$ is the 
additive closure 
of the $\CQ$-linear
category with:
\smallskip
\begin{enumerate}[label=$\vartriangleright$]

\setlength\itemsep{.15cm}

\item objects $\smallone_{\lambda}$ for 
$\lambda \in \weightlattice$, and
\item morphisms $\Hom_{\Udot(\frakg)}(\smallone_{\lambda},\smallone_{\mu}) = \quantumg(\frakg)/ I_{\lambda,\mu}$, where
\begin{gather*}
I_{\lambda,\mu} =
{\textstyle\sum_{h \in \weightlattice^{*}}}  \quantumg(\frakg)(\quantumq^h - \qpar^{\langle h, \lambda\rangle}) 
+ 
{\textstyle\sum_{h \in \weightlattice^{*}}} (\quantumq^h - \qpar^{\langle h, \mu \rangle})\quantumg(\frakg).
\end{gather*}
\end{enumerate}
\end{definition}

The reader unfamiliar with the idempotented version of $\quantumg(\frakg)$
in its categorical disguise is 
referred to \cite[\textsection{}4.1]{CKM},
whose type $\typeA$ treatment immediately generalizes to a general $\frakg$. 
Sometimes it is also convenient to regard 
$\Udot(\frakg)$ as an algebra, and we 
use both viewpoints interchangeably.

We denoted the morphism of $\Udot(\frakg)$ by 
$\dotX\smallone_{\lambda}=\smallone_{\mu}\dotX\smallone_{\lambda}
\in\Hom_{\Udot(\frakg)}(\smallone_{\lambda},\smallone_{\mu})
$ 
for $\dotX$ being some product 
of $\dotE_i$'s and $\dotF_i$'s, and appropriate $\lambda$ and $\mu$. 
In particular,
\[
\dotE_i\smallone_{\lambda}
\in\Hom_{\Udot(\frakg)}(\smallone_{\lambda},\smallone_{\lambda{+}\alpha_i})
\quad\text{and}\quad
\dotF_i\smallone_{\lambda}
\in\Hom_{\Udot(\frakg)}(\smallone_{\lambda},\smallone_{\lambda{-}\alpha_i}).
\]
(Note that we write $\sfE_i$
etc.\ for elements of 
$\quantumg(\frakg)$, and $\dotE_i\smallone_{\lambda}$
etc.\ for $\Udot(\frakg)$.)

\subsubsection{The quantum enveloping algebra \texorpdfstring{$\quantumg(\frakgl_n)$}{Uq(gl(n))}}

We denote by $\Repq{n}$ 
the braided monoidal
category of finite-dimensional representations of $\quantumg(\frakgl_n)$.
Let us recall some 
basic facts about some representations 
of $\quantumg(\frakgl_n)$.

We denote by $\trivmod=\CQ$ the trivial 
and by $\vecrep$ the (quantum analog of the) vector representation of 
$\quantumg(\frakgl_n)$. On the standard basis $v_1,\dots,v_n$ of $\vecrep$,
the action of the generators is explicitly given by
\begin{gather*}
\sfK^{\pm 1}_iv_j=\begin{cases}\qpar^{\pm 1}v_{j}, &\text{if }i=j,\\
\qpar^{\mp 1}v_{j}, &\text{if }i=j-1,\\
v_j, &\text{else,}\end{cases}\\
\sfE_iv_j=\begin{cases}v_{j-1}, &\text{if }i=j-1,\\0, &\text{else,}\end{cases}\qquad
\sfF_iv_j=\begin{cases}v_{j+1}, &\text{if }i=j,\\0, &\text{else.}\end{cases}
\end{gather*}

As usual, we define the 
\emph{($\qpar$-)exterior algebra} of $\vecrep$ as
\begin{gather*}
\exterior{\bullet}\vecrep = \mathrm{T} \vecrep / \langle\mathrm{S}_{\qpar}^2\vecrep\rangle,
\end{gather*}
where 
$\mathrm{T}\vecrep$ denotes the tensor algebra of $\vecrep$
and
$\mathrm{S}_{\qpar}^2\vecrep\subset\vecrep \otimes \vecrep$ is the 
$\CQ$-linear subspace spanned by
\begin{gather}\label{eq:sym-ideal}
v_i \otimes v_i, \text{ for all }i = 1,\dots,n,
\quad\text{and}\quad
\qpar^{-1}v_i \otimes v_j + v_j \otimes v_i, \text{ for all } i < j.
\end{gather}
Since $\mathrm{T} V$
is naturally graded and the ideal $\langle\mathrm{S}_{\qpar}^2\vecrep\rangle$ is 
homogeneous, $\exterior{\bullet}\vecrep$ is also 
graded and decomposes as a $\quantumg(\frakgl_n)$-module
as $\bigoplus_{a \in \N} \exterior{a}\vecrep$, 
with $\exterior{0}\vecrep\cong\trivmod$ 
and $\exterior{1}\vecrep\cong\vecrep$.
We call $\exterior{a}\vecrep$ the $a$th 
\emph{exterior power} (of $\vecrep$), and
we
write $v_{i_1}\wedge\cdots\wedge v_{i_a}$ for 
the image of $v_{i_1} \otimes \dotsb \otimes v_{i_a}$ 
in the quotient $\exterior{a}\vecrep$. 

Similarly, we define the \emph{($\qpar$-)symmetric algebra} as
\begin{gather*}
\symmetric{\bullet}\vecrep = \mathrm{T} \vecrep / \langle\mathrm{E}_{\qpar}^2\vecrep\rangle,
\end{gather*}
where 
$\mathrm{E}_{\qpar}^2\vecrep\subset\vecrep \otimes \vecrep$ is spanned by
\begin{gather}\label{eq:ext-ideal}
\qpar v_i \otimes v_j - v_j \otimes v_i, \text{ for all } i < j.
\end{gather}
As before, we have a $\quantumg(\frakgl_n)$-module
decomposition $\symmetric{\bullet}\vecrep=\bigoplus_{a \in \N} \symmetric{a}\vecrep$, 
with $\symmetric{0}\vecrep\cong\trivmod$ 
and $\symmetric{1}\vecrep\cong\vecrep$.
We call $\symmetric{a}\vecrep$ the $a$th 
\emph{symmetric power} (of $\vecrep$). 
We
write $v_{i_1}\cdots v_{i_a}$ for 
the corresponding element of $\symmetric{a}\vecrep$.

Clearly,  $\exterior{a}\vecrep$ and 
$\symmetric{a}\vecrep$ are $\CQ$-linearly spanned by elements 
of the form
\[
v_{i_1}\wedge\cdots\wedge v_{i_a},\quad i_1<\cdots<i_a,
\quad\text{and}\quad
v_{i_1} \cdots v_{i_a},\quad i_1\leq\cdots\leq i_a.
\]
Henceforth, we will always assume that the 
indices are increasing (strictly increasing in the exterior and weakly 
increasing in the symmetric case).

The multiplication of the tensor algebra $\mathrm{T}\vecrep$ 
is clearly $\quantumg(\frakgl_n)$-equivariant, 
and therefore induces $\quantumg(\frakgl_n)$-equivariant 
multiplications on $\exterior{\bullet}\vecrep$ and $\symmetric{\bullet}\vecrep$.
Moreover, both $\exterior{\bullet}\vecrep$ 
and $\symmetric{\bullet}\vecrep$ are coalgebras, 
with $\quantumg(\frakgl_n)$-equivariant comultiplications.
(This follows from Howe duality in type $\typeA$, 
see \cite[Lemma 3.1.2]{CKM} for $\exterior{\bullet}\vecrep$ and 
\cite[Lemma 2.21]{RT} for $\symmetric{\bullet}\vecrep$.)
Thus,
we can define $\quantumg(\frakgl_n)$-equivariant maps
\begin{gather*}
\begin{aligned}
\extmerge_{a,b}^{a{+}b}\colon
\exterior{a}\vecrep\otimes\exterior{b}\vecrep\to\exterior{a{+}b}\vecrep
\quad&\text{and}\quad
\symmerge_{a,b}^{a{+}b}\colon
\symmetric{a}\vecrep\otimes\symmetric{b}\vecrep\to\symmetric{a{+}b}\vecrep,\\
\extsplit_{a{+}b}^{a,b}\colon
\exterior{a{+}b}\vecrep\to\exterior{a}\vecrep\otimes\exterior{b}\vecrep
\quad&\text{and}\quad
\symsplit_{a{+}b}^{a,b}\colon
\symmetric{a{+}b}\vecrep\to\symmetric{a}\vecrep\otimes\symmetric{b}\vecrep,
\end{aligned}
\end{gather*}
to be the corresponding (co)multiplications.

\begin{remark}\label{remark:color-code}
In order to facilitate the distinction between the exterior and the symmetric power,
we use
the color code from \cite{TVW}, 
i.e.\ ``exterior=red'' and ``symmetric=green'' 
(with ``black=$\exterior{1}\vecrep=\vecrep=\symmetric{1}\vecrep$''). However, 
our web categories are ``red and green at the same time'' 
(cf.\ \fullref{figure:commute}), 
so we do not color their webs.
\end{remark}

\begin{example}\label{example:merge-split}
The base cases of the $\quantumg(\frakgl_n)$-intertwiners from 
above are the ones with $a=b=1$. In these cases we 
omit the sub- and superscripts and we have
\begin{gather*}
\extsplitt\colon
\exterior{2}\vecrep\to\vecrep\otimes\vecrep,
\quad v_i\wedge v_j\mapsto \qpar v_i\otimes v_j - v_j\otimes v_i,\\
\symsplitt\colon
\symmetric{2}\vecrep\to\vecrep\otimes\vecrep,
\quad v_iv_j\mapsto
\begin{cases}
\qpar^{-1}v_i\otimes v_j + v_j\otimes v_i, & \text{for }i<j,\\
\qbracket{2} v_i\otimes v_i, & \text{for }i=j.
\end{cases}
\end{gather*}
\end{example}

\subsection{The coideal subalgebra \texorpdfstring{$\coideal(\frakso_n)$}{corresponding to so(n)}}\label{subsec-coideal}

Next, we recall the definition of the coideal subalgebra 
$\coideal(\frakso_n)$ of $\quantumg(\frakgl_n)$, following \cite[Section 3]{KP}.

\begin{definition}\label{definition-typeD-coideal}
Let $\coideal(\frakso_n)$ be the 
$\CQ$-subalgebra 
of $\quantumg(\frakgl_n)$ generated by 
\begin{gather}\label{eq:gen-type-D}
\sfB_i = \sfF_i - \sfK_i^{-1}\sfE_i,\qquad\text{for }i=1,\dots,n.
\end{gather}
\end{definition}

\begin{remark}\label{remark:hopf-vs-coideal-D}
Despite the similar notation, 
$\quantumg(\frakso_n)$ and $\coideal(\frakso_n)$ 
are different algebras. 
In fact, the standard embedding $\U(\frakso_n)\hookrightarrow\U(\frakgl_n)$
does not lift to the quantum level as an embedding of 
$\quantumg(\frakso_n)$ into $\quantumg(\frakgl_n)$.
In contrast, $\coideal(\frakso_n)$  is, 
by definition, a subalgebra of $\quantumg(\frakgl_n)$. 
Both of them are, however, quantizations of the 
$\C$-algebra $\U(\frakso_n)$,
cf.\ \cite[Section 4, especially Theorem 4.8]{Le}.
\end{remark}

The algebra $\coideal(\frakso_n)$ is not a 
Hopf subalgebra of $\quantumg(\frakgl_n)$ (in particular, it is not closed 
under the comultiplication). Indeed, using \eqref{eq:conventionsHopf}, we get
\begin{equation}
\label{eq:coideal-condition}
\Delta(\sfB_i) = \sfB_i \otimes 1 + \sfK_i^{-1} \otimes \sfB_i \in \quantumg(\frakgl_n) \otimes \coideal(\frakso_n).
\end{equation}
However, \eqref{eq:coideal-condition} shows 
that $\coideal(\frakso_n)$ is a so-called
\emph{left coideal subalgebra}.

\subsubsection{The representation category of \texorpdfstring{$\coideal(\frakso_n)$}{corresponding to so(n)}}

We denote
the category of finite-dimensional representations 
of $\coideal(\frakso_n)$ by 
$\RepcoD{n}$.
Via restriction, we see that the objects and morphisms from 
$\Repq{n}$ are also in $\RepcoD{n}$. 
In particular, 
the $\quantumg(\frakgl_n)$-in\-tertwiners 
$\extmerge_{a,b}^{a{+}b}$, $\symmerge_{a,b}^{a{+}b}$, $\extsplit_{a{+}b}^{a,b}$ 
and $\symsplit_{a{+}b}^{a,b}$ 
are $\coideal(\frakso_n)$-equivariant as well.

Moreover, as recalled above, $\coideal(\frakso_n)$ 
is not 
closed under comultiplication. Hence,  
$\RepcoD{n}$
does not inherit the structure of a monoidal category from $\Repq{n}$. 
However, since 
$\coideal(\frakso_n)$ 
is a coideal subalgebra,
$\RepcoD{n}$
is a $\Repq{n}$-category 
in the sense of \fullref{def:action-of-mon-category}.
 
\subsubsection{Some intertwiners}

We define $\CQ$-linear maps
\begin{equation}\label{eq:int-typeD-1}
\begin{aligned}
    \cup &\colon \trivmod  \rightarrow \vecrep \otimes \vecrep, \quad
    1  \mapsto {\textstyle\sum_{i=1}^n}\; v_i \otimes v_i,\\   
     \cap &\colon \vecrep \otimes \vecrep  \rightarrow \trivmod,\quad
        v_i \otimes v_j   \mapsto 
        \begin{cases} 
          \qpar^{n+1-2i}, & \text{if } i = j,\\
          0,  & \text{else,} 
        \end{cases}
\end{aligned}
\end{equation}
\begin{equation}\label{eq:int-typeD-2}
\begin{aligned}
    \symdotdown &\colon \trivmod  \rightarrow \symmetric{2} \vecrep,  \quad
    1  \mapsto \tfrac{1}{\qbracket{2}}\left({\textstyle\sum_{i=1}^{n}}\; v_{i} v_{i}\right),\\
\symdotup &\colon \symmetric{2} \vecrep   \rightarrow \trivmod,    \quad
        v_i v_j  \mapsto 
        \begin{cases} 
          \qpar^{n+1-2i}, & i=j,\\
          0,  & \text{else.} 
        \end{cases}
\end{aligned}
\end{equation}

\begin{lemma}\label{lemma:int-typeD}
The $\CQ$-linear maps $\cup$, $\cap$, $\symdotdown$ and $\symdotup$ 
intertwine the $\coideal(\frakso_n)$-actions.
\end{lemma}

\begin{proof}
First we note that
\[
\cup=\symsplitt\circ\symdotdown
\quad\text{and}\quad
\cap=\symdotup\circ\symmerget.
\]
We already know that $\symsplitt$ and $\symmerget$ 
intertwine the action of $\quantumg(\frakgl_n)$. 
Thus, via restriction, they intertwine 
the action of $\coideal(\frakso_n)$ as well.
So it remains 
to show that $\symdotdown$ and $\symdotup$ 
intertwine the action of $\coideal(\frakso_n)$.
\begin{description}[leftmargin=0pt,itemsep=1ex]
\item[\textbf{The $\symdotdown$ case}]
One just has to show that 
$\sfB_j \big(\sum_{i=1}^n v_iv_i\big) = 0$ 
for all $j=1,\dots,n$, which follows via direct 
and straightforward computation.
\item[\textbf{The $\symdotup$ case}]
The computation boils down to checking that 
\begin{align*}
\symdotup (\sfB_i (v_i \otimes v_{i+1})) & = \symdotup (v_{i+1} \otimes v_{i+1} - \qpar^{-2} v_i \otimes v_i) = 0,
\end{align*}
and the claim follows.\qedhere
\end{description}
\end{proof}

\begin{remark}\label{remark:not-an-intertwiner-so}
Beware that $\notequiso$ is not $\coideal(\frakso_n)$-equivariant. 
To see this we note that
\[
\sfB_i(\notequiso(v_{j}))
\neq
\notequiso(\sfB_i(v_j)),
\]
which can be easily verified by observing that
\begin{align*}
\sfB_i({\textstyle\sum_{i=1}^n}\; v_i \otimes v_i \otimes v_j)
&=
\sfB_i({\textstyle\sum_{i=1}^n}\; v_i \otimes v_i) \otimes v_j
+
\sfK_i^{-1}({\textstyle\sum_{i=1}^n}\; v_i \otimes v_i)\otimes \sfB_i(v_j)
\\
&=
(\cdots + \qpar^{2} v_{i-1} \otimes v_{i-1} + \qpar^{-2} v_i \otimes v_i + \cdots)\otimes \sfB_i(v_j),
\end{align*}
which is not equal to $\notequiso(\sfB_i(v_j))$. Hereby we used 
that $\sfB_i({\textstyle\sum_{i=1}^n}\; v_i \otimes v_i)=0$ 
and \eqref{eq:coideal-condition}. However, using almost the same calculation, one can see 
that $\equiso$ is indeed $\coideal(\frakso_n)$-equivariant. 
This is the representation theoretical incarnation of the left-right partitioning of the 
$\typeBDC$-web calculus in \fullref{sec-Dwebs}, cf. \fullref{figure:webs}.
\end{remark}

\subsection{The coideal subalgebra \texorpdfstring{$\coideal(\fraksp_n)$}{corresponding to sp(n)}}\label{subsec-coideal-C}

Similarly to the orthogonal case, 
we define now the coideal subalgebra $\coideal(\fraksp_n)$, following \cite[Section 5]{KP}.

\begin{definition}\label{definition-typeC-coideal}
Let $\coideal(\fraksp_n)$ be the $\CQ$-subalgebra 
of $\quantumg(\frakgl_n)$ generated by
\begin{gather}\label{eq:gen-type-C}
\begin{aligned}
          \sfE_i,\sfF_i,\sfK_i^{\pm 1}, \qquad\qquad\qquad &\text{for }i=1,3,\dots,n-1,\\        
        \sfB_i= \sfF_i
        -\sfK_i^{-1}\mathrm{ad}(\sfE_{i-1}\sfE_{i+1})\centerdot\sfE_i,
         \;\; &\text{for }i=2,4,\dots,n,
\end{aligned}
\end{gather}
where $\mathrm{ad}(\sfX)\centerdot\sfY$ denotes the right adjoint action for $\sfX,\sfY\in\quantumg(\frakgl_n)$, 
cf.\ \cite[\S 4.18]{Ja1}, i.e.\ 
in Sweedler notation $\mathrm{ad}(\sfX)\centerdot\sfY=S(\sfX_{(2)})\sfY\sfX_{(1)}$.
\end{definition}

Explicitly, the adjoint action in \eqref{eq:gen-type-C} is
\begin{gather*}
\mathrm{ad}(\sfE_{i-1}\sfE_{i+1})\centerdot\sfE_i
=
\sfE_{i-1}\sfE_{i+1}\sfE_i
-\qpar^{-1}\sfE_{i-1}\sfE_i\sfE_{i+1}
-\qpar^{-1}\sfE_{i+1}\sfE_i\sfE_{i-1}
+\qpar^{-2}\sfE_i\sfE_{i-1}\sfE_{i+1}.
\end{gather*}

This expression is the one which we use below, e.g. in \fullref{lemma:int-typeC}.

\begin{remark}\label{remark:hopf-vs-coideal-C}
As before, $\coideal(\fraksp_{n})$ should not 
be confused with $\quantumg(\fraksp_{n})$, although 
they  both de-quantize to $\U(\fraksp_{n})$ (cf.\ \cite[Section 4]{Le}).
\end{remark}

One again checks that $\coideal(\fraksp_{n})$ is a left coideal subalgebra of 
$\quantumg(\frakgl_{n})$. However, we do not need the explicit formula for the comultiplication in this paper.

\subsubsection{The representation category of \texorpdfstring{$\coideal(\fraksp_n)$}{corresponding to sp(n)}}

We denote by 
$\RepcoC{n}$
the category of finite-di\-men\-sio\-nal representations 
of $\coideal(\fraksp_n)$.
Again, the category $\RepcoC{n}$
is a $\Repq{n}$-category since 
$\coideal(\fraksp_n)$ is a coideal subalgebra of $\quantumg(\frakgl_n)$, 
and, via restriction, the objects and morphisms from 
$\Repq{n}$ are also in $\RepcoC{n}$.

\subsubsection{Some more intertwiners}

We define $\CQ$-linear maps
\begin{equation}\label{eq:int-typeC-1}
\begin{aligned}
    \extdotdown &\colon \trivmod  \rightarrow \exterior{2} \vecrep,  \quad
    1  \mapsto {\textstyle\sum_{i=1}^{\nicefrac{n}{2}}}\; \qpar^{1-i} v_{2i-1} \wedge v_{2i},\\
\extdotup &\colon \exterior{2} \vecrep   \rightarrow \trivmod,    \quad
        v_i \wedge v_j  \mapsto 
        \begin{cases} 
          \qpar^{n-\nicefrac{1}{2}(3i+1)}, & \text{if }i\text{ is odd and } j = i+1,\\
          0,  & \text{else.} 
        \end{cases}
\end{aligned}
\end{equation}
\begin{gather}\label{eq:int-typeC-2}
\begin{aligned}
    \cup &\colon \trivmod  \rightarrow \vecrep \otimes \vecrep,  \quad
    1  \mapsto {\textstyle\sum_{i=1}^{\nicefrac{n}{2}}}\; \qpar^{1-i} 
    (\qpar v_{2i-1} \otimes v_{2i} - v_{2i} \otimes v_{2i-1}),\\
\cap &\colon \vecrep \otimes \vecrep   \rightarrow \trivmod,    \quad
        v_i \otimes v_j  \mapsto 
        \begin{cases} 
          \qpar^{n-\nicefrac{1}{2}(3i+1)}, & \text{if }i\text{ is odd and } j = i+1,\\
          -\qpar^{n-\nicefrac{1}{2}(3i)}, & \text{if }i\text{ is even and } j = i-1,\\
          0,  & \text{else.} 
        \end{cases}
\end{aligned}
\end{gather}

\begin{lemma}\label{lemma:int-typeC}
The $\CQ$-linear maps $\extdotdown$, $\extdotup$, $\cup$ and $\cap$
intertwine the $\coideal(\fraksp_n)$-actions.
\end{lemma}

\begin{proof}
As in the proof of \fullref{lemma:int-typeD} we have
\[
\cup=\extsplitt\circ\extdotdown
\quad\text{and}\quad
\cap=\extdotup\circ\extmerget.
\]
Hence, as before, we only need to check that $\extdotdown$ 
and $\extdotup$ are $\coideal(\fraksp_n)$-equivariant.
\begin{description}[leftmargin=0pt,itemsep=1ex]
\item[\textbf{The $\extdotdown$ case}]
We need to show 
for $i$ odd that $\sfK_i^{\pm 1}$ acts 
on $\extdotdown (1)$ as the identity and $\sfE_i$, $\sfF_i$ 
as zero, and for $i$ even that $\sfB_i(\extdotdown(1))=0$. 
The former is clear, while the latter computation essentially boils down to 
\begin{gather*}
\sfB_i(v_{i-1} \wedge v_i + \qpar^{-1} v_{i+1} \wedge v_{i+2}) = \sfF_i(v_{i-1} \wedge v_i) - \qpar^{-1} \sfK_i^{-1}\sfE_{i-1}\sfE_{i+1}\sfE_i(v_{i+1} \wedge v_{i+2})\\ =  v_{i-1} \wedge v_{i+1} - \qpar^{-1} \qpar v_{i-1} \wedge v_{i+1} =0,
\end{gather*}
since $\sfE_{i-1}(v_{i-1} \wedge v_i)=v_{i-1} \wedge v_{i-1}=0$
and $\sfE_{i+1}(v_{i+1} \wedge v_{i+2})=v_{i+1} \wedge v_{i+1}=0$.
\item[\textbf{The $\extdotup$ case}]
We have to show that 
\[
\extdotup(\sfX(v_i \wedge v_j)) = \sfX(\extdotup (v_i \wedge v_j)),\quad\text{for 
all } \sfX \text{ as in \eqref{eq:gen-type-C}}.
\]
This is clear for $\sfX=\sfK_l^{\pm 1}$ with $l$ odd, 
so let us assume that $\sfX$ is either an $\sfE$, an $\sfF$ or a $\sfB$. Of course, 
we can also assume that $i<j$. Still, we have a few cases to check, where 
we only need to verify $\extdotup(\sfX(v_i \wedge v_j)) =0$, since 
the other side is always zero:
\smallskip
\begin{enumerate}[label=$\blacktriangleright$]

\setlength\itemsep{.15cm}

\item If $j>i+2$, then it is easily shown that
$\extdotup(\sfX(v_i \wedge v_j))= 0$.
Indeed, the only thing to observe hereby is
\[
\sfE_i\sfE_{i+1}\sfE_{i+2}(v_i \wedge v_{i+3})=v_i \wedge v_{i}=0,
\]
which shows that $\extdotup(\sfB_{i+1}(v_i \wedge v_{i+3}))= 0$ for $i$ odd.
\item If $j=i+1$ and $i$
is odd, then
$\sfE_i(v_i \wedge v_{i+1})=\sfF_i(v_i \wedge v_{i+1})=0$. Moreover,
\[
\extdotup(\sfB_{i-1}(v_i \wedge v_{i+1})) = \extdotup(-\qpar v_{i-2}\wedge v_i)=0\;\text{ and }\;
\extdotup(\sfB_{i+1}(v_i \wedge v_{i+1})) = \extdotup(v_i \wedge v_{i+2}) = 0.
\]
\item If $j=i+1$ and $i$ is even, then
clearly $\extdotup(\sfX(v_i \wedge v_{i+1}))=0$ for
$\sfX$ being either of 
$\sfE_{i-1},\sfE_{i+1},\sfF_{i-1},\sfF_{i+1}$. Moreover, one also directly sees 
that
$\sfB_i(v_i \wedge v_{i+1})=0$.
\item If $j=i+2$ and $i$ is odd, then clearly 
$\sfE_l(v_i \wedge v_{i+2})=0$ for all $l$ odd. 
We also see directly that $\extdotup(\sfF_{i+2}(v_i \wedge v_{i+2}))=0$ 
and $\sfB_{i+1}(v_i \wedge v_{i+2}) =0$.
Moreover, noting that $i+1$ is even, we get 
\[
\extdotup(\sfF_{i}(v_i \wedge v_{i+2}))=\extdotup(v_{i+1} \wedge v_{i+2})=0.
\]
\item Finally, if $j=i+2$ and $i$ is even, then 
$\sfF_l(v_i \wedge v_{i+2})=0$ for all $l$ odd. 
We also directly see that $\extdotup(\sfE_{i-1}(v_i \wedge v_{i+2}))=0$.
Further, because $i$ is even, we have
\[
\extdotup(\sfE_{i+1}(v_i \wedge v_{i+2}))=\extdotup(v_{i} \wedge v_{i+1})=0.
\] 
Moreover, noting that $i-1$ and $i+1$ are odd, we get
\begin{gather*}
\extdotup(\sfB_i(v_i \wedge v_{i+2}))=\extdotup(v_{i+1} \wedge v_{i+2} -\qpar^{-3} v_{i-1} \wedge v_i)\\ = \qpar^{n-\nicefrac{3i+4}{2}}-\qpar^{-3}\qpar^{n-\nicefrac{3i-2}{2}} = 0,
\end{gather*}
and $\extdotup(\sfB_{i+2}(v_i \wedge v_{i+2}))=0$ follows again 
because $\extdotup(v_i \wedge v_{i+1})=0$.\qedhere
\end{enumerate}
\end{description}
\end{proof}

\begin{remark}\label{remark:not-an-intertwiner-sp}
Similarly as in \fullref{remark:not-an-intertwiner-so} 
one can show that $\notequisp$ is not $\coideal(\fraksp_n)$-equivariant, 
but $\equisp$ is. Again, this is related to the left-right partitioning 
of the $\typeCBD$-web calculus in \fullref{sec-Cwebs}, cf. \fullref{figure:webs}.
\end{remark}

\subsection{An integral form}\label{subsec:int-form}

For the purpose of later specialization, we need a 
version of Lusztig's integral form for
$\coideal(\frakso_n)$ and $\coideal(\fraksp_n)$.
To this end, we 
let $\Aalg=\C[\qpar,\qpar^{-1},\tfrac{1}{[n]}]$. 
We denote by $\quantumgA(\frakg)$ 
the \emph{$\Aalg$-form of} $\quantumg(\frakg)$, which 
is the $\Aalg$-subalgebra generated by the 
$\sfE_i$'s, $\sfF_i$'s and $\quantumq^{h}$'s. 
Note that we clearly have
$\quantumgA(\frakg) \otimes_{\Aalg} \CQ = \quantumg(\frakg)$.

\begin{definition}\label{definition:int-form}
We let $\coidealA(\frakso_{n}) \subset \quantumgA(\frakgl_{n})$ be 
the \emph{$\Aalg$-form of} $\coideal(\frakso_n)$, which 
is defined to be the $\Aalg$-subalgebra generated by 
the $\sfB_i$'s from \eqref{eq:gen-type-D}. 
Similarly, we define the 
\emph{$\Aalg$-form of} $\coideal(\fraksp_n)$ 
using the $\sfB_i$'s from \eqref{eq:gen-type-C}.
\end{definition}

Again, we clearly have that
\begin{gather*}
\coidealA(\frakso_{n}) \otimes_{\Aalg} \CQ = \coideal(\frakso_{n})
\quad\text{and}\quad
\coidealA(\fraksp_{n}) \otimes_{\Aalg} \CQ = \coideal(\fraksp_{n}).
\end{gather*}

\section{Connecting webs and representation categories}\label{sec-main}
We are now going to define the functors from \fullref{figure:commute}.

\subsection{Actions on representations in types \texorpdfstring{$\typeB\typeC\typeD$}{BCD}}
\label{subsec:acti-repr-type}

We will now define actions of our diagrammatic web categories
on representations of $\coideal(\frakso_n)$ and $\coideal(\fraksp_n)$.

\subsubsection{The presentation functors for \texorpdfstring{$\quantumg(\frakgl_n)$}{Uq(gl(n))}}

First, we recall that in type $\typeA$
we can define functors
$\diaA\colon\WebA \rightarrow 
\Repq{n}$ and $\symdiaA\colon\WebA \rightarrow 
\Repq{n}$
(sending the object $a$ to 
$\exterior{a}\vecrep$ and  $\symmetric{a}\vecrep$, respectively)
using the $\quantumg(\frakgl_n)$-intertwiners 
$\smash{\extmerge_{a,b}^{a{+}b}}$, $\smash{\extsplit_{a{+}b}^{a,b}}$
and $\smash{\symmerge_{a,b}^{a{+}b}}$, $\smash{\symsplit_{a{+}b}^{a,b}}$
from \fullref{sec-reps}.
By \fullref{example:merge-split}, 
we get
\begin{gather}\label{eq:dia-typeA}
\begin{aligned}
\diaA
\left(
      \begin{tikzpicture}[anchorbase,scale=.25,tinynodes]
      \draw[uncolored] (0,0) node[below] {$1$} .. controls ++(0,1) and ++(-0.5,-0.5) .. ++(1,1.5) .. controls ++(0.5,-0.5) and ++(0,1) .. ++(1,-1.5) node[below] {$1$}; 
      \draw[ccolored] (1,1.5) to (1,2.5);
      \draw[uncolored] (0,4) node[above] {$1$} .. controls ++(0,-1) and ++(-0.5,0.5) .. ++(1,-1.5) .. controls ++(0.5,0.5) and ++(0,-1) .. ++(1,1.5) node[above] {$1$}; 
      \end{tikzpicture}
\right)
&=
\extmergesplit \colon v_i \otimes v_j  \mapsto 
  \begin{cases}
    \qpar v_i \otimes v_j - v_j \otimes v_i, & \text{if } i<j,\\
    \qpar^{-1} v_i \otimes v_j - v_j \otimes v_i, & \text{if } i>j,\\
    0, & \text{if } i=j,
  \end{cases}\\
\symdiaA
\left(
      \begin{tikzpicture}[anchorbase,scale=.25,tinynodes]
      \draw[uncolored] (0,0) node[below] {$1$} .. controls ++(0,1) and ++(-0.5,-0.5) .. ++(1,1.5) .. controls ++(0.5,-0.5) and ++(0,1) .. ++(1,-1.5) node[below] {$1$}; 
      \draw[ccolored] (1,1.5) to (1,2.5);
      \draw[uncolored] (0,4) node[above] {$1$} .. controls ++(0,-1) and ++(-0.5,0.5) .. ++(1,-1.5) .. controls ++(0.5,0.5) and ++(0,-1) .. ++(1,1.5) node[above] {$1$}; 
      \end{tikzpicture}
\right)
&=
\symmergesplit \colon v_i \otimes v_j  \mapsto 
  \begin{cases}
    \qpar^{-1} v_i \otimes v_j + v_j \otimes v_i, & \text{if } i<j,\\
    \qpar v_i \otimes v_j + v_j \otimes v_i, & \text{if } i>j,\\
    \qbracket{2} v_i \otimes v_i, & \text{if } i=j.
  \end{cases}
\end{aligned}
\end{gather} 
We will use \eqref{eq:dia-typeA} frequently below.

\begin{remark}\label{remark:CKM-functor}
Note that $\diaA$ is 
the functor from \cite[\S 3.2]{CKM}, 
while $\symdiaA$ is its cousin 
as in \cite[Definition 2.18]{RT} 
or \cite[Definition 3.17]{TVW}.
\end{remark}

One can check that both $\diaA$ and $\symdiaA$
are functors of braided monoidal categories 
(see e.g.\ \cite[Theorem 3.20]{TVW}) -- a fact that 
we use silently below.

\subsubsection{The presentation functors for \texorpdfstring{$\coideal(\frakso_n)$}{corresponding to so(n)}}

We now specialize $\zpar=\qpar^{n}\in\CQ$ 
in the exterior and $\zpar=-\qpar^{-n}\in\CQ$ in the symmetric case. 
(Note that in both cases $\qbracket{\zpar;a}$ specializes to $\qbracket{n+a}$ and  
$\qbracketC{\zpar;a}$ specializes to $\qbracketC{n+a}$.)

We define 
$\diaD\colon \WebD \rightarrow 
\RepcoD{n}$ on objects by  $a \mapsto \exterior{a}\vecrep$
and on the generating morphisms by the assignment
\begin{gather}\label{eq:dia-typeD}
    \begin{tikzpicture}[anchorbase,scale=.25, tinynodes]
	\draw[uncolored] (-1,3) node[above] {$1$} to [out=270, in=180] (0,2) to [out=0, in=270] (1,3) node[above] {$1$};
	\node at (0,0) {$\phantom{.}$};
    \end{tikzpicture} 
\mapsto 
\big( \cup \colon \trivmod \rightarrow \vecrep \otimes \vecrep)
\quad\text{and}\quad
    \begin{tikzpicture}[anchorbase,scale=.25, tinynodes]
	\draw[uncolored] (-1,-3) node[below] {$1$} to [out=90, in=180] (0,-2) to [out=0, in=90] (1,-3) node[below] {$1$};
	\node at (0,0) {$\phantom{.}$};
    \end{tikzpicture}
\mapsto 
\big( \cap \colon \vecrep \otimes \vecrep \rightarrow \trivmod),
\end{gather}
and to be $\diaA$ on the $\typeA$-web generators \eqref{eq:Aweb-gens}. 
Similarly, we define its symmetric counterpart $\symdiaD\colon \WebC \rightarrow 
\RepcoD{n}$ on objects by $a \mapsto \symmetric{a}\vecrep$
and on the generating morphisms by the assignment
\begin{gather}\label{eq:dia-typeD-sym}
    \begin{tikzpicture}[anchorbase,scale=.25, tinynodes]
	\draw[ccolored] (0,0) to (0,1.5) node[above] {$2$};
	\node[dotbullet] at (0,0) {};
	\node at (0,1.5) {$\phantom{1}$};
   \end{tikzpicture} 
\mapsto 
\big( \symdotdown \colon \trivmod \rightarrow \symmetric{2}\vecrep)
\quad\text{and}\quad
    \begin{tikzpicture}[anchorbase,scale=.25, tinynodes]
	\draw[ccolored] (0,0) to (0,-1.5) node[below] {$2$};
	\node[dotbullet] at (0,0) {};
	\node at (0,1.5) {$\phantom{1}$};
   \end{tikzpicture}
\mapsto 
\big( \symdotup \colon \symmetric{2}\vecrep \rightarrow \trivmod),
\end{gather}
and to be $\symdiaA$ on the $\typeA$-web generators \eqref{eq:Aweb-gens}. The 
$\coideal(\frakso_n)$-intertwiners 
in \eqref{eq:dia-typeD} and \eqref{eq:dia-typeD-sym} are defined
in \eqref{eq:int-typeD-1} and \eqref{eq:int-typeD-2}.

In order to prove that $\diaD$ and $\symdiaD$ are well-defined, 
we need to show that the defining relations of 
$\WebD$ are satisfied in the image.
For $\diaD$, we do this 
in detail in the following lemmas,
where
we denote by 
$\id_{a}=\id_{\exteriors{a}\!\vecrep}$ 
the identity morphisms (we write $\id=\id_{1}$ 
for short) and all 
indexes are from $\{1,\dots,n\}$.
Further, we abbreviate 
$v_{j_1\dotsm j_\ell} = v_{j_1} \otimes \dotsb \otimes v_{j_l}$.

\begin{lemma}[Circle removal]\label{lemma:circle}
We have $\cap \circ \cup = [n]\,\id_{0}$.
\end{lemma}

\begin{proof}
By definition, $\cap \circ \cup (1) = \cap \big(\sum_{i=1}^n v_{ii}\big) = 
\sum_{i=1}^n \qpar^{n+1-2i} = [n] $.
\end{proof}

\begin{lemma}[Bubble removal]\label{lemma:bubble}
We have $(\id \otimes \cap) (\extmergesplit \otimes \id) (\id \otimes \cup) = [n-1]\,\id$.
\end{lemma}

\begin{proof}
We compute
\begin{align*}
&(\id \otimes \cap) (\extmergesplit \otimes \id) (\id \otimes \cup) (v_x)\\
&\quad  =   (\id \otimes \cap) (\extmergesplit \otimes \id) 
({\textstyle\sum_{i=1}^n} v_{xii})  
 =   (\id \otimes \cap) 
({\textstyle\sum_{i<x}} \qpar^{-1} v_{xii} + {\textstyle\sum_{i>x}} \qpar v_{xii}) \\
&\quad  =    {\textstyle\sum_{i<x}} \qpar^{-1} \qpar^{n+1-2i} v_x + {\textstyle\sum_{i>x}} \qpar \qpar^{n+1-2i} v_x 
 = {\textstyle\sum_{i=1}^{n-1}} \qpar^{n-2i} v_x = [n-1] v_x,
\end{align*}
which shows the statement.
\end{proof}

\begin{lemma}[Lasso move]\label{lemma:lasso}
We have
\[
(\id \otimes \id \otimes \cap)(\id \otimes \overcrossing \otimes \id)(\extmergesplit \otimes \extmergesplit)(\id \otimes \undercrossing \otimes \id)(\id \otimes \id \otimes \cup) - \cupcap = [n-2]\,\id \otimes \id.
\]
\end{lemma}

\begin{proof}
We compute 
\begin{align*}
&    (\id \otimes \id \otimes \cap)(\id \otimes \overcrossing \otimes \id)(\extmergesplit \otimes \extmergesplit)(\id \otimes \undercrossing \otimes \id)(\id \otimes \id \otimes \cup) ( v_{xy}) \\
& \quad \textstyle = (\id \otimes \id \otimes \cap)(\id \otimes \overcrossing \otimes \id)(\extmergesplit \otimes \extmergesplit)(\id \otimes \undercrossing \otimes \id)(\sum_{i=1}^n v_{xyii})\\
& \quad \textstyle = (\id \otimes \id \otimes \cap)(\id \otimes \overcrossing \otimes \id)(\extmergesplit \otimes \extmergesplit)(- \sum_{i \neq b} v_{xiyi}).
\end{align*}
Now, if $x<y$, then we get
\begin{align*}
& \quad \textstyle = (\id \otimes \id \otimes \cap)(\id \otimes \overcrossing \otimes \id)(- \sum_{i<x} \qpar^{-2} v_{xiyi} -\qpar^{-1} v_{xiiy} -\qpar^{-1} v_{ixyi} + v_{ixiy})\\
& \quad \quad\textstyle + (\id \otimes \id \otimes \cap)(\id \otimes \overcrossing \otimes \id)(-\sum_{x<i<y} v_{xiyi} -\qpar^{-1} v_{xiiy} -\qpar v_{ixyi} +  v_{ixiy})\\
& \quad \quad \textstyle + (\id \otimes \id \otimes \cap)(\id \otimes \overcrossing \otimes \id)(-\sum_{i>y} \qpar^{2} v_{xiyi} -\qpar v_{xiiy} - v_{ixyi} + v_{ixiy}) \\
& \quad \textstyle = (\id \otimes \id \otimes \cap)(\sum_{i<y} \qpar^{-2} v_{xyii}  + \sum_{x<i<y} v_{xyii}+\sum_{i>y} \qpar^{2} v_{xyii}) \\
& \quad \textstyle = \sum_{i=1}^{n-2} \qpar^{n-2i-1} v_{xy}=[n-2] v_{xy}.
\end{align*}
Similarly, if $x>y$, then we get
\begin{align*}
& \quad \textstyle = (\id \otimes \id \otimes \cap)(\id \otimes \overcrossing \otimes \id)(-\sum_{i<y} \qpar^{-2} v_{xiyi} -\qpar^{-1} v_{xiiy} -\qpar^{-1} v_{ixyi} + v_{ixiy})\\
& \quad \quad\textstyle + (\id \otimes \id \otimes \cap)(\id \otimes \overcrossing \otimes \id)(-\sum_{y<i<x} v_{xiyi} -\qpar v_{xiiy} -\qpar^{-1} v_{ixyi} +  v_{ixiy})\\
& \quad \quad \textstyle + (\id \otimes \id \otimes \cap)(\id \otimes \overcrossing \otimes \id)(-\sum_{i>x} \qpar^{2} v_{xiyi} -\qpar v_{xiiy} - v_{ixyi} + v_{ixiy}) \\
& \quad \textstyle = (\id \otimes \id \otimes \cap)(\sum_{i<x} \qpar^{-2} v_{xyii}  + \sum_{y<i<x} v_{xyii}
+ \sum_{i>x} \qpar^{2} v_{xyii} ) \\
& \quad \textstyle = \sum_{i=1}^{n-2} \qpar^{n-2i-1} v_{xy}=[n-2] v_{xy}.
\end{align*}
So the statement is proved on $v_x \otimes v_y$ if $x \neq y$. 
Finally, if $x=y$, then we get
\begin{align*}
& \quad \textstyle = (\id \otimes \id \otimes \cap)(\id \otimes \overcrossing \otimes \id)(-\sum_{i<x} \qpar^{-2} v_{xixi} -\qpar^{-1} v_{xiix} -\qpar^{-1} v_{ixxi} + v_{ixix})\\
& \quad \quad \textstyle + (\id \otimes \id \otimes \cap)(\id \otimes \overcrossing \otimes \id)(-\sum_{i>x} \qpar^{2} v_{xixi} -\qpar v_{xiix} - v_{ixxi} + v_{ixix}) \\
& \quad \textstyle = (\id \otimes \id \otimes \cap)(\sum_{i<x} \qpar^{-2} v_{xxii} + v_{iixx} )  + (\id \otimes \id \otimes \cap)(\sum_{i>x} \qpar^{2} v_{xxii} + v_{iixx}) \\
& \quad \textstyle = \sum_{i<x}( \qpar^{n-2i-1} v_{xx} + \qpar^{n-2x+1} v_{ii} )+ \sum_{i>x}( \qpar^{n-2i+3} v_{xx} + \qpar^{n-2x + 1} v_{ii})\\
& \quad \textstyle = \sum_{i=1}^{n-2}( \qpar^{n-2i-1} v_{xx})  + \sum_{i=1}^{n} ( \qpar^{n-2x + 1} v_{xx}) =  [n-2] v_{xx}  + \sum_{i=1}^{n} ( \qpar^{n-2x + 1} v_{ii})\\
& \quad = ([n-2]\id + \cupcap)(v_{xx}),
\end{align*}
and we are done.
\end{proof}

\begin{lemma}[Lollipop relation]\label{lemma:lollipop}
We have $\extmergesplit \circ \cup = 0$ and $\cap \circ \extmergesplit = 0$.
\end{lemma}

\begin{proof} 
First, if $x<y$, then
$\textstyle (\cap \circ \extmergesplit) 
(v_{xy}) = \cap( \qpar v_{xy} - v_{yx}) = 0$ 
while, if $x > y$, then 
$\textstyle (\cap \circ \extmergesplit) (v_x \otimes v_y) = 
\cap( \qpar^{-1} v_{xy} - v_{yx}) = 0$.
Next,
$\textstyle   (\extmergesplit \circ \cup)(1) = 
\extmergesplit (\sum_{i=1}^n v_{ii}) = 0$.
\end{proof}

\begin{lemma}[Merge-split sliding relations]\label{lemma:merge-split-slide}
We have 
\begin{gather*}
(\cap \otimes  \cap)(\id \otimes \overcrossing \otimes \id)(\extmergesplit \otimes \id \otimes \id) = (\cap \otimes  \cap)(\id \otimes \overcrossing \otimes \id)( \id \otimes \id \otimes \extmergesplit),\\
(\extmergesplit \otimes \id \otimes \id) (\id \otimes \overcrossing \otimes \id) (\cup \otimes  \cup) = ( \id \otimes \id \otimes \extmergesplit)(\id \otimes \overcrossing \otimes \id)(\cup \otimes  \cup).
\end{gather*}
\end{lemma}

\begin{proof}
First, we compute
\begin{gather*}
     (\cap \otimes  \cap)(\id \otimes \overcrossing \otimes \id) (v_{wxyz}) 
     =
      \begin{cases}
        -\qpar^{2(n+1-2w)-1}, & \text{if } w=x=y=z,\\
        \qpar^{2(n+1)-w-y}(\qpar-\qpar^{-1}), & \text{if } w=x<y=z,\\
        -\qpar^{2(n+1)-w-x}, & \text{if } w=y \neq x=z,\\
        0, & \text{else.}           
      \end{cases}
\end{gather*}
Now, it is easy to see that both 
$(\cap \otimes  \cap)(\id \otimes \overcrossing 
\otimes \id)(\extmergesplit \otimes \id \otimes \id)(v_{wxyz})$ 
and $(\cap \otimes  \cap)(\id \otimes \overcrossing 
\otimes \id)( \id \otimes \id \otimes \extmergesplit)(v_{wxyz})$ 
can only be non-zero if $w=y$ and 
$x=z$, and that 
they are equal in this case. This shows the first equation.

For the second equation, we compute
\begin{gather}\label{eq:lemma-merge-split-slide}
\begin{aligned}
         &(\id \otimes \overcrossing \otimes \id)  (\cup \otimes \cup) (1) 
        =(\id \otimes \overcrossing \otimes \id)(\textstyle \sum_{i,j=1}^n v_{iijj}) \\
        &=- \textstyle \sum_{i\neq j}  v_{ijij} + (\qpar-\qpar^{-1}) \sum_{i< j}  v_{iijj} -\qpar^{-1} \sum_{i=j} v_{iiii}.
\end{aligned}
\end{gather}
Next, applying 
both $\extmergesplit \otimes \id \otimes \id$ or 
$\id \otimes \id \otimes \extmergesplit$ to \eqref{eq:lemma-merge-split-slide} 
yields
\begin{gather*}
        \textstyle \sum_{i<j} \big(v_{jiij} -\qpar v_{ijij}\big) + \sum_{i>j} \big(v_{jiij} - \qpar^{-1} v_{ijij}\big),
\end{gather*}
which proves the lemma.
\end{proof}

The proof that \eqref{eq:dia-typeD-sym} is well-defined 
works very similarly. 
It follows basically by 
the above, by comparison of the topological 
version of the relations in $\WebD$ and $\WebCpara$, 
and by comparison of \eqref{eq:int-typeD-1} 
and \eqref{eq:int-typeD-2}. We omit the details for brevity.
Hence, we get:

\begin{proposition}\label{proposition:dia-typeD-well-defined}
The two functors $\diaD$ and $\symdiaD$ are well-defined. 
Moreover, we have commuting diagrams
\[
  \begin{tikzpicture}[anchorbase]
  \matrix (m) [matrix of math nodes, row sep=2em, column
  sep=4em, text height=1.8ex, text depth=0.25ex] {
\WebA  & \Repq{n}\\
\WebD  & \RepcoD{n} \\};
  \path[->, myred] (m-1-1) edge node[above] {$\diaA$}  (m-1-2);
  \path[->, myred] (m-2-1) edge node[below] {$\diaD$}  (m-2-2);
  \draw[] ($(m-2-1.north) + (-.1,0.175)$)
  arc (250:-70:0.2cm);
  \draw[thick, ->] ($(m-2-1.north) + (.1,0.225)$) to ($(m-2-1.north) + (.05,0.175)$);
  \draw[] ($(m-2-2.north) + (-.1,0.175)$)
  arc (250:-70:0.2cm);
  \draw[thick, ->] ($(m-2-2.north) + (.1,0.225)$) to ($(m-2-2.north) + (.05,0.175)$);
  \end{tikzpicture}
\quad\text{and}\quad
  \begin{tikzpicture}[anchorbase]
  \matrix (m) [matrix of math nodes, row sep=2em, column
  sep=4em, text height=1.8ex, text depth=0.25ex] {
\WebA  & \Repq{n}\\
\WebCpara  & \RepcoD{n}. \\};
  \path[->, mygreen] (m-1-1) edge node[above] {$\symdiaA$}  (m-1-2);
  \path[->, mygreen] (m-2-1) edge node[below] {$\symdiaD$}  (m-2-2);
  \draw[] ($(m-2-1.north) + (-.1,0.175)$)
  arc (250:-70:0.2cm);
  \draw[thick, ->] ($(m-2-1.north) + (.1,0.225)$) to ($(m-2-1.north) + (.05,0.175)$);
  \draw[] ($(m-2-2.north) + (-.1,0.175)$)
  arc (250:-70:0.2cm);
  \draw[thick, ->] ($(m-2-2.north) + (.1,0.225)$) to ($(m-2-2.north) + (.05,0.175)$);
  \end{tikzpicture}
\hspace*{1.25cm}
\raisebox{-.7cm}{\qedmake}
\hspace*{-1.25cm}
\]
\end{proposition}

\subsubsection{The presentation functors for \texorpdfstring{$\coideal(\fraksp_n)$}{corresponding to sp(n)}}

Again, we specialize to $\zpar=\qpar^{n}\in\CQ$ 
in the exterior and to $\zpar=-\qpar^{-n}\in\CQ$ in the symmetric case.

We define 
$\diaC\colon \WebC \rightarrow 
\RepcoC{n}$ on generators by the assignment
\begin{gather}\label{eq:dia-typeC}
    \begin{tikzpicture}[anchorbase,scale=.25, tinynodes]
	\draw[ccolored] (0,0) to (0,-1.5) node[below] {$2$};
	\node[dotbullet] at (0,0) {};
	\node at (0,1.5) {$\phantom{1}$};
    \end{tikzpicture} 
\mapsto \big( \extdotup \colon \exterior{2}\vecrep  \rightarrow \trivmod)
\quad \text{and} \quad
    \begin{tikzpicture}[anchorbase,scale=.25, tinynodes]
	\draw[ccolored] (0,0) to (0,1.5) node[above] {$2$};
	\node[dotbullet] at (0,0) {};
	\node at (0,1.5) {$\phantom{1}$};
    \end{tikzpicture} 
\mapsto \big( \extdotdown \colon \trivmod \rightarrow \exterior{2} \vecrep),
\end{gather}
and, as before, to be $\diaA$ on $\typeA$-web generators. 
Analogously, we define its symmetric counterpart
$\symdiaC\colon \WebDpara \rightarrow 
\RepcoC{n}$ on generators via
\begin{gather}\label{eq:dia-typeC-sym}
     \begin{tikzpicture}[anchorbase,scale=.25, tinynodes]
	\draw[uncolored] (-1,3) node[above] {$1$} to [out=270, in=180] (0,2) to [out=0, in=270] (1,3) node[above] {$1$};
	\node at (0,0) {$\phantom{.}$};
    \end{tikzpicture}
\mapsto \big( \cup \colon \vecrep \otimes \vecrep  \rightarrow \trivmod)
\quad \text{and} \quad
    \begin{tikzpicture}[anchorbase,scale=.25, tinynodes]
	\draw[uncolored] (-1,-3) node[below] {$1$} to [out=90, in=180] (0,-2) to [out=0, in=90] (1,-3) node[below] {$1$};
	\node at (0,0) {$\phantom{.}$};
    \end{tikzpicture}
\mapsto \big( \cap \colon \trivmod \rightarrow \vecrep \otimes \vecrep),
\end{gather}
and, as before, to be $\symdiaA$ on $\typeA$-web generators.

Again, in order to prove that \eqref{eq:dia-typeC} is well-defined, 
we need to show that the defining relations of $\WebC$ 
are satisfied in the image. 
This boils down to prove the following 
lemmas, which can be verified, similarly as in type $\typeB\typeD$,
via involved and lengthy computations. 
In order to keep 
the number of (boring) computations in this paper 
in reasonable boundaries, we omit their proofs. 

\begin{lemma}[Barbell removal]\label{lemma:barbell-C}
We have $\extdotup \circ \extdotdown = [\tfrac{n}{2}]_2\id_0$.\qedmake
\end{lemma}

\begin{lemma}[Thin K removal]\label{lemma:thin-k}
We have
$(\id\otimes \extdotup)\circ\extmergesplit\circ(\id \otimes \extdotdown) 
= [\tfrac{n}{2}-1]_2\id$.\qedmake
\end{lemma}

\begin{lemma}[Thick K opening]\label{lemma:thick-k}
We have
$(\id_{2} \otimes \extdotup)\circ\extmergesplit\circ(\id_{2} \otimes \extdotdown) 
= \extdotdowndotup + [\tfrac{n}{2}-2]_2\id_{2}$.\qedmake
\end{lemma}

\begin{lemma}[Merge-split sliding relations]\label{lemma:merge-split-slide-C}
We have 
\begin{gather*}
(\cap \otimes  \cap)(\id \otimes \extmergesplit \otimes \id)(\extmergesplit \otimes \id \otimes \id) = (\cap \otimes  \cap)(\id \otimes \extmergesplit \otimes \id)( \id \otimes \id \otimes \extmergesplit),\\
(\extmergesplit \otimes \id \otimes \id) (\id \otimes \extmergesplit \otimes \id) (\cup \otimes  \cup) = ( \id \otimes \id \otimes \extmergesplit)(\id \otimes \extmergesplit \otimes \id)(\cup \otimes  \cup).
\hspace*{1.35cm}
\qedmake
\hspace*{-1.35cm}
\end{gather*}
\end{lemma}

Again, the proof that \eqref{eq:dia-typeC-sym} is well-defined 
goes similarly, and we immediately obtain:

\begin{proposition}\label{proposition:dia-typeC-well-defined}
The two functors $\diaC$ and $\symdiaC$ are  well-defined. 
Moreover, we have commuting diagrams
\[
  \begin{tikzpicture}[anchorbase]
  \matrix (m) [matrix of math nodes, row sep=2em, column
  sep=4em, text height=1.8ex, text depth=0.25ex] {
\WebA  & \Repq{n}\\
\WebC  & \RepcoC{n}. \\};
  \path[->, myred] (m-1-1) edge node[above] {$\diaA$}  (m-1-2);
  \path[->, myred] (m-2-1) edge node[below] {$\diaC$}  (m-2-2);
  \draw[] ($(m-2-1.north) + (-.1,0.175)$)
  arc (250:-70:0.2cm);
  \draw[thick, ->] ($(m-2-1.north) + (.1,0.225)$) to ($(m-2-1.north) + (.05,0.175)$);
  \draw[] ($(m-2-2.north) + (-.1,0.175)$)
  arc (250:-70:0.2cm);
  \draw[thick, ->] ($(m-2-2.north) + (.1,0.225)$) to ($(m-2-2.north) + (.05,0.175)$);
  \end{tikzpicture}
\quad\text{and}\quad
  \begin{tikzpicture}[anchorbase]
  \matrix (m) [matrix of math nodes, row sep=2em, column
  sep=4em, text height=1.8ex, text depth=0.25ex] {
\WebA  & \Repq{n}\\
\WebDpara  & \RepcoC{n}. \\};
  \path[->, mygreen] (m-1-1) edge node[above] {$\symdiaA$}  (m-1-2);
  \path[->, mygreen] (m-2-1) edge node[below] {$\symdiaC$}  (m-2-2);
  \draw[] ($(m-2-1.north) + (-.1,0.175)$)
  arc (250:-70:0.2cm);
  \draw[thick, ->] ($(m-2-1.north) + (.1,0.225)$) to ($(m-2-1.north) + (.05,0.175)$);
  \draw[] ($(m-2-2.north) + (-.1,0.175)$)
  arc (250:-70:0.2cm);
  \draw[thick, ->] ($(m-2-2.north) + (.1,0.225)$) to ($(m-2-2.north) + (.05,0.175)$);
  \end{tikzpicture}
\hspace*{.9cm}
\raisebox{-.7cm}{\qedmake}
\hspace*{-.9cm}
\]
\end{proposition}

\subsection{The ladder functor in types \texorpdfstring{$\typeB\typeC\typeD$}{BCD}}
\label{subsec:action-on-web}

We now define the \emph{ladder functors} $\ladD$ and $\ladC$, which 
relate our web categories to the quantum groups 
$\quantumg(\frakso_{2k})$ and $\quantumg(\fraksp_{2k})$. 
We stress that the definition of the ladder functors 
do not depend on whether we are in the exterior or the symmetric case.

\subsubsection{The ladder functor for \texorpdfstring{$\typeBDC$}{cup}-webs}

Let $\overline \lambda = \lambda + \tfrac{n}{2}$.
We define
the \emph{ladder functor}
$\ladD \colon \Udot(\frakso_{2k}) \rightarrow \WebD$ via
\begin{gather}\label{eq:lad-functor-D}
\begin{aligned}
\smallone_{\lambda} & \textstyle \longmapsto (\overline\lambda_1 =\lambda_1 + \frac{n}{2} , \dots ,\overline\lambda_k=\lambda_k + \frac{n}{2}),
\\
\dotE_i \smallone_{\lambda} &  \longmapsto 
    \begin{tikzpicture}[anchorbase,xscale=.45,yscale=.25,tinynodes]
	\draw[colored] (0,0) node[below] {$\overline\lambda_1$} -- ++(0,4) node[above] {$\overline\lambda_1$}; 
    \node[below=0.8ex] at (1.5,.5) {$\dots$};
    \node[below=0.8ex] at (6.5,.5) {$\dots$};
    \node at (1.5,2) {$\dots$};
    \node at (6.5,2) {$\dots$};
    \node[above=0.8ex] at (1.5,3.95) {$\dots$};
    \node[above=0.8ex] at (6.5,3.95) {$\dots$};
	\draw[colored] (3,0) node[below] {$\overline\lambda_i$} --  ++(0,4) node[above] {$\overline\lambda_i{+}1$}; 
	\draw[colored] (5,0) node[below] {$\overline\lambda_{i{+}1}$} --  ++(0,4) node[above] {$\overline\lambda_{i{+}1}{-}1$};  
	\draw[colored] (8,0) node[below] {$\overline\lambda_k$} -- ++(0,4) node[above] {$\overline\lambda_k$}; 
    \draw[uncolored] (5,1.5) -- (3,2.5);
    \end{tikzpicture} 
,\quad \text{for all } i=1,\dots,k-1,
\\
\dotF_i \smallone_{\lambda} &  \longmapsto 
    \begin{tikzpicture}[anchorbase,xscale=.45,yscale=.25,tinynodes]
	\draw[colored] (0,0) node[below] {$\overline\lambda_1$} -- ++(0,4) node[above] {$\overline\lambda_1$}; 
    \node[below=0.8ex] at (1.5,.5) {$\dots$};
    \node[below=0.8ex] at (6.5,.5) {$\dots$};
    \node at (1.5,2) {$\dots$};
    \node at (6.5,2) {$\dots$};
    \node[above=0.8ex] at (1.5,3.95) {$\dots$};
    \node[above=0.8ex] at (6.5,3.95) {$\dots$};
	\draw[colored] (3,0) node[below] {$\overline\lambda_i$} --  ++(0,4) node[above] {$\overline\lambda_i{-}1$}; 
	\draw[colored] (5,0) node[below] {$\overline\lambda_{i{+}1}$} --  ++(0,4) node[above] {$\overline\lambda_{i{+}1}{+}1$};  
	\draw[colored] (8,0) node[below] {$\overline\lambda_k$} -- ++(0,4) node[above] {$\overline\lambda_k$}; 
    \draw[uncolored] (5,2.5) -- (3,1.5);
    \end{tikzpicture} 
,\quad \text{for all } i=1,\dots,k-1,
\\
\dotE_k \smallone_{\lambda} &  \longmapsto 
    \begin{tikzpicture}[anchorbase,xscale=.45,yscale=.25,tinynodes]
	\draw[colored] (0,0) node[below] {$\overline\lambda_1$} -- ++(0,4) node[above] {$\overline\lambda_1$}; 
    \node[below=0.8ex] at (1.5,.5) {$\dots$};
    \node at (1.5,2) {$\dots$};
    \node[above=0.8ex] at (1.5,3.95) {$\dots$};
	\draw[colored] (3,0) node[below] {$\overline\lambda_{k{-}2}$} --  ++(0,4) node[above] {$\overline\lambda_{k{-}2}$}; 
	\draw[colored] (5,0) node[below] {$\overline\lambda_{k{-}1}$} --  ++(0,4) node[above] {$\overline\lambda_{k{-}1}{+}1$};  
	\draw[colored] (7,0) node[below] {$\overline\lambda_k$} -- ++(0,4) node[above] {$\overline\lambda_k{+}1$}; 
    \begin{scope}[xshift=5cm]
    \draw[uncolored] (0,3) -- ++(1,-1);      \draw[uncolored, cross line] (1,2) .. controls ++(1,-1) and ++ (0,-1.5) ..  ++(3,-1);
    \draw[uncolored] (4,1) .. controls ++(0,0.8) and ++(1,-1) .. ++(-2,2)  ;
    \end{scope}
    \end{tikzpicture}
\\
\dotF_k \smallone_{\lambda} &  \longmapsto 
    \begin{tikzpicture}[anchorbase,xscale=.45,yscale=.25,tinynodes]
	\draw[colored] (0,0) node[below] {$\overline\lambda_1$} -- ++(0,4) node[above] {$\overline\lambda_1$}; 
    \node[below=0.8ex] at (1.5,.5) {$\dots$};
    \node at (1.5,2) {$\dots$};
    \node[above=0.8ex] at (1.5,3.95) {$\dots$};
	\draw[colored] (3,0) node[below] {$\overline\lambda_{k{-}2}$} --  ++(0,4) node[above] {$\overline\lambda_{k{-}2}$}; 
	\draw[colored] (5,0) node[below] {$\overline\lambda_{k{-}1}$} --  ++(0,4) node[above] {$\overline\lambda_{k{-}1}{-}1$};  
	\draw[colored] (7,0) node[below] {$\overline\lambda_k$} -- ++(0,4) node[above] {$\overline\lambda_k{-}1$}; 
    \begin{scope}[xshift=5cm]
    \draw[uncolored] (0,1) -- ++(1,1);      \draw[uncolored, cross line] (1,2) .. controls ++(1,1) and ++ (0,1.5) ..  ++(3,1); 
    \draw[uncolored] (4,3) .. controls ++(0,-0.8) and ++(1,1) .. ++(-2,-2)  ;
    \end{scope}
    \end{tikzpicture}
\end{aligned}
\end{gather}
Here, we 
silently assume that $\lambda$, as 
an object of $\WebD$, is the 
zero object if $\lambda \notin \Z^k_{\geq 0}$.

\begin{lemma}\label{lemma:dia-functor-well-defined-D}
The ladder functor $\ladD$ is well-defined.
\end{lemma}

\begin{proof}
We need to check that the 
relations of $\Udot(\frakso_{2k})$ are 
satisfied in the image. 
\begin{description}[leftmargin=0pt,itemsep=1ex]
\item[\textbf{Assignment of the generators}]
Recall that
\[
\dotE_i\smallone_{\lambda}
\in\Hom_{\Udot(\frakso_{2k})}(\smallone_{\lambda},\smallone_{\lambda{+}\alpha_i})
\quad\text{and}\quad
\dotF_i\smallone_{\lambda}
\in\Hom_{\Udot(\frakso_{2k})}(\smallone_{\lambda},\smallone_{\lambda{-}\alpha_i}),
\]
where $\alpha_i$ 
are the simple roots. By our conventions 
for types $\typeA$ and $\typeD$ 
(cf.\ at the beginning of \fullref{subsec:conventionsqgroup}),
we see 
that \eqref{eq:lad-functor-D} 
lands in the correct morphisms spaces.
\item[\textbf{The $\Udot(\frakgl_{2k})$ relations}]
The relations involving 
only $\dotE_i$'s and $\dotF_i $'s with 
$i \neq k-1$ are clearly satisfied by the web calculus in 
type $\typeA$, i.e.\ by \cite[Proposition 5.2.1]{CKM}. 
\item[\textbf{The $\Udot(\frakso_{2k})$ relations}]
We just have to check case by case that the 
defining relations of $\Udot(\frakso_{2k})$ which involve 
$\dotE_k$'s and $\dotF_k$'s hold in the 
web calculus 
(for this purpose, recall the anti-involution $\omega$ 
from \fullref{remark:over-under-crossing}):
\smallskip
\begin{enumerate}[label=$\blacktriangleright$]

\setlength\itemsep{.15cm}

\item The commutator relation \eqref{eq:EF-rel} between $\dotE_k$ and $\dotF_k$ 
holds in $\WebD$ by \fullref{lemma:typeD-EF-FE}.
\item The 
images of $\dotF_{k-1}$ and $\dotE_k$ commute thanks to 
\fullref{lemma-typeD-EF}. Applying $\omega$ 
shows that the images of $\dotE_{k-1}$ and $\dotF_k$ commute as well.
\item The Serre relation \eqref{eq:serre-so-1}
holds in $\WebD$ by \fullref{lemma-typeD-EE-2}. 
The $\dotF$ version of it 
holds by applying $\omega$.
\item The Serre relation \eqref{eq:serre-so-2} 
holds in $\WebD$ by \fullref{lemma:typeD-serre-relation-2}. 
The versions involving $\dotF$'s 
hold by applying $\omega$.
\item The Serre relation \eqref{eq:serre-so-3} 
holds in $\WebD$ by \fullref{lemma:typeD-serre-relation}. 
Again, the versions involving $\dotF$'s 
hold by applying $\omega$.
\end{enumerate}
\end{description}
Note here that the quantum numbers work 
out thanks to the shift by $\tfrac{n}{2}$ in \eqref{eq:lad-functor-D}.
All other relations, e.g.\ far-commutativity, are clearly satisfied.
\end{proof}

\subsubsection{The ladder functor for \texorpdfstring{$\typeCBD$}{C,BD}-webs}

Using the same notation as above, 
we define the \emph{ladder functor} 
$\ladC \colon \Udot(\fraksp_{2k}) \rightarrow \WebC$ via
\begin{gather}\label{eq:lad-functor-C}
\begin{aligned}
\smallone_{\lambda} & \textstyle \longmapsto (\overline\lambda_1 =\lambda_1 + \frac{n}{2} , \dots ,\overline\lambda_k=\lambda_k + \frac{n}{2}),
\\
\dotE_i \smallone_{\lambda} &  \longmapsto 
    \begin{tikzpicture}[anchorbase,xscale=.45,yscale=.25,tinynodes]
	\draw[colored] (0,0) node[below] {$\overline\lambda_1$} -- ++(0,4) node[above] {$\overline\lambda_1$}; 
    \node[below=0.8ex] at (1.5,.5) {$\dots$};
    \node[below=0.8ex] at (6.5,.5) {$\dots$};
    \node at (1.5,2) {$\dots$};
    \node at (6.5,2) {$\dots$};
    \node[above=0.8ex] at (1.5,3.95) {$\dots$};
    \node[above=0.8ex] at (6.5,3.95) {$\dots$};
	\draw[colored] (3,0) node[below] {$\overline\lambda_i$} --  ++(0,4) node[above] {$\overline\lambda_i{+}1$}; 
	\draw[colored] (5,0) node[below] {$\overline\lambda_{i{+}1}$} --  ++(0,4) node[above] {$\overline\lambda_{i{+}1}{-}1$};  
	\draw[colored] (8,0) node[below] {$\overline\lambda_k$} -- ++(0,4) node[above] {$\overline\lambda_k$}; 
    \draw[uncolored] (5,1.5) -- (3,2.5);
    \end{tikzpicture} 
,\quad \text{for all } i=1,\dots,k-1,
\\
\dotF_i \smallone_{\lambda} &  \longmapsto 
    \begin{tikzpicture}[anchorbase,xscale=.45,yscale=.25,tinynodes]
	\draw[colored] (0,0) node[below] {$\overline\lambda_1$} -- ++(0,4) node[above] {$\overline\lambda_1$}; 
    \node[below=0.8ex] at (1.5,.5) {$\dots$};
    \node[below=0.8ex] at (6.5,.5) {$\dots$};
    \node at (1.5,2) {$\dots$};
    \node at (6.5,2) {$\dots$};
    \node[above=0.8ex] at (1.5,3.95) {$\dots$};
    \node[above=0.8ex] at (6.5,3.95) {$\dots$};
	\draw[colored] (3,0) node[below] {$\overline\lambda_i$} --  ++(0,4) node[above] {$\overline\lambda_i{-}1$}; 
	\draw[colored] (5,0) node[below] {$\overline\lambda_{i{+}1}$} --  ++(0,4) node[above] {$\overline\lambda_{i{+}1}{+}1$};  
	\draw[colored] (8,0) node[below] {$\overline\lambda_k$} -- ++(0,4) node[above] {$\overline\lambda_k$}; 
    \draw[uncolored] (5,2.5) -- (3,1.5);
    \end{tikzpicture} 
,\quad \text{for all } i=1,\dots,k-1,
\\
\dotE_k \smallone_{\lambda} &  \longmapsto 
    \begin{tikzpicture}[anchorbase,xscale=.45,yscale=.25,tinynodes]
	\draw[colored] (0,0) node[below] {$\overline\lambda_1$} -- ++(0,4) node[above] {$\overline\lambda_1$}; 
    \node[below=0.8ex] at (1.5,.5) {$\dots$};
    \node at (1.5,2) {$\dots$};
    \node[above=0.8ex] at (1.5,3.95) {$\dots$};
	\draw[colored] (3,0) node[below] {$\overline\lambda_{k{-}2}$} --  ++(0,4) node[above] {$\overline\lambda_{k{-}2}$}; 
	\draw[colored] (5,0) node[below] {$\overline\lambda_{k{-}1}$} --  ++(0,4) node[above] {$\overline\lambda_{k{-}1}$};  
	\draw[colored] (7,0) node[below] {$\overline\lambda_k$} -- ++(0,4) node[above] {$\overline\lambda_k{+}2$}; 
    \draw[ccolored] (7,2) to (8,1) to (8,0);
    \node[dotbullet] at (8,0) {};
    \end{tikzpicture}
\\
\dotF_k \smallone_{\lambda} &  \longmapsto 
    \begin{tikzpicture}[anchorbase,xscale=.45,yscale=.25,tinynodes]
	\draw[colored] (0,0) node[below] {$\overline\lambda_1$} -- ++(0,4) node[above] {$\overline\lambda_1$}; 
    \node[below=0.8ex] at (1.5,.5) {$\dots$};
    \node at (1.5,2) {$\dots$};
    \node[above=0.8ex] at (1.5,3.95) {$\dots$};
	\draw[colored] (3,0) node[below] {$\overline\lambda_{k{-}2}$} --  ++(0,4) node[above] {$\overline\lambda_{k{-}2}$}; 
	\draw[colored] (5,0) node[below] {$\overline\lambda_{k{-}1}$} --  ++(0,4) node[above] {$\overline\lambda_{k{-}1}$};  
	\draw[colored] (7,0) node[below] {$\overline\lambda_k$} -- ++(0,4) node[above] {$\overline\lambda_k{-}2$}; 
    \draw[ccolored] (7,2) to (8,3) to (8,4);
    \node[dotbullet] at (8,4) {};
    \end{tikzpicture}
\end{aligned}
\end{gather}
Again, we 
assume that $\lambda$, as 
an object of $\WebC$, is the 
zero object if $\lambda \notin \Z^k_{\geq 0}$.

\begin{lemma}\label{lemma:dia-functor-well-defined-C}
The ladder functor $\ladC$ is well-defined.
\end{lemma}

\begin{proof}
The proof is, mutatis mutandis, as the proof 
of \fullref{lemma:dia-functor-well-defined-D}.
In particular:
\smallskip
\begin{enumerate}[label=$\blacktriangleright$]

\setlength\itemsep{.15cm}

\item The $\dotE_k$-$\dotF_k$ commutator relation 
holds in $\WebC$ by \fullref{lemma:typeC-EF-FE}.
\item The images of $\dotF_{k-1}$ 
and $\dotE_k$ commute by \fullref{lemma-typeC-EF}.
That the images of $\dotE_{k-1}$ 
and $\dotF_k$ commute follows  by applying $\omega$.
\item The Serre relation \eqref{eq:serre-sp-1}
holds in $\WebC$ by \fullref{lemma-typeC-Serre}. 
As before, the versions involving $\dotF$'s 
follow then applying $\omega$.
\item The Serre relation \eqref{eq:serre-sp-2}
holds in $\WebC$ by \fullref{lemma-typeC-Serre-2}. 
As usual, the versions involving $\dotF$'s 
follow then applying $\omega$.\qedhere
\end{enumerate}
\end{proof}

\subsection{The Howe functors}

Note that  we never used that $\zpar$ was specialized 
to $\qpar^n$ in the definition of the ladder 
functors, and we actually get ladder 
functors $\Udot(\frakso_{2k}) \rightarrow \WebDwithpara{z}$ 
and $\Udot(\fraksp_{2k}) \rightarrow \WebCwithpara{z}$ for 
any $z \in \CQ$. In particular, we also get ladder 
functors $\Udot(\frakso_{2k}) \rightarrow \WebDpara$ and 
$\Udot(\fraksp_{2k}) \rightarrow \WebCpara$, which, by slight abuse 
of notation, we still denote by $\ladD$ and $\ladC$.

Composing the presentation and the 
ladder functors, we finally obtain the \emph{Howe functors}:
\begin{gather}\label{eq:Howe-functors}
\begin{aligned}
\howeD\colon&
\Udot(\frakso_{2k})
\xrightarrow{\ladD}
\WebD
\xrightarrow{\diaD}
\RepcoD{n},
\\
\symhoweC\colon&
\Udot(\frakso_{2k})
\xrightarrow{\ladD}
\WebDpara
\xrightarrow{\symdiaC}
\RepcoC{n},
\\
\symhoweD\colon&
\Udot(\fraksp_{2k})
\xrightarrow{\ladC}
\WebCpara
\xrightarrow{\symdiaD}
\RepcoD{n},
\\
\howeC\colon&
\Udot(\fraksp_{2k})
\xrightarrow{\ladC}
\WebC
\xrightarrow{\diaC}
\RepcoC{n}.
\end{aligned}
\end{gather}

\section{Main results}\label{sec:howe-duality}
We are finally ready to state and prove our main results.

\subsection{Quantizing Howe dualities in types \texorpdfstring{$\typeB\typeC\typeD$}{BCD}}\label{subsec:typeD-howe}

\subsubsection{A brief reminder on (quantum) highest weight theory}

The finite-dimensional representation theory
of $\quantumg(\frakg)$ at generic $\qpar$
is fairly well-understood. In particular, all such representations are
semisimple, and, if we restrict to so-called \emph{type $1$ representations}
(where $\quantumq^h$ acts by powers of $\qpar$, 
cf.\ \cite[Section 5.2]{Ja1}), then the simple
modules are in bijection with dominant integral 
weights $\lambda\in\weightlattice^+$.
We denote by $\simple{\frakg}{\lambda}$ the corresponding simple
$\quantumg(\frakg)$-module.

The situation for the coideal subalgebras, on the contrary, is more 
difficult and less understood.
For $\coideal(\frakso_n)$ and $\coideal(\fraksp_n)$, 
in particular, one cannot consider weights and weight 
representations, since there is 
no natural analog of a Cartan subalgebra  
(although see \cite{Let} for some progress in this direction).
Still, we will encounter some of their representations through Howe duality.

Before we can start, we need some more notation.
Let $\partition$ be the set of partitions (or Young diagrams). 
Given a partition $\lambda=(\lambda_1,\dots,\lambda_s)\in\partition$ (with $\lambda_s \neq 0$), we write $\ell(\lambda)=s$ for its length, and we denote by 
$\lambda^{\mathrm{T}}=(\lambda_1^{\mathrm{T}},\dots,\lambda_t^{\mathrm{T}})\in\partition$
its transpose.
For the rest, we keep the notation from \fullref{sec-main}.

We start with the sympletic case since 
it is easier to state (cf.\ \fullref{remark:sovso}).

\subsubsection{Skew quantum Howe duality for the pair \texorpdfstring{$(\coideal(\fraksp_n),\quantumg(\fraksp_{2k}))$}{spn-sp2k}}

\begin{theorem}\label{theorem:typeC-howe}
There are commuting actions
\begin{equation}\label{eq:Howe-action-sp}
\coideal(\fraksp_n) 
\acts  
\underbrace{\exterior{\bullet} \vecrep \otimes 
\dots \otimes \exterior{\bullet} \vecrep}_{k \text{ times}}
\actsreverse
\quantumg(\fraksp_{2k})
\end{equation}
generating each other's centralizer. 
Hence,
the $\coideal(\fraksp_n)$--$\quantumg(\fraksp_{2k})$-bimodule \eqref{eq:Howe-action-sp} 
is multi\-pli\-city-free. The $\quantumg(\fraksp_{2k})$-modules 
appearing in its decomposition are
\begin{equation}
  \label{eq:bimodule-decomp-C}
  \simple{\fraksp_{2k}}{
    {\textstyle\sum_{j=1}^k} (\lambda^{\mathrm{T}}_j - \tfrac{n}{2})\varepsilon_j}, \quad \text{for } \lambda\in\partition \text{ with }
\ell(\lambda^{\mathrm{T}})\leq k,\ell(\lambda)\leq \tfrac{n}{2}.
\end{equation}
\end{theorem}

\begin{proof}
We denote the space in the middle
of \eqref{eq:Howe-action-sp} 
by $\mathrm{M}_{\qpar}$. All $\lambda$'s appearing below will always satisfy 
the conditions in \eqref{eq:bimodule-decomp-C}.

By construction, $\mathrm{M}_{\qpar}$ is acted on by $\coideal(\fraksp_n)$ 
via restriction of the action by $\quantumg(\frakgl_n)$.
Using $\howeC$ from \eqref{eq:Howe-functors}, we see that 
there is a commuting action 
of $\quantumg(\fraksp_{2k})$. 
(In fact, we get an action of $\Udot(\fraksp_{2k})$ 
which then gives an action of $\quantumg(\fraksp_{2k})$ since 
$\mathrm{M}_{\qpar}$ is finite-dimensional, cf.\ \cite[Section 23.1.4]{Lus}.)

Next, we want to use the 
analogous result in the non-quantized 
setting (see \cite{Ho} and \cite[Corollary 5.33]{CW}, but 
beware that the roles of $k$ and $n$ 
are swapped in \cite{CW}). 
It states that there is an action of $\U(\fraksp_{2k})$ on 
$\mathrm{M}=\exteriornoq{\bullet}(\C^n \otimes \C^{k})$ 
commuting with the natural action 
of $\U(\fraksp_n)$
and that these two actions generate 
each others centralizer. Moreover, \cite[Corollary 5.33]{CW} 
gives the bimodule decomposition of $\mathrm{M}$,
similar to \eqref{eq:bimodule-decomp-C}.

Now, we can easily compare the action 
of $\quantumg(\fraksp_{2k})$ on $\mathrm{M}_{\qpar}$ and the 
action of $\U(\fraksp_{2k})$ on 
$\mathrm{M}$, and see that the weights 
and their multiplicities are the same. Hence, we can deduce that 
the decomposition of $\mathrm{M}_{\qpar}$ as a $\quantumg(\fraksp_{2k})$-module 
is the quantum analog of the one in \cite[Corollary 5.33]{CW}. 
It follows that the $\coideal(\fraksp_n)$--$\quantumg(\fraksp_{2k})$-bimodule 
$\mathrm{M}_{\qpar}$ decomposes as
\begin{equation*}
\mathrm{M}_{\qpar} \cong 
{\textstyle\bigoplus_{\lambda}}\,
\simpleco{\fraksp_n}{\lambda}
\otimes
\simple{\fraksp_{2k}}{\textstyle\sum_{j=1}^k 
(\lambda^{\mathrm{T}}_j - \tfrac{n}{2})\varepsilon_j}
,
\end{equation*}
with $\lambda$ as in \eqref{eq:bimodule-decomp-C} 
and where the
$\simpleco{\fraksp_n}{\lambda}$'s denote just some $\coideal(\fraksp_n)$-modules 
(which are indexed by the $\lambda$'s).
 
We want to show that all appearing 
$\simpleco{\fraksp_n}{\lambda}$ are irreducible, 
or, equivalently, that the action gives a surjection
\begin{equation}\label{eq:action-surjection}
\coideal(\fraksp_n) 
\twoheadrightarrow
\End_{\quantumg(\fraksp_{2k})}(\mathrm{M}_{\qpar}) 
\cong 
\End_{\CQ}\big(
{\textstyle\bigoplus_{\lambda}}\,
\simpleco{\fraksp_n}{\lambda}\big).
\end{equation}
To this end, consider the integral version $\repA{\mathrm{M}_{\qpar}}$ 
of the representation $\mathrm{M}_{\qpar}$, defined as the $\Aalg$-span 
of tensor products of wedges 
of the standard basis vectors $v_i$ inside $\mathrm{M}_{\qpar}$. 
Note that 
$\repA{\mathrm{M}_{\qpar}}$ is a free $\Aalg$-module, and this 
will be important for what follows. 

It can be easily checked 
that $\repA{\mathrm{M}_{\qpar}}$ is stable under the 
actions of 
$\coidealA(\fraksp_n)$ and $\quantumgA(\fraksp_{2k})$. 
Moreover, setting $\qpar=1$, we can identify 
$\repA{\mathrm{M}_{\qpar}} \otimes_\Aalg \Aalg/(\qpar-1)$ with 
$\mathrm{M}$, and it is then clear 
that the action of $\coidealA(\fraksp_n)\otimes_\Aalg \Aalg/(\qpar-1)$ 
matches the natural action of $\U(\fraksp_n)$, i.e.
\begin{equation*}
\begin{tikzpicture}[baseline=(current bounding box.center),yscale=0.6]
  \matrix (m) [matrix of math nodes, row sep=1em, column
  sep=4em, text height=1.5ex, text depth=0.25ex] {
\coidealA(\fraksp_n)\otimes_\Aalg \Aalg/(\qpar-1) & \End_{\Aalg/(q-1)}(
\repA{\mathrm{M}_{\qpar}} \otimes_\Aalg \Aalg/(\qpar-1))\\
     \U(\fraksp_n) &  \End_\C(\mathrm{M}). \\};
  \path[->] (m-1-1) edge  (m-1-2);
  \path[->] (m-2-1) edge  (m-2-2);
  \path[draw,double,double distance=0.4ex] (m-1-2) -- (m-2-2);
\end{tikzpicture}
\end{equation*}
(One could actually show that 
$\coidealA(\fraksp_n) \otimes_\Aalg \Aalg/(q-1)$ 
and $\U(\fraksp_n)$ are isomorphic, but 
since we do not need it, we avoid this 
additional complication.) In particular, the images of 
these two actions agree, and their dimensions are both equal to
\[
{\textstyle\sum_{\lambda}} \dim_{\C} \simplenoq{\fraksp_n}{\lambda}
=d=
{\textstyle\sum_{\lambda}} \dim_{\CQ} \simpleco{\fraksp_n}{\lambda}.
\] 
It follows 
that the dimension of the image for generic $\qpar$ cannot be 
strictly smaller, and in particular the dimension of the 
image of \eqref{eq:action-surjection} has to be greater or equal than $d$. 
Hence, the map in \eqref{eq:action-surjection} is surjective, and we are 
done.
\end{proof}

\subsubsection{Symmetric quantum Howe duality for the pair \texorpdfstring{$(\coideal(\fraksp_n),\Udot(\frakso_{2k}))$}{spn-so2k}}

\begin{theorem}\label{theorem:typeC-howe-sym}
There are commuting actions
\begin{equation}\label{eq:Howe-action-sp-sym}
\coideal(\fraksp_n) 
\acts  
\underbrace{\symmetric{\bullet} \vecrep \otimes 
\dots \otimes \symmetric{\bullet} \vecrep}_{k \text{ times}}
\actsreverse 
\Udot(\frakso_{2k})
\end{equation}
generating each other's centralizer. 
Hence, 
the $\coideal(\fraksp_n)$--$\Udot(\frakso_{2k})$-bimodule \eqref{eq:Howe-action-sp-sym}
is multi\-pli\-city-free. The $\Udot(\frakso_{2k})$-modules
appearing in its decomposition are
\begin{equation}
  \label{eq:bimodule-decomp-C-sym}
  \simple{\frakso_{2k}}{
    {\textstyle\sum_{j=1}^k} (\lambda_j + \tfrac{n}{2})\varepsilon_j}, \quad \text{for } \lambda\in\partition \text{ with }
  \ell(\lambda)\leq \min \{\tfrac{n}{2},k\}.
\end{equation}
\end{theorem}

\begin{proof}
The proof is similar to the proof of \fullref{theorem:typeD-howe}, but 
using the functor $\symhoweC$ and the 
non-quantized Howe duality from \cite[Corollary 5.32]{CW}. 
(Note hereby that we cannot easily pass from 
$\Udot(\frakso_{2k})$ to $\quantumg(\frakso_{2k})$ 
since the $\CQ$-vector space 
in \eqref{eq:Howe-action-sp-sym} 
is infinite-dimensional.)
\end{proof}

\subsubsection{Skew quantum Howe duality for the pair \texorpdfstring{$(\coideal(\frakso_n),\quantumg(\frakso_{2k}))$}{son-so2k}}

\begin{theorem}\label{theorem:typeD-howe}
There are commuting actions
\begin{equation}\label{eq:Howe-action-so}
\coideal(\frakso_n)
\acts  
\underbrace{\exterior{\bullet} \vecrep \otimes 
\dots \otimes \exterior{\bullet} \vecrep}_{k \text{ times}}
\actsreverse 
\quantumg(\frakso_{2k}).
\end{equation}
In case $n$ is odd they generate each other's
centralizer.
In any case,
the $\quantumg(\frakso_{2k})$-mo\-dules appearing 
in the decomposition of \eqref{eq:Howe-action-so} are
\begin{equation}\label{eq:bimodule-decomp-D}
  \simple{\frakso_{2k}}{
    {\textstyle\sum_{j=1}^k} (\lambda^{\mathrm{T}}_j - \tfrac{n}{2})\varepsilon_j}, \quad \text{for } \lambda\in\partition \text{ with }
\ell(\lambda^{\mathrm{T}})\leq k,\lambda^{\mathrm{T}}_1+\lambda^{\mathrm{T}}_2\leq n. 
\end{equation}
\end{theorem}

\begin{proof}
Mutatis mutandis as in the proof of \fullref{theorem:typeC-howe}, but 
using the functor $\howeD$ and the 
non-quantized Howe duality from \cite[Corollary 5.41]{CW}. 
Note that one has 
$\mathrm{O}(n) \cong \mathrm{SO}(n) \times \Z/2\Z$ 
in type $\typeB$. 
As explained in \cite[above Proposition 5.35]{CW} 
or \cite[\S 5.1.3]{LZ1}, 
the extra generator in $\mathrm{O}(n)- \mathrm{SO}(n)$ 
acts trivially on the 
de-quantized analog of \eqref{eq:Howe-action-so}. 
It follows that \cite[Corollary 5.41]{CW} works in this 
case for $\mathrm{SO}(n)$ instead of 
$\mathrm{O}(n)$, and hence also for $\frakso_n$,
cf.\ also \fullref{remark:sovso}.
\end{proof}

\subsubsection{Symmetric quantum Howe duality for the pair \texorpdfstring{$(\coideal(\frakso_n),\Udot(\fraksp_{2k}))$}{son-sp2k}}

\begin{theorem}\label{theorem:typeD-howe-sym}
There are commuting actions
\begin{equation}\label{eq:Howe-action-so-sym}
\coideal(\frakso_n)
\acts  
\underbrace{\symmetric{\bullet} \vecrep \otimes 
\dots \otimes \symmetric{\bullet} \vecrep}_{k \text{ times}}
\actsreverse 
\Udot(\fraksp_{2k}).
\end{equation}
In case $n$ is odd they generate each other's 
centralizer.
In any case,
the $\Udot(\fraksp_{2k})$-mo\-dules appearing 
in the decomposition of \eqref{eq:Howe-action-so-sym} are
\begin{equation}\label{eq:bimodule-decomp-D-sym}
\simple{\fraksp_{2k}}{
  {\textstyle\sum_{j=1}^k} (\lambda_j + \tfrac{n}{2})\varepsilon_j},\quad
\text{for } \lambda\in\partition \text{ with }
\ell(\lambda)\leq k, \lambda_1^{\mathrm{T}}+\lambda_2^{\mathrm{T}} \leq n.
\end{equation}
\end{theorem}

\begin{proof}
Mutatis mutandis as in the proof of \fullref{theorem:typeC-howe}, but 
using the functor $\symhoweD$ and the 
non-quantized Howe duality from \cite[Corollary 5.40]{CW}. \makeautorefname{theorem}{Theorems} 
(Keeping the same remarks as in the proofs of 
\fullref{theorem:typeC-howe-sym} and \ref{theorem:typeD-howe} in mind.)\makeautorefname{theorem}{Theorem}
\end{proof} 

\subsubsection{Some concluding remarks}

\begin{remark}\label{remark:the-weird-type-D}\makeautorefname{theorem}{Theorems}
We stress again that \fullref{theorem:typeD-howe} 
and \ref{theorem:typeD-howe-sym} can be strengthened
to include the double centralizer property for type $\typeD$ as well,
cf.\ \fullref{remark:sovso}.
\end{remark}

\makeautorefname{theorem}{Theorem}

\begin{remark}\label{remark:super-stuff}
In the spirit of \cite{TVW}, one could use the 
Howe dualities involving the 
orthosymplectic Lie superalgebra $\frakosp$, as 
in \cite{Ho1}, \cite{CZ} or \cite{CW},
to give a unified treatment of the exterior 
and the symmetric story. Since quantization in our setup is already 
quite involved, we decided to not pursue this further.
\end{remark}

\begin{remark}\label{remark:categorification}
One feature 
of web categories is that they are 
``amenable to categorification''. 
For example, one can use \emph{foams} in the sense of \cite{Kh},
see e.g.\ \cite{Bl}, \cite{LQR}, \cite{EST1} and \cite{EST2} for categorifiying webs.
Or category $\mathcal{O}$ as e.g.\ in \cite{Sa} or \cite{Sa2}. 
Categorifications of Howe dualities involving coideal subalgebras (of different kinds)
have already been obtained in \cite{ES} 
(which also connects to foams, cf.\ \cite{ETW}), and there are good reasons to hope 
that our story categorifies as well.
\end{remark}

\subsection{Relation of the web categories to the (quantum) Brauer algebra}\label{subsec-brauer}

\makeautorefname{theorem}{Theorems}

In groundbreaking work, Brauer \cite{Br} introduced 
the so-called \emph{Brauer algebra},
which arose naturally in his study
of the centralizer
of the action of the orthogonal group $\mathrm{O}(n)$
and of the symplectic group $\mathrm{Sp}(n)$
acting on the $k$-fold tensor product 
$\vecrepnoq^{\otimes k}$ of their 
vector representations.
Comparing this to the de-quantized 
versions of 
\fullref{theorem:typeC-howe}, \ref{theorem:typeC-howe-sym}, \ref{theorem:typeD-howe}
and \ref{theorem:typeD-howe-sym}
suggests that there should be a 
connection to our web categories. 
We make this more precise in the following.

\makeautorefname{theorem}{Theorem}

\subsubsection{Various quantizations of the Brauer algebra}

The first quantization of the Brauer algebra, 
called \emph{BMW algebra}, was introduced 
by Birman-Wenzl \cite{BW} and Murakami \cite{Mu}. 
The BMW algebra plays the role of Brauer's algebra 
with respect to the actions of $\quantumg(\frakso_n)$ 
and $\quantumg(\fraksp_n)$ on their quantum tensor spaces. 
However, since we are looking at the centralizers of actions of
$\coideal(\frakso_n)$ and $\coideal(\fraksp_n)$, and not of
$\quantumg(\frakso_n)$ and $\quantumg(\fraksp_n)$,
the BMW algebra does not fit into our picture.

In contrast, Molev \cite{Mo} defined a new 
quantization $\qbrauer{k}{\qpar}{\zpar}$ of the Brauer algebra, 
called \emph{quantum or $\qpar$-Brauer algebra}. 
This $\CQZ$-algebra 
is related by a version of $\qpar$-Schur-Weyl duality 
to $\coideal(\frakso_n)$ and $\coideal(\fraksp_n)$. 
Thus, $\qbrauer{k}{\qpar}{\zpar}$ is the natural candidate 
to be connected to our web categories.

\subsubsection{A quantized Brauer category}

First, let us quickly recall the situation in type $\typeA$:

\begin{definition}\label{definition:hecke}
The \emph{Hecke category} $\heckecat$ is the
additive closure of the (strict)
monoidal, $\CQ$-linear category generated
by one object 
$1$ and by one morphism $T \colon 1 \otimes 1 \rightarrow 1 \otimes 1$ 
modulo the relations
\begin{gather*}
    T^2 = (\qpar-\qpar^{-1})T + \id_{1 \otimes 1},\\
    (T \otimes \id_1) (\id_1 \otimes T) (T \otimes \id_1) = (\id_1 \otimes T) (T \otimes \id_1) (\id_1 \otimes T).
\end{gather*}
(The second relation is known as the \emph{braid relation}.)
\end{definition}

We depict the generator $T$ 
by an overcrossing, cf.\ \eqref{eq:braiding}.
Then, by sending $T$ in the evident way to the braiding of $\WebA$, 
we get a functor 
\[
\brauerA \colon \heckecat \rightarrow \WebA,
\quad
T\mapsto 
\begin{tikzpicture}[anchorbase,scale=.25,tinynodes]
	\draw[uncolored] (2,0) node[below] {$1$} .. controls ++(0,2) and  ++(0,-2) .. ++(-2,4) node[above] {$1$}; 
	\draw[uncolored,cross line] (0,0) node[below] {$1$} .. controls ++(0,2) and  ++(0,-2) .. ++(2,4) node[above] {$1$}; 
    \end{tikzpicture}
\]
which is fully faithful, 
see e.g.\ \cite[Proposition 5.9]{QS} or \cite[Lemma 2.25]{TVW}.
Note, in particular, that crossings span $\End_{\WebA}(1^{\otimes k})$.

Our next goal is to extend this to 
types $\typeB\typeC\typeD$.

\begin{definition}\label{definition:brauer}
The \emph{quantum or $\qpar$-Bauer category} 
$\qbrauercat{\qpar}{\zpar}$ is the 
additive closure of the $\CQZ$-linear
$\heckecat$-category generated by 
$\varnothing$ and by the cup and cap morphisms
(depicted as in \eqref{eq:cup-cap}) modulo 
the relations \eqref{eq:circle}, \eqref{eq:bubble-equi}, \eqref{eq:lasso-equi}, \eqref{eq:lolli-equi} 
and \eqref{eq:sliding-equi}.
\end{definition}

\makeautorefname{definition}{Definitions}

Recall that the 
relations \eqref{eq:bubble-equi}, \eqref{eq:lasso-equi}, \eqref{eq:lolli-equi} 
and \eqref{eq:sliding-equi} are 
the topological analogs of 
the relations in 
\fullref{definition:typeBDwebs} and \ref{definition:typeCwebs} 
(for $\typeCBD$-webs with slightly different parameters), 
and that \eqref{eq:circle} is equivalent 
to \eqref{eq:barbell-C} in case of $\typeCBD$-webs. 
Hence, the functor $\brauerA$ extends to two functors
\[
\brauerD\colon\qbrauercat{\qpar}{\zpar}\to\WebDz
\quad\text{and}\quad
\brauerC\colon\qbrauercat{-\qpar^{-1}}{\zpar}\to\WebCz.
\]

\makeautorefname{definition}{Definition}

\subsubsection{Connection with the quantum Brauer algebra}
\label{sec:conn-with-molevs}

Let us now denote by $\qbrauer{k}{\qpar}{\zpar}$ the $\qpar$-Brauer algebra
as defined by Molev in \cite[Definition 2.3]{Mo}. 
Precisely, the $\qpar$-Brauer algebra is a $\CQZ$-algebra with generators 
$T_i$ for $i=1,\dots,k-1$ and additionally $e_{k-1}$.
(Note that Molev uses the notation $\sigma_i$ instead of $T_i$.)

\begin{lemma}\label{proposition:brauer}
The assignment
\begin{equation*}
      T_i  \mapsto
      \begin{tikzpicture}[baseline=0.42cm,scale=.25,tinynodes]
        \draw[uncolored] (2,0) node[below] {$1$} .. controls ++(0,2) and  ++(0,-2) .. ++(-2,4) node[above] {$1$}; 
	  \draw[uncolored,cross line] (0,0) node[below] {$1$} .. controls ++(0,2) and  ++(0,-2) .. ++(2,4) node[above] {$1$}; 
	  \draw[uncolored] (-5,0) node[below] {$1$} -- ++(0,4) node[above] {$1$};
	  \draw[uncolored] (-2,0) node[below] {$1$} -- ++(0,4) node[above] {$1$};
	  \draw[uncolored] (4,0) node[below] {$1$} -- ++(0,4) node[above] {$1$};
	  \draw[uncolored] (7,0) node[below] {$1$} -- ++(0,4) node[above] {$1$};
      \path (-3.5,0) node[below] {$\cdots$} -- node {$\cdots$} ++(0,4) node[above] {$\cdots$};
      \path (5.5,0) node[below] {$\cdots$} -- node {$\cdots$} ++(0,4) node[above] {$\cdots$};
      \draw[decorate,decoration={brace,mirror},yshift=-1.4cm] (-5,0) -- node[below=0.3ex] {$i-1$} ++(3,0);
        \draw[decorate,decoration={brace,mirror},yshift=-1.4cm] (4,0) -- node[below=0.3ex] {$k{-}i{-}1$} ++(3,0);
      \end{tikzpicture}
\quad\text{and}\quad 
e_{k-1} \mapsto
      \begin{tikzpicture}[baseline=0.42cm,scale=.25,tinynodes]
      \draw[uncolored] (0,0) node[below] {$1$} -- ++(0,.4)  arc (180:0:1cm)  -- ++(0,-.4) node[below] {$1$};
      \draw[uncolored] (0,4) node[above] {$1$} -- ++(0,-.4)   arc (180:360:1cm)  -- ++(0,.4) node[above] {$1$};
	  \draw[uncolored] (-5,0) node[below] {$1$} -- ++(0,4) node[above] {$1$};
	  \draw[uncolored] (-2,0) node[below] {$1$} -- ++(0,4) node[above] {$1$};
      \path (-3.5,0) node[below] {$\cdots$} -- node {$\cdots$} ++(0,4) node[above] {$\cdots$};
      \draw[decorate,decoration={brace,mirror},yshift=-1.4cm] (-5,0) -- node[below=0.3ex] {$k{-}2$} ++(3,0);
      \end{tikzpicture} 
\end{equation*}
defines an algebra homomorphism 
$\psi_k \colon \qbrauer{k}{-\qpar^{-1}}{-\zpar^{-1}}
\rightarrow \End_{\qbrauercat{\qpar}{\zpar}}(1^{\otimes k})$.
\end{lemma}

\begin{proof}
This is immediate up to the last relation in \cite[Definition 2.3]{Mo}.
Verifying the last relation in \cite[Definition 2.3]{Mo} is a lengthy, 
but straightforward computation, which can 
be done by using \eqref{eq:lasso-equi} 
and \eqref{eq:sliding-equi} repeatedly.
\end{proof}

In particular, the composite 
$\diaD \circ \brauerD \circ \psi_k$ defines an action 
of the $\qpar$-Brauer algebra which commutes 
which the natural action of $\coideal(\frakso_n)$:
\begin{equation}\label{eq:so-brauer-rep}
\coideal(\frakso_n)
\acts
\underbrace{\vecrep \otimes 
\dots \otimes \vecrep}_{k \text{ times}}
\actsreverse
\qbrauer{k}{-\qpar}{-\qpar^{-n}}.
\end{equation}
Up to scaling conventions, 
this is the action defined in \cite[Theorem 4.2]{Mo}. Similarly, 
the composite $\diaC \circ \brauerC \circ \psi_k$ provides 
commuting actions
\begin{equation}\label{eq:sp-brauer-rep}
\coideal(\fraksp_n)
\acts
\underbrace{\vecrep \otimes 
\dots \otimes \vecrep}_{k \text{ times}}
\actsreverse 
\qbrauer{k}{-\qpar}{\qpar^{n}}.
\end{equation}
(Clearly, we could have also chosen $\symdiaD$ and 
$\symdiaC$ instead of $\diaD$ and 
$\diaC$.)

We show now that Molev's $\qpar$-Brauer algebra can be identified
with the endomorphism algebra of $1^{\otimes k}$ in our $\qpar$-Brauer category:

\begin{proposition}\label{proposition:brauer-into-webs}
The map $\psi_k$ is an algebra 
isomorphism, and the functors $\brauerD$ 
and $\brauerC$ are fully faithful.
\end{proposition}

\begin{proof}
\begin{description}[leftmargin=0pt,itemsep=1ex]

\item[\textbf{Surjectivity of $\psi_k$}]
Because crossings span the space $\End_{\WebA}(1^{\otimes k})$,
it is 
enough to show that 
$\End_{\qbrauercat{\qpar}{\zpar}}(1^{\otimes k})$ is spanned by 
diagrams of the form $w_{\mathrm{top}} e^{(l)} w_{\mathrm{bot}}$, with 
$w_{\mathrm{bot}},w_{\mathrm{top}} \in \End_{\heckecat}(1^{\otimes k})$ and 
diagrams  
\begin{equation*}
e^{(l)} =
      \begin{tikzpicture}[baseline=0.42cm,scale=.25,tinynodes]
      \draw[uncolored] (0,0) node[below] {$1$} -- ++(0,.4)  arc (180:0:1cm)  -- ++(0,-.4) node[below] {$1$};
      \draw[uncolored] (0,4) node[above] {$1$} -- ++(0,-.4)   arc (180:360:1cm)  -- ++(0,.4) node[above] {$1$};
      \draw[uncolored] (5,0) node[below] {$1$} -- ++(0,.4)  arc (180:0:1cm)  -- ++(0,-.4) node[below] {$1$};
      \draw[uncolored] (5,4) node[above] {$1$} -- ++(0,-.4)   arc (180:360:1cm)  -- ++(0,.4) node[above] {$1$};
      \path (3.5,0) node[below] {$\cdots$} -- node {$\cdots$} ++(0,4) node[above] {$\cdots$};
	  \draw[uncolored] (-5,0) node[below] {$1$} -- ++(0,4) node[above] {$1$};
	  \draw[uncolored] (-2,0) node[below] {$1$} -- ++(0,4) node[above] {$1$};
      \path (-3.5,0) node[below] {$\cdots$} -- node {$\cdots$} ++(0,4) node[above] {$\cdots$};
      \draw[decorate,decoration={brace,mirror},yshift=-1.4cm] (0,0) -- node[below=0.3ex] {$l$ caps} ++(7,0);
      \end{tikzpicture} 
\end{equation*}
This can be easily seen by induction on the number of 
crossings of some fixed diagram.

\item[\textbf{Injectivity of $\psi_k$}]

This follows because the 
representations in \eqref{eq:so-brauer-rep} and \eqref{eq:sp-brauer-rep} are 
faithful for $n\gg k$ (the precise bound is 
irrelevant for us). Indeed, the proof that they are faithful for $n\gg k$ 
follows, as in the proof 
of \cite[Theorem 3.8]{We}, by the same results in the non-quantized 
setting (see e.g.\ \cite[Theorem 3.17]{AST}, but the statement therein
can already be found implicitly in the 
work of Brauer \cite{Br}).

\item[\textbf{Fully faithfulness of $\brauerD$ and $\brauerC$}]
Very similar arguments as for the proof of bijectivity of $\psi_k$ imply that the 
functors $\brauerD$ and $\brauerC$ are fully faithful.\qedhere

\end{description}
\end{proof}

\begin{remark}\label{remark:brauer-into-webs}
Because of \fullref{proposition:brauer-into-webs}, our web categories can be seen as (vast) 
generalizations of the (quantum) Brauer calculus.
\end{remark}

\providecommand{\bysame}{\leavevmode\hbox to3em{\hrulefill}\thinspace}
\providecommand{\MR}{\relax\ifhmode\unskip\space\fi MR }
\providecommand{\MRhref}[2]{%
  \href{http://www.ams.org/mathscinet-getitem?mr=#1}{#2}
}
\providecommand{\href}[2]{#2}

\end{document}